\documentclass[envcountsect,envcountsame,envcountreset]{svjour3}       

\smartqed  
\usepackage[a4paper, margin=1in]{geometry}
\usepackage{color}
\usepackage{amsmath,amssymb}
\usepackage{epsfig}
\usepackage{graphicx}
\usepackage{latexsym}
\usepackage{pdfpages}

\usepackage{ulem} 

\usepackage{ifpdf}

\ifpdf 
  \usepackage[hidelinks]{hyperref}
\else 
\fi 

\usepackage{mathtools}
\mathtoolsset{showonlyrefs}

\newcommand{\res}{\!\!\mathop{\hbox{
                                \vrule height 7pt width .5pt depth 0pt
                                \vrule height .5pt width 6pt depth 0pt}}
                                \nolimits}

\spnewtheorem{assumption}{Assumption}{\bf}{\rm}

\numberwithin{equation}{section}

\newcommand{\R}{\mathbb R}
\newcommand{\N}{\mathbb N}

\newcommand{\1}{\raisebox{2pt}{\rm{$\chi$}}}

\newcommand{\z}{{\bf z}}
\renewcommand{\a}{{\bf a}}

\begin{document}

\title{Nonlocal doubly nonlinear diffusion problems with nonlinear  boundary conditions}


\titlerunning{Nonlocal nonlinear diffusion  with  nonlinear boundary conditions}        

\author{Marcos Solera \and Juli\'an Toledo}


\institute{M. Solera-Diana \at
Departamento de An\'{a}lisis Matem\'atico,
Universitat de Val\`encia, Valencia, Spain.
              \email{marcos.solera@uv.es},
              https://orcid.org/0000-0001-7774-4516        %
           \and
J. Toledo (corresponding author) \at
Departamento de An\'{a}lisis Matem\'atico,
Universitat de Val\`encia, Valencia, Spain.
\email{toledojj@uv.es}, https://orcid.org/0000-0001-6960-7351
}

\maketitle

\begin{abstract}
We study the existence and uniqueness of mild and strong solutions of  nonlocal  nonlinear diffusion problems of $p$-Laplacian type   with nonlinear boundary conditions posed in metric random walk spaces. These spaces include, among others, weighted discrete graphs and $\mathbb{R}^N$ with a random walk induced by a nonsingular kernel. We also study the case of nonlinear  dynamical boundary conditions. The generality of the nonlinearities considered allow us to cover the nonlocal counterparts of a large scope of local diffusion problems like, for example, Stefan problems, Hele-Shaw problems, diffusion in porous media problems and obstacle problems. Nonlinear semigroup theory is the basis for this study.
\keywords{Random walks, nonlocal operators, weighted graphs, $p-$Laplacian,  Neumann boundary conditions, diffusion  in porous media, Stefan problem, Hele-Shaw problem, obstacle problems,   dynamical boundary conditions}
 \subclass{35K55 \and 47H06 \and 47J35}
\end{abstract}

\section{Introduction and preliminaries}

In this article we study the existence and uniqueness of mild and strong solutions of nonlocal nonlinear diffusion problems of $p$-Laplacian type with nonlinear boundary conditions. The problems are posed in a subset~$W$ of a metric random walk space $[X,d,m]$ with a reversible measure $\nu$ for the random walk $m$ (see Subsection~\ref{semrw}  for details). The nonlocal diffusion can hold either in $W$, in its nonlocal boundary $\partial_mW$, or in both at the same time. We will assume that $W\cup\partial_mW$ is $m$-connected and $\nu$-finite. The formulations of the diffusion problems that we study are the following
\begin{equation}\label{sabore001particularintro01}
\left\{ \begin{array}{ll} v_t(t,x) - \hbox{div}_m\a_p u(t,x)=f(t,x),    &x\in  W,\ 0<t<T,
\\ \\ \displaystyle v(t,x)\in\gamma\big(u(t,x)\big), &
    x\in  W,\ 0<t<T,
\\ \\ -\mathcal{N}^{\a_p}_\mathbf{1} u(t,x) \in\beta\big(u(t,x)\big),    &x\in\partial_mW, \  0<t<T, \\ \\ v(0,x) = v_0(x),    &x\in W, \end{array} \right.
\end{equation}
  and, for nonlinear dynamical boundary conditions,
\begin{equation}\label{sabore001bevolparticularintro01}
\left\{ \begin{array}{ll}
\displaystyle v_t(t,x) - \hbox{div}_m\a_p u(t,x)=f(t,x),    &x\in  W,\ 0<t<T,
\\ \\ \displaystyle v(t,x)\in\gamma\big(u(t,x)\big), &
    x\in  W,\ 0<t<T,
\\ \\ w_t(t,x)+\mathcal{N}^{\a_p}_\mathbf{1} u(t,x)=g(t,x) ,   &x\in\partial_m W, \  0<t<T,
\\ \\ w(t,x) \in\beta\big(u(t,x)\big),    &x\in\partial_m W, \  0<t<T,
 \\ \\ v(0,x) = v_0(x),    &x\in W,
 \\ \\ w(0,x) = w_0(x),    &x\in \partial_m W, \end{array} \right.
\end{equation}
where
  $\gamma$ and $\beta$ are maximal monotone (multivalued) graphs in $\mathbb{R}\times \mathbb{R}$,
  $\hbox{div}_m\a_p$ is a nonlocal Leray-Lions type operator whose model is the nonlocal $p$-Laplacian type diffusion operator,
and $\mathcal{N}^{\a_p}_\mathbf{1}$ is a nonlocal Neumann boundary operator (see   Subsection~\ref{sub21} for details).   In fact, we solve these problems with greater generality, as we will not only consider them for a set $W$ and its nonlocal boundary $\partial_m W$, but rather for any two disjoint subsets $\Omega_1$ and $\Omega_2$ of $X$ such that their union is $m$-connected.

These problems can be seen as the nonlocal counterpart of local diffusion problems governed by the $p$-Laplacian diffusion operator (or a Leray-Lions operator) where two further nonlinearities are induced by $\gamma$ and $\beta$ (see for example~\cite{AIMTifb} and~\cite{BCrS} for  local problems). In~\cite{ElLibro}, and the references therein, one can find an interpretation of the nonlocal diffusion process involved in these kind of problems. On the nonlinearities (brought about by) $\gamma$ and $\beta$ we do not impose any further assumptions aside from the natural one (see B\'{e}nilan, Crandall and Sacks~\cite{BCrS}):
$$ 0\in\gamma(0)\cap\beta(0),
$$
and (in order for diffusion to take place)
$$ \nu(W)\varGamma^- + \nu(\partial_mW)\mathfrak{B}^-< \nu(W)\varGamma^+ + \nu(\partial_mW)\mathfrak{B}^+,
$$
where
$$\varGamma^-=\inf  \mbox{Ran}(\gamma) , \ \varGamma^+=\sup \mbox{Ran}(\gamma), \ \mathfrak{B}^-=\inf \mbox{Ran}(\beta) \ \hbox{ and } \
\mathfrak{B}^+=\sup \mbox{Ran}(\beta).$$
Therefore, we work with a rather general class of nonlocal nonlinear diffusion problems with nonlinear   boundary conditions.  We are able to directly cover: obstacle problems, with unilateral or bilateral obstacles (either in $W$, in $\partial_mW$, or in both at the same time); the nonlocal counterpart of Stefan like problems, that involve monotone graphs like   the graph inverse of
$$\theta_S(r)=\left\{\begin{array}{ll}
r & \hbox{if $r<0$},\\[4pt]
[0,\lambda]&\hbox{if $r=0$,}\\[4pt]
\lambda+r  & \hbox{if $r>0$},
\end{array}\right.$$
  for $\lambda>0$; diffusion problems in porous media, where monotone graphs like $p_s(r)=|r|^{s-1}r$, $s>0$, are involved; and Hele-Shaw type problems, which involve graphs like
  $$H(r)=\left\{\begin{array}{ll}
0 & \hbox{if $r<0$},\\[4pt]
[0,1]&\hbox{if $r=0$,}\\[4pt]
1  & \hbox{if $r>0$}.
\end{array}\right.$$
Moreover, if $\gamma=0$ in problem \eqref{sabore001particularintro01} then the dynamics only appear in the nonlocal boundary and we obtain the evolution problem for a nonlocal Dirichlet-to-Neumann operator as a particular case. In addition, the homogeneous Dirichlet boundary condition ($\beta=\{0\}\times \mathbb{R}$) and the Neumann boundary condition ($\beta=\mathbb{R}\times\{0\}$) are also covered.

  Nonlocal diffusion problems of $p$-Laplacian type involving nonlocal Neumann boundary operators have been recently studied in \cite{MST4} inspired by the nonlocal Neumann boundary operators for the linear case studied in~\cite{DR-OV} and~\cite{GL1}. Nevertheless, due to the generality of the hypotheses considered in this study, the results that we obtain lead to new existence and uniqueness results,   which do not follow from previous works, for a great range of problems. This is true even when the problems are considered on weighted discrete graphs or $\mathbb{R}^N$ with a random walk induced by a nonsingular kernel, spaces for which only some particular cases of these problems have been studied (some references are given afterwards). For these ambient spaces and for the precise choice of the nonlocal $p$-Laplacian operator,   Problem~\eqref{sabore001particularintro01} has the following formulations (see Subsection~\ref{semrw}, in particular Examples \ref{ejem01} and \ref{ejem02}, and Definition \ref{boundaryandclosure}, for the necessary definitions and notations):
$$
\left\{ \begin{array}{ll} v_t(t,x) = \displaystyle \frac{1}{d_x} \sum_{y \in V(G)} w_{x,y} \vert u(y)-u(x)\vert^{p-2}(u(y) - u(x)), \quad
   &x\in  W,\ 0<t<T,
    \\ \\ \displaystyle v(t,x)\in\gamma\big(u(t,x)\big), &
    x\in  W,\ 0<t<T,
    \\ \\  \displaystyle \frac{1}{d_x} \sum_{y \in W_{m^G}} w_{x,y} \vert u(y)-u(x)\vert^{p-2}(u(y) - u(x)) \in \beta(u(t,x)), \quad   &x\in \partial_{m^G}W, \  0<t<T, \\ \\ u(x,0) = u_0(x),    &x\in W, \end{array} \right.
$$
for weighted discrete graphs, and
$$
\left\{ \begin{array}{ll} v_t(t,x) = \displaystyle\int_{\R^N} J(y-x) \vert u(y)-u(x)\vert^{p-2}(u(y) - u(x)) dy, \quad
   &x\in  W,\ 0<t<T,
\\ \\ \displaystyle v(t,x)\in\gamma\big(u(t,x)\big), &
    x\in  W,\ 0<t<T,
   \\ \\   \displaystyle\int_{W_{m^J}} J(y-x) \vert u(y)-u(x)\vert^{p-2}(u(y) - u(x)) dy \in \beta(u(t,x)), \quad   &x\in\partial_{m^J}W, \  0<t<T, \\ \\ v(x,0) = v_0(x),    &x\in W. \end{array} \right.
$$
 for the case of $\mathbb{R}^N$ with the random walk induced by the nonsingular kernel $J$. We have detailed these problems with well-known formulations in order to show the extent to which Problems~\eqref{sabore001particularintro01} and \eqref{sabore001bevolparticularintro01} cover specific nonlocal problems of great interest.

Nonlinear semigroup theory will be the basis for the study of the existence and uniqueness of solutions of the above problems. This study is developed in Section~\ref{lasec3}, where we prove, as a particular case of Theorem~\ref{nsth01bevol}, the existence of mild solutions of Problem~\eqref{sabore001bevolparticularintro01} for general data in $L^1$, and of strong solutions assuming extra integrability conditions on the data. Moreover, a contraction and comparison principle is obtained. The same is done for Problem~\eqref{sabore001particularintro01} in Theorem~\ref{elsegundo01}. See \cite{BARBU}, \cite{BARBU2}, \cite{Benilantesis},  \cite{Brezis}, \cite{Cr}, \cite{Cr2} and \cite{CrandallLiggett}, for details on such theory, which is completely covered in the well known unpublished manuscript  {\it Evolution equations governed by accretive operators} written by  Ph. B\'{e}nilan, M. G. Crandall and A. Pazy. A summary of it can be found in \cite[Appendix]{ElLibro}.

 To apply the nonlinear semigroup theory our first aim is to prove the existence and uniqueness of solutions of the problem
 \begin{equation}\label{02091131intro001}
\left\{ \begin{array}{ll} \gamma\big(u(x))-\hbox{div}_m\a_p u(x) \ni \varphi(x), \quad &  x\in W, \\ [8pt]  \mathcal{N}^{\a_p}_1  u(x)+\beta\big(u(x)\big)\ni \varphi(x),  \quad & x\in\partial_m W, \end{array} \right.
\end{equation}
for general maximal monotone graphs $\gamma$ and $\beta$. This is
 the nonlocal counterpart of (local) quasilinear elliptic problems with  nonlinear boundary conditions (see~\cite{AIMTq} and~\cite{BCrS} for the general study of the local case) and is an interesting problem in itself due to the generality with which we address it. To this aim, we make use of a kind of nonlocal Poincar\'{e} type inequalities (see Appendix~\ref{secineqpoin}) which help us obtain boundedness arguments. These boundedness arguments together with some monotonicity arguments allow us to prove our results by adapting some of the ideas used in~\cite{AIMTq} and~\cite{BCrS}  (see also~\cite{BLGJEvolE} for a very particular case). The same holds for the diffusion problems.
The study of Problem~\eqref{02091131intro001} is developed in Section~\ref{lasec2}, where we prove, for a more general problem, the existence of solutions (Theorem~\ref{existenceeli01}) and a contraction and comparison principle (Theorem~\ref{maxandcont01}). At the end of that section we deal with another nonlocal Neumann boundary operator.

  For linear or quasilinear elliptic problems with boundary conditions, obstacles complicate the existence of solutions. The appearance of this difficulty is better understood when one takes into account the continuity of the solution between the inside of the domain and the boundary via the trace. In fact, for a bounded smooth domain $\Omega$ in $\mathbb{R}^N$,    $\gamma$ with bounded domain $[0,1]$ and $\beta(r)=0$ for all $r$, it is not possible to find a weak solution of
 $$\left\{
 \begin{array}{ll}
 -\Delta u+\gamma(u)\ni\varphi&\hbox{in }\Omega,\\[8pt]
 \nabla u\cdot\eta=\widetilde\varphi&\hbox{in }\partial\Omega,
 \end{array}
 \right.
 $$
 for data satisfying $\varphi\le 0$, $\widetilde\varphi\le 0$ and $ \widetilde\varphi\not\equiv 0$ (see~\cite{AIMTq}). However, in our nonlocal setting this sort of continuity is not present and the  study of these nonlocal diffusion problems with obstacles hence differs from the study of the local ones (see~\cite{AIMTobs} for a detailed study of these local problems). In particular, we do not need to impose any assumptions on the nonlinearities $\gamma$ and $\beta$ aside from the natural ones.

  There is a very long list of references for the local elliptic and parabolic counterparts of the problems that we study; see, for example,~\cite{AIMTifb}, \cite{AIMTq}, \cite{Benilantesis}, \cite{BBrCr}, \cite{BCr1},  \cite{BCrS}, \cite{Chilletal}, \cite{Sauter}, \cite{Vazquezbook}, and the references therein. See also~\cite{Nour01} for a  Hele-Shaw problem with dynamical boundary conditions and the references therein. For some   particular nonlocal problems   we refer to~\cite{BLGJEvolE}, \cite{ElLibro}, \cite{julioetal},  \cite{brandeletal}, \cite{Chasseigneetal},  \cite{fadili01}, \cite{KSZ} and~\cite{MST4}.  For fractional diffusion problems we refer, for example, to~\cite{MRTFrac}, where Dirichlet and   Neumann boundary conditions are considered;   to~\cite{Bonforteetal}, \cite{BonforteVazquez}, \cite{cianietal}, \cite{dePabloetal} and~\cite{giacomoni},  where fractional porous medium equations are studied, see also J.~L.~V\'{a}zquez's survey~\cite{JLsurvey} and the references therein;   and to~\cite{delTeso2} and~\cite{delTeso01} for fractional diffusion problems for the  Stefan problem.

We now introduce the framework space considered and some other concepts that will be used later on.

 \subsection{Metric random walk  spaces}\label{semrw}
  Let $(X,d)$ be a  Polish metric  space equipped with its Borel $\sigma$-algebra. In the following, whenever we consider
a measure on $X$ we assume that it is defined on this $\sigma$-algebra.

 As introduced  in~\cite{O}, a {\it random walk} $m$ on $X$ is a family of Borel probability measures $m_x$ on $X$, $x \in X$, satisfying the two technical conditions: (i) the measures $m_x$  depend measurably on the point  $x \in X$, i.e., for any Borel set $A$ of $X$ and any Borel set $B$ of $\R$, the set $\{ x \in X \ : \ m_x(A) \in B \}$ is Borel; (ii) each measure $m_x$ has finite first moment, i.e. for some (hence any) $z \in X$, and for any $x \in X$ one has $\int_X d(z,y) dm_x(y) < +\infty$.

 A {\it metric random walk  space} $[X,d,m]$  is a Polish  metric space $(X,d)$ together with a  random walk $m$.

 A $\sigma$-finite measure $\nu$ on $X$ is {\it invariant} with respect to the random walk $m=(m_x)$ if
 $$\nu (A):=\int_X m_x(A)d\nu(x)\ \  \hbox{ for every Borel set $A$}.$$
Moreover, the measure $\nu$ is said to be {\it reversible} with respect to $m$ if the following balance condition holds:
$$dm_x(y)d\nu(x)=dm_y(x)d\nu(y),$$
that is, for any Borel set $C \subset X \times X$,  $$  \int_{X}\left(\int_X \1_{C}(x,y)  dm_x(y)\right)d\nu(x)  =  \int_X\left(\int_X\1_C(x,y) dm_y(x)\right)d\nu(y).$$
Under suitable assumptions on the metric random walk space $[X,d,m]$, such a reversible measure $\nu$ exists and is unique. Note that the reversibility condition implies the invariance condition.

\begin{assumption}\label{assumption1}
From this point onwards, $[X,d,m]$ is a metric random walk space equipped with a $\sigma$-finite measure $\nu$ which is reversible (thus invariant) with respect to $m$.
\end{assumption}

Let $\mathcal{B}$ be the Borel $\sigma$-algebra of $(X,d)$. Since $\nu$ is a $\sigma$-finite measure on $(X,\mathcal{B})$ and $m$ is a stochastic kernel on $(X,\mathcal{B})$, we may define the tensor product $\nu\otimes m_x$ of $\nu$ and $m$ (see, for example, \cite[Section 1.2.2]{Douc}, see also~\cite[Section~2.5]{Ambrosio}), which is a measure on $(X\times X, \mathcal{B}\otimes\mathcal{B})$, by
$$\nu\otimes m_x (A\times B):=\int_A m_x(B) d\nu(x) \quad \hbox{for every } A,\, B\in \mathcal{B}. $$
Then, a $\sigma$-finite measure $\nu$ invariant with respect to $m$ is reversible if, and only if, the measure $\nu\otimes m_x$ is symmetric.
 Note that, for every  $g\in L^1(X\times X,\nu\otimes m_x)$,
  $$\int_{X \times X} g  d(\nu \otimes m_x)   = \int_X   \int_X g(x,y) dm_x(y)  d\nu(x).$$

\begin{example}\label{ejem01} An important class of examples of metric random walk spaces is composed by those which are obtained from weighted discrete graphs. Let $G = (V(G), E(G),(w_{xy})_{x,y\in V(G)})$ be a weighted discrete graph, where $V(G)$ is the set of vertices, $E(G)$ is the set of edges and $w_{xy} = w_{yx}$ is the nonnegative weight assigned to the edge $(x,y) \in E(G)$ (we suppose that $w_{xy} = 0$ if $(x,y) \not\in E(G)$ for $x,y\in V(G)$). In this case, the following probability measures define a random walk on $(V(G),d_G)$ (here, $d_G$ is the standard graph distance):
$$m_x^G:=\frac{1}{d_x}\sum_{y\in V(G)}w_{xy},$$
where $d_x:= \sum_{y\sim x} w_{xy} = \sum_{y\in V(G)} w_{xy}$. Note that, if $w_{x,y}=1$ for every $(x,y)\in E(G)$, then $d_x$ coincides with the degree of the vertex $x$ in the graph, that is,  the number of edges containing the vertex $x$. Moreover, the measure $\nu_G$ defined by
 $$\nu_G(A):= \sum_{x \in A} d_x,  \ \ \ A \subset V(G),$$
is a reversible measure with respect to this random walk.
\end{example}

\begin{example}\label{ejem02}  Another important class of examples is given by those of the form $[\R^N, d,m^J]$ where $d$ is the Euclidean distance and $m^J$ is defined as follows: let  $J:\R^N\to[0,+\infty[$ be a measurable, nonnegative and radially symmetric
function  satisfying $\int_{\R^N}J(z)d\mathcal{L}^N(z)=1$ ($\mathcal{L}^N$ is the Lebesgue measure) and set
$$m^J_x(A) :=  \int_A J(x - y) d\mathcal{L}^N(y) \quad \hbox{ for every Borel set } A \subset  \R^N \hbox{ and }x\in\R^N.$$
In this case $\mathcal{L}^N$ is a reversible measure with respect to this random walk.

See~\cite{MST0} (in particular ~\cite[Example 1.2]{MST0}) for a more detailed exposition of these and other examples.
\end{example}

\begin{definition}
Given two measurable subsets $A$, $B \subset X$, we define the {\it $m$-interaction} between $A$ and $B$ as
$$
L_m(A,B):= \int_A \int_B dm_x(y) d\nu(x).
$$
 \end{definition}
Note that, whenever  $L_m(A,B) < +\infty$, if $\nu$ is reversible with respect to $m$,
 $$L_m(A,B)=L_m(B,A).$$

\begin{definition}\label{boundaryandclosure}{\rm
 Given a measurable set $\Omega \subset X$, we define its {\it $m$-boundary} as
  $$\partial_m\Omega:=\{ x\in X\setminus \Omega : m_x(\Omega)>0 \}$$
  and its {\it $m$-closure} as $$\Omega_m:=\Omega\cup\partial_m\Omega.$$}
  \end{definition}

Moreover, we define the following ergodicity property.

\begin{definition}\label{defomegaconnected}
{\rm Let $[X,d,m]$ be a metric random walk space with a reversible measure $\nu$  with respect to $m$, and let $\Omega\subset X$ be a measurable and non-$\nu$-null subset. We say that $\Omega$ is {\it $m$-connected} if $L_m(A,B)>0$ for every pair of measurable non-$\nu$-null sets $A$, $B\subset \Omega$ such that $A\cup B=\Omega$   (see~\cite{MST0}).}
\end{definition}

We recall the following nonlocal notions of gradient and divergence.
\begin{definition}\label{nonlocalgraddiv}
{\rm Given a function $u : X \rightarrow \R$ we define its {\it nonlocal gradient} $\nabla u: X \times X \rightarrow \R$ as
$$\nabla u (x,y):= u(y) - u(x), \quad \, x,y \in X.$$
For a function $\z : X \times X \rightarrow \R$, its {\it $m$-divergence} ${\rm div}_m \z : X \rightarrow \R$ is defined as
 $$({\rm div}_m \z)(x):= \frac12 \int_{X} (\z(x,y) - \z(y,x)) dm_x(y), \quad x\in X.$$}
\end{definition}

\subsection{Yosida approximation and a B\'{e}nilan-Crandall relation}\label{secacc}\

  Given a maximal monotone graph $\vartheta$ in $\R\times\R$ (see~\cite{Brezis}) and   $\lambda>0$, let us denote by
  $$\vartheta_\lambda:=\lambda\left(I-\left(I+\frac1\lambda \vartheta\right)^{-1}\right)$$
  the {\it Yosida approximation of $\vartheta$} of parameter $1/\lambda$.

The function $\vartheta_\lambda$ is maximal monotone and Lipschitz continuous with Lipschitz constant $\lambda$ (see~\cite[Proposition 2.6]{Brezis}. Moreover,
 $\lim_{\lambda\to +\infty} \vartheta_\lambda (s) = \vartheta^0 (s)$ where
$$\vartheta^0(s) :=\left\{ \begin{array}{ll}
                         \hbox{the element of minimal absolute value of $\vartheta(s)$} & \hbox{if } s\in D(\vartheta), \\
                         +\infty & \hbox{if } [s,+\infty)\cap D(\vartheta)=\emptyset, \\
                         -\infty & \hbox{if } (-\infty,s]\cap D(\vartheta)=\emptyset,
                       \end{array}\right.$$
is an extension to $\R$ of the minimal section of $\vartheta$. Furthermore, if $s\in D(\vartheta)$, $|\vartheta_\lambda(s)|\le |\vartheta^0(s)|$ for every $\lambda>0$, and $|\vartheta_\lambda(s)|$ is nondecreasing in $\lambda$.

Given a maximal monotone graph $\vartheta$ in $\R\times\R$ with $0\in\vartheta(0)$, we define, for $s\in D(\vartheta)$,
$$\vartheta_+(s):=\left\{\begin{array}{ll}
\vartheta(s) & \hbox{ if $s>0$,}\\
\vartheta(0)\cap [0,+\infty) & \hbox{ if $s=0$,}\\
\{0\} & \hbox{ if $s<0$,}
\end{array}\right.$$
and
$$\vartheta_-(s):=\left\{\begin{array}{ll}
\{0\} & \hbox{ if $s>0$,}\\
\vartheta(0)\cap (-\infty,0] & \hbox{ if $s=0$,}\\
\vartheta(s) & \hbox{ if $s<0$.}
\end{array}\right.$$
Note that the Yosida approximation $(\vartheta_+)_\lambda$ of $\vartheta_+$ is nondecreasing in $\lambda>0$ and $(\vartheta_-)_\lambda$ is nonincreasing in $\lambda>0$. Observe also that $(\vartheta_+)_\lambda(s)=0$ for $s\le 0$ and $(\vartheta_-)_\lambda(s)=0$ for $s\ge 0$, for every $\lambda>0$, and $\vartheta_++\vartheta_-=\vartheta$.

Given a maximal monotone graph $\vartheta$ with $0\in D(\vartheta)$, $j_\vartheta(r):=\int_0^r\vartheta^0(s)ds$, $r\in\R$, defines a convex and lower semicontinuous function such   that $\vartheta$ is equal to the subdifferential of $j_\vartheta$: $$\vartheta=\partial j_\vartheta.$$
Moreover, if $j_\vartheta{}^*$ is the Legendre transform of $j_\vartheta$, then
$$\vartheta{}^{-1}=\partial j_\vartheta{}^*.$$

We now recall a B\'enilan-Crandall relation between functions $u, v\in L^1(\Omega,\nu)$. Denote by $J_0$ and $P_0$ the following sets of functions:
$$J_0 := \{ j : \R \rightarrow [0, +\infty] \ : \ \mbox{$j$ is convex,
lower semicontinuous and} \ j(0) = 0 \},$$
  $$ P_0:= \left\{\rho\in  C^\infty(\R) \ : \ 0\le \rho'\le 1, \hbox{ supp}(\rho')  \hbox{ is compact and }
  0\notin \hbox{supp}(\rho) \right\}.
  $$
Assume that $\nu(\Omega) < +\infty$ and let $u,v\in L^1(\Omega,\nu)$. The following relation between $u$ and $v$ is defined in \cite{BCr2}:
\begin{equation}\label{Def.menormenor}
  u\ll v \ \hbox{ if} \ \int_{\Omega} j(u)\,  d\nu  \leq \int_{\Omega} j(v)
 \, d\nu \ \ \hbox{for every} \ j \in J_0.
\end{equation}
Moreover, the following equivalences are proved in~\cite[Proposition 2.2]{BCr2} (we only give the particular cases that we use):
\begin{equation}\label{bCrProp22laotra}\int_\Omega v\rho(u)d\nu\ge 0\quad \hbox{for every } \rho\in P_0\  \Longleftrightarrow \  u\ll u+\lambda v\quad \hbox{for every } \lambda>0,
\end{equation}
\begin{equation}\label{bCrProp22}\int_\Omega v\rho(u)d\nu\ge 0\quad \hbox{for every } \rho\in P_0\  \Longleftrightarrow \   \int_{\{u<-h\}}vd\nu\le 0\le\int_{\{u>h\}}vd\nu\quad \hbox{for every } h>0.
\end{equation}

\section{Nonlocal stationary problems}\label{lasec2}

  In this section we give our main results concerning the existence and uniqueness of solutions of the nonlocal stationary Problem \eqref{02091131intro001}. We start by recalling the class of nonlocal Leray-Lions type operators and the Neumann boundary operators
  that we will be working with, and which were introduced in~\cite{MST4}.

\subsection{Nonlocal diffusion operators of Leray-Lions type and nonlocal Neumann boundary operators}\label{sub21}
For $1<p<+\infty$, let us consider a function $\a_p:X\times X\times \mathbb{R}\to \mathbb{R}$ such that
$$ (x,y)\mapsto \a_p(x,y,r) \quad \hbox{is measurable for every $r\in\R$;}
$$
\begin{equation}\label{ll002} \hbox{$\a_p(x,y,.)$  is continuous for $\nu\otimes m_x$-a.e $(x,y)\in X\times X$;}
\end{equation}
\begin{equation}\label{llo4}
 \a_p(x,y,r)=-\a_p(y,x,-r) \quad \hbox{for $\nu\otimes m_x$-a.e $(x,y)\in X\times X$ and for every  $r\in\R$;}
\end{equation}
\begin{equation}\label{llo3}
(\a_p(x,y,r)-\a_p(x,y,s))(r-s) > 0 \quad \hbox{for $\nu\otimes m_x$-a.e. $(x,y)\in X\times X$ and for every  $r\neq s$;}
\end{equation}
there exist constants $c_p,C_p>0$ such that
\begin{equation}\label{llo1}
|\a_p(x,y,r)|\le C_p\left(1+|r|^{p-1}\right) \quad \hbox{for $\nu\otimes m_x$-a.e. $(x,y)\in X\times X$ and for every  $r\in\R$,}
\end{equation}
and
\begin{equation}\label{llo2}
\a_p(x,y,r)r\ge c_p\vert r \vert^p \quad \hbox{for $\nu\otimes m_x$-a.e. $(x,y)\in X\times X$ and for every  $r\in\R$.}
\end{equation}

Condition \eqref{llo4} and the last condition imply that
$$
\a_p(x,y,0)=0 \ \hbox{ and } \   \hbox{sign}_0(\a_p(x,y,r))=\hbox{sign}_0(r) \quad \hbox{for $\nu\otimes m_x$-a.e. $(x,y)\in X\times X$ and for every  $r\in\R$}.
$$

 For $u:X\to \mathbb{R}$, let us define $\z_{\a_p,u}:X\times X\rightarrow \R$ by $\z_{\a_p,u}(x,y):=\a_p\left(x,y,\nabla u(x,y)\right)$. Then (recall Definition \ref{nonlocalgraddiv}),
on account of~\eqref{llo4},
$$
\begin{array}{l}
\displaystyle\hbox{div}_m \z_{\a_p,u} (x) =\frac12 \int_{X}\big(\a_p(x,y,u(y)-u(x)) - \a_p(y,x,u(x)-u(y))\big) dm_x(y)
\\[12pt]
\displaystyle
\phantom{\hbox{div}_m\a_p u (x)}
=\int_X \a_p(x,y,u(y)-u(x)) dm_x(y).
\end{array}
$$
For simplicity, we write
$$\hbox{div}_m \a_p u(x)=\hbox{div}_m \z_{\a_p,u} (x).$$

An example of  a function $\a_p$ satisfying the above assumptions  is
$$\a_p(x,y,r):=\frac{\varphi(x)+\varphi(y)}{2}|r|^{p-2}r,$$
where  $\varphi:X\rightarrow \R$ is a measurable function satisfying $0<{c}\le \varphi\le {C}$, where ${c}$ and ${C}$ are constants.
In particular, if $\varphi(x)=2$ for every $ x\in X$,
$$
\hbox{div}_m \a_p u(x) =  \int_{X}  |u(y)-u(x)|^{p-2}(u(y)-u(x))  dm_x(y)=\int_{X}  |\nabla u(x,y)|^{p-2}\nabla u(x,y)  dm_x(y)
$$
is the (nonlocal) $p$-Laplacian operator on the metric random walk space $[X,d,m]$.

Observe that $\hbox{div}_m \a_p u(x)$ defines
  {\it a  kind of Leray--Lions operator for the random walk~$m$}.

   We now recall  the   {\it nonlocal Neumann boundary operators} introduced in~\cite{MST4}. Let us consider a measurable set $W\subset X$ with $\nu(W)>0$. The Gunzburger--Lehoucq type Neumann boundary operator on $\partial_mW$ is given by
 $$
\mathcal{N}^{\a_p}_1 u(x):=  -\int_{W_m} \a_p(x,y,u(y)-u(x)) dm_x(y),    \quad x \in \partial_mW,
$$
 where, taking into account the supports of the $m_x$, we have that, in fact, the integral is being calculated over the nonlocal tubular boundary $\partial_mW\cup\partial_m(X\setminus W)$ of $W$. On the other hand, the Dipierro--Ros-Oton--Valdinoci type  Neumann boundary  operator on $\partial_mW$ is given by
$$
\mathcal{N}^{\a_p}_2 u(x):=  -\int_{W} \a_p(x,y,u(y)-u(x)) dm_x(y)    \quad x \in \partial_mW,
$$
  for which, in this case, the integral is being calculated over the nonlocal boundary $\partial_m(X\setminus W)$ of $X\setminus W$.

For each of these Neumann boundary operators and for $\varphi$ defined on $W_m=W \cup \partial_m W$, we can look for solutions of the following problem
 $$
\left\{ \begin{array}{ll} \gamma\big(u(x))-\hbox{div}_m\a_p u(x) \ni \varphi(x), \quad &  x\in W, \\ [8pt]   \mathcal{N}^{\a_p}_\mathbf{j} u(x)+\beta\big(u(x)\big)\ni \varphi(x),  \quad & x\in\partial_m W, \end{array} \right.
$$
$\mathbf{j}\in\{1,2\}$.
 Observe that, by the reversibility of $\nu$ with respect to $m$ and recalling the definitions of $\partial_m W$ and $W_m$ (Definition \ref{boundaryandclosure}), $m_x(X\setminus W_m)=0$ for $\nu$-a.e. $x\in W$. Indeed,
$$\displaystyle \int_{W}m_x(X\setminus W_m)d\nu(x)=\int_{X\setminus W_m}m_x(W)d\nu(x)=0  . $$
Consequently,
\begin{equation}\label{dom1251}
 \hbox{div}_m\a_p u (x) =\int_{W_m} \a_p(x,y,u(y)-u(x)) dm_x(y) \quad  \hbox{for every $x\in W$.}
\end{equation}

\begin{lemma}\label{convergenciaap}
Let $\Omega\subset X$ be a $\nu$-finite set and let $\{u_k\}_{k\in\N}\subset L^p(\Omega,\nu)$ such that $u_k\stackrel{k}{\longrightarrow} u\in L^p(\Omega,\nu)$ in $L^p(\Omega,\nu)$ and pointwise $\nu$-a.e. in $\Omega$. Suppose also that there exists $h\in L^p(\Omega,\nu)$ such that $|u_k|\le h$ $\nu$-a.e. in $\Omega$. Then
 $$\z_{\a_p,u_k}\stackrel{k}{\longrightarrow} \z_{\a_p,u} \ \hbox{ in } L^{p'}(\Omega\times\Omega,\nu\otimes m_x)$$
 and, in particular,
  $$\int_\Omega \a_p(\cdot,y, \nabla u_{k}(\cdot,y))dm_{(\cdot)}(y)\stackrel{k}{\longrightarrow} \int_\Omega \a_p(\cdot,y, \nabla u(\cdot,y))dm_{(\cdot)}(y) \ \hbox{ in } L^{p'}(\Omega,\nu).$$
\end{lemma}

Taking a subsequence if necessary, the $\nu$-a.e. pointwise convergence and the domination by the function $h$ in the hypotheses are a consequence of the convergence in $L^p(\Omega,\nu)$.

\begin{proof}
Let $A\subset\Omega$ be a $\nu$-null set such that $|u_k(x)|\le h(x)<+\infty$ for every $x\in\Omega\setminus A$ and every $k\in\N$, and such that $u_k(x)\stackrel{k}{\longrightarrow} u(x)$ for every $x\in\Omega\setminus A$.
By \eqref{ll002}, there exists a $\nu\otimes m_x$-null set $N_1 \subset \Omega\times\Omega$ such that $\a_p(x,y,\cdot)$ is continuous for every $(x,y)\in (\Omega\times\Omega)\setminus N_1$. Therefore,
 $$\a_p(x,y, u_{k}(y)-u_{k}(x))\stackrel{k}{\longrightarrow} \a_p(x,y, u(y)-u(x))$$
 for every $(x,y)\in(\Omega\times\Omega)\setminus (N_1\cup (A\times \Omega)\cup (\Omega\times A))$, where, by the reversibility of $\nu$ with respect to $m$, $N_1\cup (A\times \Omega)\cup (\Omega\times A)$ is also $\nu\otimes m_x$-null. Moreover, by \eqref{llo1}, there exists a $\nu\otimes m_x$-null set $N_2 \subset \Omega\times\Omega$ such that
 $$\begin{array}{rl}
 \displaystyle|\a_p(x,y,u_{k}(x)-u_{k}(y))|& \displaystyle\le C_p(1+|u_{k}(x)-u_{k}(y)|^{p-1})\le \widetilde{C}(1+|u_{k}(x)|^{p-1}+|u_{k}(y)|^{p-1}) \\ [8pt]
 & \displaystyle\le \widetilde{C}(1+|h(x)|^{p-1}+|h(y)|^{p-1})
 \end{array}$$
 for every $(x,y)\in (\Omega\times\Omega)\setminus (N_2\cup(A\times \Omega)\cup (\Omega\times A))$ and some constant $\widetilde{C}$, where, again, $N_2\cup (A\times \Omega)\cup (\Omega\times A)$ is $\nu\otimes m_x$-null. Then, taking $(x,y)\in(\Omega\times\Omega)\setminus(N_1\cup N_2\cup(A\times \Omega)\cup (\Omega\times A))$,
  $$\a_p(x,y, u_{k}(y)-u_{k}(x))\stackrel{k}{\longrightarrow} \a_p(x,y, u(y)-u(x))$$
  and
  $$|\a_p(x,y,u_{k}(x)-u_{k}(y))|\le \widetilde{C}(1+|h(x)|^{p-1}+|h(y)|^{p-1}).$$
   Now, by the invariance of $\nu$ with respect to $m$, since $h\in L^{p}(\Omega,m_x)$ and $\nu(\Omega)<+\infty$, we have that, for  $\tilde h(x,y):= 1+|h(x)|^{p-1}+|h(y)|^{p-1}$, $\tilde h\in L^{p'}(\Omega\times\Omega,\nu\otimes m_x)$, so we may apply the dominated convergence theorem to conclude.
   \qed\end{proof}

\subsection{Existence and uniqueness of solutions of doubly nonlinear stationary problems under nonlinear boundary conditions}\label{efzly}

As mentioned in the introduction the aim here is to study the existence and uniqueness of solutions of the problem
 \begin{equation}\label{02091131}
\left\{ \begin{array}{ll} \gamma\big(u(x))-\hbox{div}_m\a_p u(x) \ni \varphi(x), \quad &  x\in W, \\ [8pt]  \mathcal{N}^{\a_p}_1  u(x)+\beta\big(u(x)\big)\ni \varphi(x),  \quad & x\in\partial_m W, \end{array} \right.
\end{equation}
  where $W\subset X$ is $m$-connected and $\nu(W_m)<+\infty$.  See~\cite{AIMTq} and~\cite{BCrS} for the reference local models. In Subsection~\ref{secdipierro001} we address this problem but with the nonlocal Neumann boundary operator $\mathcal{N}^{\a_p}_2$ instead.

Problem \eqref{02091131} is a particular case (recall~\eqref{dom1251}) of the following general, and interesting by itself, problem. Let $\Omega_1,\Omega_2\subset X$ be disjoint measurable non-$\nu$-null sets and let $$\Omega:=\Omega_1\cup \Omega_2.$$ Given $\varphi\in L^{1}(\Omega,\nu)$ we consider the problem
\begin{equation}\label{02091131general}
(GP_\varphi^{ \a_p,\gamma,\beta})\quad\left\{ \begin{array}{ll}\displaystyle \gamma\big(u(x))- \int_{\Omega} \a_p(x,y,u(y)-u(x)) dm_x(y) \ni \varphi(x), \quad &  x\in\Omega_1, \\ [12pt]
\displaystyle \beta\big(u(x)\big)- \int_{\Omega} \a_p(x,y,u(y)-u(x)) dm_x(y)\ni\varphi(x),  \quad & x\in\Omega_2. \end{array} \right.
\end{equation}
  For simplicity, we generally use the notation $(GP_\varphi)$ in place of $(GP_\varphi^{ \a_p,\gamma,\beta})$. However, we use the more detailed notation further on.    Moreover, we make the following assumptions.

\begin{assumption}\label{assumptionOmega}
We assume that $\Omega=\Omega_1\cup\Omega_2$ is $m$-connected and $\nu(\Omega)<+\infty$.
\end{assumption}

\begin{remark}\label{omegamconnectedthenmxomegapos}
  Observe that, given an $m$-connected set $\Omega\subset X$ (recall Definition \ref{defomegaconnected}), $m_x(\Omega)>0$ for $\nu$-a.e. $x\in \Omega$. Indeed, if
$$N:=\{x\in\Omega \, : \, m_x(\Omega)=0\},$$
then
$$L_m(N,\Omega)=0,$$
thus $\nu(N)=0$.
\end{remark}

\begin{assumption}\label{assumption4}
  Let
$$\mathcal{N}_\perp^\Omega:=\left\{x\in \Omega \, :   (m_x\res\Omega) \perp  (\nu\res\Omega) \right\},$$
where the notation $(m_x\res\Omega) \perp  (\nu\res\Omega)$ means that $m_x\res\Omega$ and $\nu\res\Omega$ are mutually singular.
We assume that
$$\nu\left(\mathcal{N}_\perp^\Omega\right)=0.$$
\end{assumption}

\begin{remark}

Note that, for $x\in\Omega$ such that $m_x(\Omega)>0$, if $m_x\ll \nu$ (i.e., $m_x$ is   absolutely continuous with respect to $\nu$, do not confuse the use of $\ll$ in this context with its use in the notation in Subsection~\ref{secacc}) then $(m_x\res\Omega)\not\perp (\nu\res\Omega)$. Therefore, by Remark \ref{omegamconnectedthenmxomegapos}, if $m_x\ll\nu$ for $\nu$-a.e. $x\in\Omega$ then $\nu\left(\mathcal{N}_\perp^\Omega\right)=0$. Hence, the above condition is weaker than assuming that $m_x\ll \nu$ for $\nu$-a.e. $x\in\Omega$.
\end{remark}

\begin{assumption}\label{assumption2}
We assume, together with $0\in\gamma(0)\cap\beta(0)$, that
$$\mathcal{R}_{\gamma,\beta}^-< \mathcal{R}_{\gamma,\beta}^+,$$
where
$$\begin{array}{c}
\mathcal{R}_{\gamma,\beta}^-:=\nu(\Omega_1)\inf \mbox{Ran}(\gamma) + \nu(\Omega_2)\inf \mbox{Ran}(\beta),
\\[6pt]
 \mathcal{R}_{\gamma,\beta}^+:=\nu(\Omega_1)\sup \mbox{Ran}(\gamma) + \nu(\Omega_2)\sup \mbox{Ran}(\beta).
 \end{array}
$$
\end{assumption}

\begin{assumption}\label{assumption3}
We assume that the following generalised Poincar\'{e} type inequality holds:
For every $0<l\le \nu(\Omega)$,   there exists a constant $\Lambda>0$ such that,  for every $u \in L^p(\Omega,\nu)$ and any measurable set $Z\subset \Omega$ with $\nu(Z)\ge l$,
$$  \left\Vert  u \right\Vert_{L^p(\Omega,\nu)}  \leq \Lambda\left(\left(\int_{ \Omega\times\Omega} |u(y)-u(x)|^p dm_x(y) d\nu(x) \right)^{\frac1p}+\left| \int_Z u\,d\nu\right|\right).
$$
{\rm This assumption holds true in many important examples (see Appendix~\ref{secineqpoin}).}
\end{assumption}

From now on in this subsection we work under Assumptions~\ref{assumption1} to~\ref{assumption3}.

\begin{definition}\label{defsol01}
A solution of $(GP_\varphi)$ is  a pair $[u,v]$ with $u\in L^p(\Omega,\nu)$ and   $v\in L^{p'}(\Omega,\nu)$  such that
\\
1. $v(x)\in\gamma(u(x))\ \hbox{ for $\nu$-a.e. } x\in\Omega_1,$
\\[6pt]
2. $v(x)\in\beta(u(x))\ \hbox{ for $\nu$-a.e. } x\in\Omega_2,$
\\[6pt]
3. $[(x,y)\mapsto a_p(x,y,u(y)-u(x))]\in L^{p'}(\Omega\times\Omega,\nu\otimes m_x)$,
\\[6pt]
4. and
$$
 v(x) - \int_{\Omega} \a_p(x,y,u(y)-u(x)) dm_x(y)=\varphi(x), \quad x \in \Omega.
$$

A subsolution (supersolution) of $(GP_\varphi)$ is  a pair $[u,v]$ with $u\in L^p(\Omega,\nu)$ and $v\in L^1(\Omega,\nu)$ satisfying 1., 2., 3.   and
$$
 v(x) - \int_{\Omega} \a_p(x,y,u(y)-u(x)) dm_x(y)\le \varphi(x), \quad x \in \Omega,
$$
$$
 \left(v(x) - \int_{\Omega} \a_p(x,y,u(y)-u(x)) dm_x(y)\ge \varphi(x), \quad x \in \Omega\right).
$$
\end{definition}

\begin{remark}[Integration by parts formula]\label{remmon}
The following integration by parts formula which results from the reversibility of $\nu$ with respect to~$m$, can be easily proved. Let $u$ be a measurable function such that
   $$[(x,y)\mapsto \a_p(x,y,u(y)-u(x))]\in L^{q}( \Omega\times\Omega,\nu\otimes m_x)$$ and let $w \in L^{q'}(\Omega,\nu)$. Then
$$ \begin{array}{l}
 \displaystyle
-\int_{\Omega}\int_{\Omega}  \a_p(x,y,u(y)- u (x))dm_x(y)w(x)d\nu(x)   \\ [14pt]
= \displaystyle\frac{1}{2} \int_{\Omega\times\Omega} \a_p(x,y,u(y)-u(x)) (w(y) - w(x)) d(\nu\otimes m_x)(x,y) .
\end{array}
$$

Let us see, formally, the way in which we use the above integration by parts formula in what follows. Suppose that we are in the following situation:
$$
 \left\{ \begin{array}{ll} \displaystyle-\int_{\Omega}  \a_p(x,y,u(y)- u (x))dm_x(y) = f(x), \quad & x\in\Omega_1, \\[12pt]
\displaystyle   -\int_{\Omega} \a_p(x,y,u(y)- u (x))dm_x(y)  = g(x), \quad & x \in \Omega_2. \end{array} \right.
$$
Then, multiplying both equations by a test function $w$, integrating them with respect to $\nu$ over $\Omega_1$ and $\Omega_2$, respectively, adding them and using the integration by parts formula we get
$$ \begin{array}{l}
 \displaystyle
  \frac{1}{2} \int_{\Omega\times\Omega} \a_p(x,y,u(y)-u(x)) (w(y) - w(x)) d(\nu\otimes m_x)(x,y) \\[14pt]  \displaystyle = \int_{\Omega_1} f(x)w(x)d\nu(x) +\int_{\Omega_2} g(x)w(x) d\nu(x) .
\end{array}
$$
Moreover, as a consequence of these computations and \eqref{llo3}, taking $u=u_i$, $f=f_i$ and $g=g_i$, $i=1,2$, in the above system and for every nondecreasing function $T:\mathbb{R}\to \mathbb{R}$ we obtain
\begin{equation}\label{IntByPartsConsequenceNondecreasing}
\begin{array}{l}
 \displaystyle
 \int_{\Omega_1} (f_1(x)-f_2(x))T(u_1(x)-u_2(x))d\nu(x) +\int_{\Omega_2} (g_1(x)-g_2(x))T(u_1(x)-u_2(x)) d\nu(x)  \\ [14pt]
 \displaystyle
 =\frac{1}{2} \int_{\Omega\times\Omega} \big(\a_p(x,y,u_1(y)-u_1(x))-\a_p(x,y,u_2(y)-u_2(x))\big)
  \\[14pt] \displaystyle
  \hspace{80pt} \times\big(T(u_1(y) - u_2(y)) -T(u_1(x)-u_2(x))\big)d(\nu\otimes m_x)(x,y)\,\ge 0 .
\end{array}
\end{equation}

  \end{remark}

The next result gives a maximum principle  for solutions of Problem $(GP_\varphi)$ given in~\eqref{02091131general} and, consequently, also for solutions of Problem~\eqref{02091131}.

\begin{theorem}[Contraction and comparison principle]\label{maxandcont01}
Let $\varphi_1$, $\varphi_2\in L^{1}(\Omega,\nu)$. Let $[u_{1},v_1]$ be a subsolution of $(GP_{\varphi_1})$ and $[u_{2},v_2]$ be a supersolution of $(GP_{\varphi_2})$. Then,
\begin{equation}\label{mx01}\int_\Omega (v_1-v_2)^+d\nu \le \int_\Omega (\varphi_1-\varphi_2)^+d\nu .
\end{equation}
  Moreover, if $\varphi_1\le \varphi_2$ with $\varphi_1\neq\varphi_2$, then   $v_1\le v_2$, $v_1\neq v_2$,  and $u_1\le u_2$ $\nu$-a.e. in $\Omega$.

  Furthermore, if $\varphi_1= \varphi_2$ and $[u_i,v_i]$ is a solution of $(GP_{\varphi_i})$, $i=1,2$, then $v_1=v_2$ $\nu$-a.e. in $\Omega$ and   $u_1-u_2$ is $\nu$-a.e. equal to a constant.
\end{theorem}

\begin{proof}
 By hypothesis,
$$ v_1(x)-v_2(x) - \int_{\Omega} (\a_p(x,y,u_1(y)-u_1(x))-\a_p(x,y,u_2(y)-u_2(x))) dm_x(y)\le\varphi_1(x)-\varphi_2(x)$$
for $x \in \Omega$. Let $k>0$ and $T_k:\R\rightarrow [-k,k]$ be the truncation operator defined as
\begin{equation}\label{Def.Trunc.Op}
T_k(r):= \left\{ \begin{array}{lll}
-k \quad &\hbox{if} \ r< -k,\\[4pt]
r \quad &\hbox{if} \ \vert r \vert \leq k, \\[4pt]k \quad &\hbox{if} \ r >k,  \end{array}\right.
\end{equation}
and denote $T_k^+(s):=(T_k(s))^+$. Multiplying the above inequality by $\frac1k T_k^+(u_{1}-u_{2}+ k \, \hbox{sign}_0^+(v_1-v_2))$ and integrating over $\Omega$ we get
\begin{equation}\label{comparisonineq}\begin{array}{l}
\displaystyle\int_{\Omega}\left(v_1(x)-v_2(x)\right)\frac1k T_k^+(u_{1}(x)-u_{2}(x)+ k \, \hbox{sign}_0^+(v_1(x)-v_2(x))) d\nu(x) \\ [14pt]
\displaystyle \hspace{10pt}
- \int_{\Omega}\int_{\Omega} (\a_p(x,y,u_1(y)-u_1(x))-\a_p(x,y,u_2(y)-u_2(x))) dm_x(y)\\ [12pt]
\displaystyle \hspace{80pt} \times\frac1k T_k^+(u_{1}(x)-u_{2}(x)+ k \, \hbox{sign}_0^+(v_1(x)-v_2(x)))d\nu(x) \\ [14pt]
\displaystyle \le \int_{\Omega}(\varphi_1(x)-\varphi_2(x)) \frac1k T_k^+(u_{1}(x)-u_{2}(x)+ k \, \hbox{sign}_0^+(v_1(x)-v_2(x)))d\nu(x)\\ [14pt]
\displaystyle \le \int_{\Omega}(\varphi_1(x)-\varphi_2(x))^+ d\nu(x).
\end{array}\end{equation}
Moreover, by the integration by parts formula (Remark~\ref{remmon}),
$$ \begin{array}{l}
 \displaystyle - \int_{\Omega}\int_{\Omega} (\a_p(x,y,u_1(y)-u_1(x))-\a_p(x,y,u_2(y)-u_2(x)))dm_x(y)
\\ [12pt]
\displaystyle \hspace{80pt} \times\frac1k T_k^+(u_{1}(x)-u_{2}(x)+ k \, \hbox{sign}_0^+(v_1(x)-v_2(x)))  d\nu(x)\\ \\
 \displaystyle=\frac12 \int_{\Omega}\int_{\Omega} (\a_p(x,y,u_1(y)-u_1(x))-\a_p(x,y,u_2(y)-u_2(x)))
\\ [12pt]
\displaystyle \hspace{80pt} \times\Big(\frac1k T_k^+(u_{1}(y)-u_{2}(y)+ k \, \hbox{sign}_0^+(v_1(y)-v_2(y)))
\\ [12pt]
\hspace{100 pt}\displaystyle -\frac1k T_k^+(u_{1}(x)-u_{2}(x)+ k \, \hbox{sign}_0^+(v_1(x)-v_2(x)))\Big)  dm_x(y)d\nu(x).
\end{array}$$
  Now, since the integrand on the right hand side is bounded from below by an integrable function,
we can apply Fatou's lemma to get (recall the last observation in Remark~\ref{remmon})
$$ \begin{array}{l}
\displaystyle  \liminf_{k\to 0^+}- \int_{\Omega}\int_{\Omega} (\a_p(x,y,u_1(y)-u_1(x))-\a_p(x,y,u_2(y)-u_2(x)))dm_x(y)
\\ [12pt]
\displaystyle \hspace{75pt} \times\frac1k T_k^+(u_{1}(x)-u_{2}(x)+ k \, \hbox{sign}_0^+(v_1(x)-v_2(x)))  d\nu(x)\ge 0.
\end{array}$$
Hence,  taking limits in~\eqref{comparisonineq}, we get
$$\begin{array}{l}
\displaystyle \int_{\Omega}\left(v_1(x)-v_2(x)\right)^+ d\nu(x) \\ [12pt]
\displaystyle=\lim_{k\to 0^+}\int_{\Omega}\left(v_1(x)-v_2(x)\right)\frac1k T_k^+(u_{1}(x)-u_{2}(x)+ k \, \hbox{sign}_0^+(v_1(x)-v_2(x))) d\nu(x) \\ [12pt]
\displaystyle \le \int_{\Omega}(\varphi_1(x)-\varphi_2(x))^+ d\nu(x),
\end{array}$$
and~\eqref{mx01} is proved.

Take now $\varphi_1\le \varphi_2$ with $\varphi_1\neq \varphi_2$, then, by \eqref{mx01},  $v_1\le v_2$ $\nu$-a.e. in $\Omega$. Now, since $[u_{1},v_1]$ is a subsolution of $(GP_{\varphi_1})$
$$ v_1(x) - \int_{\Omega} \a_p(x,y,u_1(y)-u_1(x)) dm_x(y)\le\varphi_1(x)$$
thus
$$ \int_\Omega v_1(x)d\nu(x) - \underbrace{\int_\Omega \int_{\Omega} \a_p(x,y,u_1(y)-u_1(x)) dm_x(y)d\nu(x)}_{=0} \le \int_\Omega \varphi_1(x)d\nu(x).$$
Therefore, with the same calculation for $[u_{2},v_2]$,
$$ \int_\Omega v_1(x)d\nu(x)  \le \int_\Omega \varphi_1(x)d\nu(x)<\int_\Omega \varphi_2(x)d\nu(x)\le \int_\Omega v_2(x)d\nu(x)$$
thus $v_1\neq v_2$.
Now, since $(\varphi_1-\varphi_2)^+=0$ and $(v_1-v_2)^+=0$, from \eqref{comparisonineq} we get that
$$\begin{array}{l}
\displaystyle\int_{\Omega}\left(v_1(x)-v_2(x)\right)\frac1k T_k^+(u_{1}(x)-u_{2}(x))) d\nu(x) \\ [14pt]
\displaystyle \hspace{10pt}
- \int_{\Omega}\int_{\Omega} (\a_p(x,y,u_1(y)-u_1(x))-\a_p(x,y,u_2(y)-u_2(x)))\frac1k T_k^+(u_{1}(x)-u_{2}(x)) dm_x(y)d\nu(x)\le 0.
\end{array}$$
However, since $v_i(x)\in \gamma(u_i(x))$ for $\nu$-a.e. $x\in\Omega_1$ and $v_i(x)\in \beta(u_i(x))$ for $\nu$-a.e. $x\in\Omega_2$, $i=1,2$, we get, by the monotonicity of the graphs, $u_1(x)\le u_2(x)$ for $\nu$-a.e. $x\in\Omega$ such that $v_1(x)<v_2(x)$. Therefore, $\left(v_1(x)-v_2(x)\right)\frac1k T_k^+(u_{1}(x)-u_{2}(x)))=0$ for $\nu$-a.e. $x\in\Omega$ and thus
$$\begin{array}{l}
\displaystyle
- \int_{\Omega}\int_{\Omega} (\a_p(x,y,u_1(y)-u_1(x))-\a_p(x,y,u_2(y)-u_2(x)))\frac1k T_k^+(u_{1}(x)-u_{2}(x)) dm_x(y)d\nu(x)  \le 0.
\end{array}$$
Now, recalling Remark~\ref{remmon} (that is, integration by parts),  we obtain
$$\begin{array}{l}
\displaystyle
 \int_{\Omega}\int_{\Omega} (\a_p(x,y,u_1(y)-u_1(x))-\a_p(x,y,u_2(y)-u_2(x)))\\[12pt]
 \displaystyle \qquad \quad \times((u_{1}(y)-u_{2}(y))^+-(u_{1}(x)-u_{2}(x))^+) dm_x(y)d\nu(x)  = 0,
\end{array}$$
and thus
\begin{equation}\label{ceroproduct}
(\a_p(x,y,u_1(y)-u_1(x))-\a_p(x,y,u_2(y)-u_2(x)))((u_{1}(y)-u_{2}(y))^+-(u_{1}(x)-u_{2}(x))^+)=0
\end{equation}
for $(x,y)\in(\Omega\times\Omega)\setminus N$ where $N\subset \Omega\times\Omega$ is a $\nu\otimes m_x$-null set.
Let $C\subset \Omega$ be a $\nu$-null set such that the section $N_x:=\{y\in\Omega \, :\, (x,y)\in N\}$ of $N$ is $m_x$-null for every $x\in \Omega\setminus C$ and let us see that $u_1(x)\le u_2(x)$ for every $x\in \Omega\setminus (C\cup\mathcal{N}_\perp^\Omega)$ (recall Assumption \ref{assumption4} for the definition of the $\nu$-null set $\mathcal{N}_\perp^\Omega$). Suppose that there exists  $x_0\in\Omega\setminus (C\cup\mathcal{N}_\perp^\Omega)$  such that $u_1(x_0)-u_2(x_0)>0$. Then, from \eqref{ceroproduct} (and \eqref{llo3}) we get that $u_1(y)-u_2(y)=u_1(x_0)-u_2(x_0)>0$ for every $y\in\Omega\setminus N_{x_0}$.    Let $$S:=\{y\in\Omega \, :\, u_1(y)-u_2(y)=u_1(x_0)-u_2(x_0)\}\supset \Omega\setminus N_{x_0}.$$
Since $x_0\not\in \mathcal{N}_\perp^\Omega$ and   $m_{x_0}(N_{x_0})=0$,
we must have $\nu(S)\ge \nu(\Omega\setminus N_{x_0})>0$. Now, following the same argument as before, if $x\in S$ then $\Omega\setminus N_x\subset S$ thus $m_x(\Omega\setminus S)\le m_x(N_x)=0$ and, therefore,
$$L_m(S,\Omega\setminus S)=0.$$
However, since $\Omega$ is $m$-connected and $\nu(S)>0$ we must have $\nu(\Omega\setminus S)=0$ thus $u_1(y)-u_2(y)=u_1(x_0)-u_2(x_0)>0$ for $\nu$-a.e. $y\in\Omega$. This contradicts that $v_1\le v_2$, $v_1\neq v_2$, $\nu$-a.e. in $\Omega$.

Finally, suppose that $[u_{1},v_1]$ and $[u_{2},v_2]$ are solutions of $(GP_{\varphi})$ for some $\varphi\in L^1(\Omega,\nu)$. Then,
$$ v_1(x)-v_2(x) - \int_{\Omega} (\a_p(x,y,u_1(y)-u_1(x))-\a_p(x,y,u_2(y)-u_2(x))) dm_x(y)=0$$
thus, since $v_1=v_2$ $\nu$-a.e. in $\Omega$,
$$-\int_{\Omega} (\a_p(x,y,u_1(y)-u_1(x))-\a_p(x,y,u_2(y)-u_2(x))) dm_x(y)=0.$$
Multiplying this equation by $u_1-u_2$, integrating over $\Omega$ and using the integration by parts formula as in Remark~\ref{remmon} we get
$$\int_{\Omega}\int_{\Omega} (\a_p(x,y,u_1(y)-u_1(x))-\a_p(x,y,u_2(y)-u_2(x)))(u_1(y)-u_1(x)-(u_2(y)-u_2(x))) dm_x(y)d\nu(x)=0$$
thus, by \eqref{llo3} and positivity,
\begin{equation}\label{ceroagain}
(\a_p(x,y,u_1(y)-u_1(x))-\a_p(x,y,u_2(y)-u_2(x)))(u_1(y)-u_1(x)-(u_2(y)-u_2(x)))=0
\end{equation}
for $(x,y)\in(\Omega\times\Omega)\setminus N'$ where $N'\subset \Omega\times\Omega$ is a $\nu\otimes m_x$-null set. Let $C'\subset \Omega$ be a $\nu$-null set such that the section $N'_x:=\{y\in\Omega \, :\, (x,y)\in N'\}$ of $N'$ is $\nu$-null for every $x\in \Omega\setminus C'$ and let us see that there exists $L\in \R$ such that $u_1(x)- u_2(x)=L$ for $\nu$-a.e. $x\in\Omega$. Let $x_0\in\Omega\setminus C'$, $L:=u_1(x_0)- u_2(x_0)$ and
$$S':=\{y\in\Omega \, :\, u_1(y)-u_2(y)=L \}\supset \Omega\setminus N'_{x_0}.$$
By \eqref{ceroagain},  $\Omega\setminus C'_{x_0}\subset S'$. Proceeding as we did before to prove that $\nu(\Omega\setminus S)=0$ we obtain that $\nu(\Omega\setminus S')=0$. \qed
\end{proof}

  In order to prove the existence of solutions of Problem~\eqref{02091131general} (Theorem~\ref{existenceeli01}) we  first prove the existence of solutions of an approximate problem. Then we  obtain some monotonicity and boundedness properties of the solutions of these approximate problems that  allow us to pass to the limit. This method lets us get around the loss of compactness results in our setting with respect to the local setting. Indeed, we follow ideas used in~\cite{AIMTq}, but, as we have said, making the most of the monotonicity arguments since the Poincar\'{e} type inequalities here only produce boundedness in $L^{p}$ spaces   (versus the boundedness in $W^{1,p}$ spaces obtained in their local setting). This will be done in the following subsections.

\subsubsection{Existence of solutions of an approximate problem}\label{existenciaurnk}

Take $\varphi\in L^{\infty}(\Omega,\nu)$. Let $n, k\in \mathbb{N}$, $K>0$
and $$A:=A_{n,k}:L^p(\Omega,\nu) \rightarrow  L^{p'}(\Omega,\nu)\equiv L^{p'}(\Omega_1,\nu)\times L^{p'}(\Omega_2,\nu)$$ be defined by
$$A(u)= \big(A_1(u),A_2(u)\big),$$ where
$$ \begin{array}{rl}
\displaystyle A_1(u)(x):=&\displaystyle T_K((\gamma_+)_k(u(x)))+T_K((\gamma_-)_n(u(x)))-\int_{\Omega}\a_p(x,y,u(y)-u(x))dm_x(y)\\ [12pt]
& \displaystyle \hspace{10pt} +\frac1n |u(x)|^{p-2}u^+(x) -\frac1k  |u(x)|^{p-2}u^-(x),
\end{array}$$
for $x\in\Omega_1$,
and
$$\begin{array}{rl}
\displaystyle A_2(u)(x):=&\displaystyle T_K((\beta_+)_k(u(x)))+T_K((\beta_-)_n(u(x)))-\int_{\Omega}\a_p(x,y,u(y)-u(x))dm_x(y)\\ [12pt]
& \displaystyle \hspace{10pt} +\frac1n |u(x)|^{p-2}u^+(x)-\frac1k |u(x)|^{p-2}u^-(x),
\end{array}$$
for $x\in \Omega_2$. Here, $T_K$ is the truncation operator defined in \eqref{Def.Trunc.Op} and $(\gamma_+)_k$, $(\gamma_-)_n$, $(\beta_+)_k$ and $(\beta_-)_n$ are Yosida approximations as defined in Subsection \ref{secacc}.

  It is easy to see that $A$ is continuous and, moreover, it is monotone and coercive in $L^p(\Omega, \nu)$.
Indeed, the monotonicity  results from the integration by parts formula (Remark~\ref{remmon}) and  the coercivity results from the following computation (where the term involving $\a_p$ has been neglected because it is nonnegative, as shown in Remark \ref{remmon}):
$$\int_{\Omega} A(u)u d\nu\ge\frac{1}{n}  || u^+||_{L^p(\Omega,\nu)}+\frac{1}{k}|| u^-||_{L^p(\Omega,\nu)} .$$
Therefore,  since $\varphi \in L^{\infty}(\Omega,\nu)\subset L^{p'}(\Omega,\nu)$, by  \cite[Corollary 30]{brezisgrenoble},  there exist $u_{n,k}\in L^{p}(\Omega, \nu)$, $n$, $k\in\N$, such that
$$\big(A_1(u_{n,k}),A_2(u_{n,k})\big)=\varphi.$$
That is,
\begin{equation}\label{E1} \begin{array}{l}T_K((\gamma_+)_k(u_{n,k}(x)))+T_K((\gamma_-)_n(u_{n,k}(x)))
-\displaystyle\int_{\Omega}\a_p(x,y,u_{n,k}(y)-u_{n,k}(x))dm_x(y)\\ [12pt]
\ \displaystyle +\frac{1}{n} |u_{n,k}(x)|^{p-2}u_{n,k}^+(x)   -\frac{1}{k} |u_{n,k}(x)|^{p-2}u_{n,k}^-(x) =\varphi(x)  \  \hbox{ for $x\in\Omega_1,$} \end{array}\end{equation}
and
\begin{equation}\label{E2} \begin{array}{l}T_K((\beta_+)_k(u_{n,k}(x)))+T_K((\beta_-)_n(u_{n,k}(x)))-\displaystyle\int_{\Omega}\a_p(x,y,u_{n,k}(y)-u_{n,k}(x))dm_x(y)\\ [12pt]
\ \displaystyle +\frac{1}{n} |u_{n,k}(x)|^{p-2}u_{n,k}^+(x)   -\frac{1}{k} |u_{n,k}(x)|^{p-2}u_{n,k}^-(x) =\varphi(x)  \  \hbox{ for $x\in\Omega_2$.} \end{array}\end{equation}

  Let $n$, $k\in\N$. We start by proving that $u_{n,k}\in L^\infty(\Omega,\nu)$.
   Set
$$M:=\left((k+n)\Vert \varphi\Vert_{L^\infty(\Omega,\nu)}\right)^{\frac{1}{p-1}}.$$
 Then, multiplying \eqref{E1} and \eqref{E2} by $(u_{n,k}-M)^+$, integrating over $\Omega_1$ and $\Omega_2$, respectively, adding both equations and neglecting the terms which are zero, we get
\begin{equation}\label{rem001}\begin{array}{l}
\displaystyle\int_{\Omega_1} T_K((\gamma_+)_k(u_{n,k}(x))) (u_{n,k}(x)-M)^+ d\nu(x) + \int_{\Omega_2} T_K((\beta_+)_k(u_{n,k}(x))) (u_{n,k}(x)-M)^+ d\nu(x) \\ [12pt]
 \ \ -\displaystyle\int_{\Omega}\int_{\Omega}\a_p(x,y,u_{n,k}(y)-u_{n,k}(x))(u_{n,k}(x)-M)^+dm_x(y)  d\nu(x) \\ [12pt]
 \ \ \displaystyle + \frac{1}{n}\int_\Omega  |u_{n,k}(x)|^{p-2}u_{n,k}^+(x)(u_{n,k}(x)-M)^+d\nu(x)\\ [12pt]
 = \displaystyle\int_{\Omega} \varphi(x)(u_{n,k}(x)-M)^+ d \nu(x).
\end{array}
\end{equation}
Now, by the integration by parts formula (recall Remark \ref{remmon}),
$$\begin{array}{l} \displaystyle -\int_{\Omega}\int_{\Omega}\a_p(x,y,u_{n,k}(y)-u_{n,k}(x))(u_{n,k}(x)-M)^+ dm_x(y)d\nu(x) \\ [12pt]
\displaystyle =\frac12\int_{\Omega}\int_{\Omega}\a_p(x,y,u_{n,k}(y)-u_{n,k}(x))\left((u_{n,k}(y)-M)^+ -(u_{n,k}(x)-M)^+\right) dm_x(y)d\nu(x)\ge 0 . \end{array}
$$
Hence, neglecting nonnegative terms in~\eqref{rem001},  we get
$$
  \int_\Omega  |u_{n,k}(x)|^{p-2}u_{n,k}^+(x)(u_{n,k}(x)-M)^+d\nu(x)  \le n\displaystyle\int_{\Omega} \varphi(x)(u_{n,k}(x)-M)^+ d \nu(x),
$$
 thus
$$  \int_\Omega T_K(|u_{n,k}(x)|^{p-2}u_{n,k}^+(x))(u_{n,k}(x)-M)^+d\nu(x)
 \le n\displaystyle\int_{\Omega} \varphi(x)(u_{n,k}(x)-M)^+ d \nu(x).
$$
Now, subtracting $\displaystyle\int_\Omega  M^{p-1}(u_{n,k}(x)-M)^+d\nu(x)$ from both sides of the above inequality yields
$$\begin{array}{l}
\displaystyle \int_\Omega \left(T_K(|u_{n,k}(x)|^{p-2}u_{n,k}^+(x))-M^{p-1}\right)(u_{n,k}(x)-M)^+d\nu(x) \\ [12pt]
 \le \displaystyle  n\int_{\Omega}\left(\varphi(x)-\frac1n M^{p-1}\right)(u_{n,k}(x)-M)^+ d \nu(x)\le 0
\end{array}$$
 and, consequently,   taking $K>M$, we get
$$u_{n,k}\le M\quad\nu\hbox{-a.e. in }\hbox{$\Omega$}.$$
Similarly,   taking $w=(u_{n,k}+M)^-$, we get
$$\begin{array}{l}
\displaystyle \int_\Omega \left(T_K(|u_{n,k}(x)|^{p-2}u_{n,k}^-(x))+M^{p-1}\right)(u_{n,k}(x)+M)^-d\nu(x) \\ [12pt]
 \ge \displaystyle  k\int_{\Omega}\left(\varphi(x)+\frac1k M^{p-1}\right)(u_{n,k}+M)^- d \nu(x)\ge 0
\end{array}$$
which yields,    taking also $K>M$,
$$
u_{n,k}\ge -M\quad\nu\hbox{-a.e. in }\hbox{$\Omega$}.
$$
Therefore,
$$
\Vert u_{n,k}\Vert_{ L^\infty(\Omega,\nu)}\le M
$$
as desired.

Now, taking
$$K>\max\left\{M, (\gamma_+)_k(M), -(\gamma_-)_k(-M), (\beta_+)_n(M), -(\beta_-)_n(-M)\right\},$$
equations \eqref{E1} and \eqref{E2} yield
\begin{equation}\label{EE1} \begin{array}{l}(\gamma_+)_k(u_{n,k}(x))+(\gamma_-)_n(u_{n,k}(x))-\displaystyle\int_{\Omega}\a_p(x,y,u_{n,k}(y)-u_{n,k}(x))dm_x(y)\\ [12pt]
\hspace{10pt} \displaystyle +\frac{1}{n}|u_{n,k}(x)|^{p-2}u_{n,k}^+(x) -\frac{1}{k}|u_{n,k}(x)|^{p-2}u_{n,k}^-(x)=\varphi(x),  \ \ \  \hbox{ $x\in\Omega_1,$}  \end{array}\end{equation}
and
\begin{equation}\label{EE2} \begin{array}{l}(\beta_+)_k(u_{n,k}(x))+(\beta_-)_n(u_{n,k}(x))-\displaystyle\int_{\Omega}\a_p(x,y,u_{n,k}(y)-u_{n,k}(x))dm_x(y)\\ [12pt]
\hspace{10pt} \displaystyle +\frac{1}{n}|u_{n,k}(x)|^{p-2}u_{n,k}^+(x)  -\frac{1}{k}|u_{n,k}(x)|^{p-2}u_{n,k}^-(x)=\varphi(x), \ \ \  \hbox{ $x\in\Omega_2$.}  \end{array}\end{equation}

Take now $\varphi\in L^{p'}(\Omega,\nu)$ and, for $n,k\in \mathbb{N}$, set
\begin{equation}\label{varphink}
\varphi_{n,k}:=\sup\{\inf\{n,\varphi\},-k\}.
\end{equation}
Then, since $\varphi_{n,k}\in L^\infty(\Omega,\nu)$, by the previous computations leading to \eqref{EE1} and \eqref{EE2},  there exists a solution $u_{n,k}\in L^\infty(\Omega,\nu)$ of the following  {\it approximate problem} \eqref{EE1nk}--\eqref{EE2nk}:
\begin{align}
  &\begin{array}{l}(\gamma_+)_k(u_{n,k}(x))+(\gamma_-)_n(u_{n,k}(x))-\displaystyle\int_{\Omega}\a_p(x,y,u_{n,k}(y)-u_{n,k}(x))dm_x(y)\\ [12pt]
 \displaystyle \hspace{10pt}+\frac{1}{n}|u_{n,k}(x)|^{p-2}u_{n,k}^+(x) -\frac{1}{k}|u_{n,k}(x)|^{p-2}u_{n,k}^-(x)=\varphi_{n,k}(x) , \ \ \    x\in\Omega_1,   \end{array}\label{EE1nk}\\[12pt]
  &\begin{array}{l}(\beta_+)_k(u_{n,k}(x))+(\beta_-)_n(u_{n,k}(x))-\displaystyle\int_{\Omega}\a_p(x,y,u_{n,k}(y)-u_{n,k}(x))dm_x(y)\\ [12pt]
 \displaystyle \hspace{10pt} +\frac{1}{n}|u_{n,k}(x)|^{p-2}u_{n,k}^+(x)  -\frac{1}{k}|u_{n,k}(x)|^{p-2}u_{n,k}^-(x)=\varphi_{n,k}(x) , \ \ \   x\in\Omega_2.  \end{array}\label{EE2nk}
\end{align}
Moreover, we obtain the following estimates which will be used later on.
Multiplying \eqref{EE1nk} and \eqref{EE2nk} by $\frac1s T_s(u_{n,k}^+)$, integrating with respect to $\nu$ over $\Omega_1$ and $\Omega_2$, respectively, adding both equations,   applying the integration by parts formula (Remark \ref{remmon}), and letting $s\downarrow 0$, we get, after neglecting some nonnegative terms, that
\begin{equation}\label{cota1}
\frac1n \int_{\Omega}|u_{n,k}|^{p-2}u_{n,k}^+d\nu+\int_{\Omega_1}(\gamma_+)_k(u_{n,k})d\nu+\int_{\Omega_2}(\beta_+)_k(u_{n,k})d\nu \le \int_{\Omega}\varphi_{n,k}^+d\nu \le \int_{\Omega}\varphi^+d\nu.
\end{equation}
Similarly, multiplying by $\frac1s T_s(u_{n,k}^-)$ we get
\begin{equation}\label{cota2}
-\frac1k \int_{\Omega}|u_{n,k}|^{p-2}u_{n,k}^-d\nu+\int_{\Omega_1}(\gamma_-)_n(u_{n,k})d\nu+\int_{\Omega_2}(\beta_-)_n(u_{n,k})d\nu \ge -\int_{\Omega}\varphi_{n,k}^-d\nu \ge -\int_{\Omega}\varphi^-d\nu.
\end{equation}

\subsubsection{Monotonicity of the solutions of the approximate problems}\label{secmon}
Using that $\varphi_{n,k}$ is nondecreasing in $n$ and nonincreasing in~$k$, and thanks to the way in which we have approximated the maximal monotone graphs $\gamma$ and $\beta$, we  obtain monotonicity properties for the solutions of the approximate problems.

Fix $k\in\N$. Let $n_1<n_2$. Multiply equations \eqref{EE1nk} and \eqref{EE2nk} with $n=n_1$ by $(u_{n_1,k}-u_{n_2,k})^+$, integrate with respect to $\nu$ over $\Omega_1$ and $\Omega_2$, respectively, and add both equations. Then, doing the same with $n=n_2$ and subtracting the resulting equation from the one that we have obtained for $n=n_1$ we get
$$
\begin{array}{l}
\displaystyle \int_{\Omega_1}\left((\gamma_+)_k(u_{n_1,k}(x))-(\gamma_+)_k(u_{n_2,k}(x))\right)(u_{n_1,k}(x)-u_{n_2,k}(x))^+d\nu(x)\\[12pt]
\displaystyle \hspace{10pt} +\int_{\Omega_1}\left((\gamma_-)_{n_1}(u_{n_1,k}(x))-(\gamma_-)_{n_2}(u_{n_2,k}(x))\right)(u_{n_1,k}(x)-u_{n_2,k}(x))^+d\nu(x)\\[12pt]
\displaystyle \hspace{10pt} +\int_{\Omega_2}\left((\beta_+)_k(u_{n_1,k}(x))-(\beta_+)_k(u_{n_2,k}(x))\right)(u_{n_1,k}(x)-u_{n_2,k}(x))^+d\nu(x)\\[12pt]
\displaystyle \hspace{10pt} +\int_{\Omega_2}\left((\beta_-)_{n_1}(u_{n_1,k}(x))-(\beta_-)_{n_2}(u_{n_2,k}(x))\right)(u_{n_1,k}(x)-u_{n_2,k}(x))^+d\nu(x)\\[12pt]
\displaystyle \hspace{10pt} -\int_{\Omega}\int_{\Omega}(\a_p(x,y,u_{n_1,k}(y)-u_{n_1,k}(x))-\a_p(x,y,u_{n_2,k}(y)-u_{n_2,k}(x)))\\[11pt]
\hspace{200pt} \displaystyle \times(u_{n_1,k}(x)-u_{n_2,k}(x))^+ dm_x(y)d\nu(x)\\ [14pt]
\displaystyle \hspace{10pt} +\int_{\Omega}\left(\frac{1}{n_1}|u_{n_1,k}(x)|^{p-2}u_{n_1,k}^+(x)  -\frac{1}{n_2}|u_{n_2,k}(x)|^{p-2}u_{n_2,k}^+(x) \right)(u_{n_1,k}(x)-u_{n_2,k}(x))^+d\nu(x)
\\ [12pt]
\displaystyle \hspace{10pt} -\frac{1}{k}\int_{\Omega}\left( |u_{n_1,k}(x)|^{p-2}u_{n_1,k}^-(x)-|u_{n_2,k}(x)|^{p-2}u_{n_2,k}^-(x)\right)(u_{n_1,k}(x)-u_{n_2,k}(x))^+d\nu(x)\\[14pt]
\displaystyle =\int_\Omega \left(\varphi_{n_1,k}(x)-\varphi_{n_2,k}(x)\right)(u_{n_1,k}(x)-u_{n_2,k}(x))^+d\nu(x)\le 0.
\end{array}$$
Since $(\gamma_+)_k$ and $(\beta_+)_k$ are maximal monotone, the first and third  summands on the left hand side are nonnegative, and the same is true for   the second and fourth  summands since $(\gamma_-)_{n_1}\ge (\gamma_-)_{n_2}$, $(\beta_-)_{n_1}\ge (\beta_-)_{n_2}$ and these are all maximal monotone. The fifth summand is also nonnegative as illustrated in Remark \ref{remmon}. Then, since the last two summands are obviously nonnegative, we get that, in fact,
$$\int_{\Omega}\left(\frac{1}{n_1}|u_{n_1,k}(x)|^{p-2}u_{n_1,k}^+(x)  -\frac{1}{n_2}|u_{n_2,k}(x)|^{p-2}u_{n_2,k}^+(x) \right)(u_{n_1,k}(x)-u_{n_2,k}(x))^+d\nu(x)=0$$
and
$$\frac{1}{k}\int_{\Omega}\left( |u_{n_1,k}(x)|^{p-2}u_{n_1,k}^-(x)-|u_{n_2,k}(x)|^{p-2}u_{n_2,k}^-(x)\right)(u_{n_1,k}(x)-u_{n_2,k}(x))^+d\nu(x)=0$$
which together imply that $$u_{n_1,k}(x)\le u_{n_2,k}(x)\quad\hbox{for $\nu$-a.e. $x\in\Omega$}.$$

Similarly, we obtain that, for a fixed $n$, $u_{n,k}$ is $\nu$-a.e. in $\Omega$ nonincreasing in $k$.

\subsubsection{An $L^{p}$-estimate for the solutions of the approximate problems}\label{lpestimate}
Multiplying \eqref{EE1nk} and \eqref{EE2nk} by $$u_{n,k}-\frac{1}{\nu(\Omega_1)}\int_{\Omega_1}u_{n,k}d\nu,$$
integrating with respect to $\nu$ over $\Omega_1$ and $\Omega_2$, respectively, adding both equations and using the integration by parts formula (Remark \ref{remmon}) we get
\begin{equation}\label{first233}
\begin{array}{l}
\displaystyle\int_{\Omega_1}\left((\gamma_+)_k(u_{n,k}(x))+(\gamma_-)_n(u_{n,k}(x))\right)\left(u_{n,k}(x)-\frac{1}{\nu(\Omega_1)}\int_{\Omega_1}u_{n,k}d\nu\right)d\nu(x) \\ [14pt]
\hspace{10pt} \displaystyle + \int_{\Omega_2}\left((\beta_+)_k(u_{n,k}(x))+(\beta_-)_n(u_{n,k}(x))\right)\left(u_{n,k}(x)-\frac{1}{\nu(\Omega_1)}\int_{\Omega_1}u_{n,k}d\nu\right)d\nu(x) \\ [14pt]
\hspace{10pt} \displaystyle+\frac12 \int_{\Omega}\int_{\Omega}\a_p(x,y,u_{n,k}(y)-u_{n,k}(x))(u_{n,k}(y)-u_{n,k}(x))dm_x(y)d\nu(x)\\ [14pt]
\hspace{10pt} \displaystyle + \int_\Omega \left(\frac{1}{n}|u_{n,k}(x)|^{p-2}u_{n,k}^+(x) -\frac{1}{k}|u_{n,k}(x)|^{p-2}u_{n,k}^-(x)\right)\left(u_{n,k}(x)-\frac{1}{\nu(\Omega_1)}\int_{\Omega_1}u_{n,k}d\nu\right)d\nu(x)\\ [14pt]
\displaystyle = \int_{\Omega}\varphi_{n,k}(x)\left(u_{n,k}(x)-\frac{1}{\nu(\Omega_1)}\int_{\Omega_1}u_{n,k}d\nu\right)d\nu(x).
\end{array}
\end{equation}
For the first summand on the left hand side of \eqref{first233} we have
$$\begin{array}{l}
\displaystyle\int_{\Omega_1}\left((\gamma_+)_k(u_{n,k})+(\gamma_-)_n(u_{n,k})\right)\left(u_{n,k}-\frac{1}{\nu(\Omega_1)}\int_{\Omega_1}u_{n,k}d\nu\right)d\nu \\ [14pt]
\displaystyle =\int_{\Omega_1}\left((\gamma_+)_k(u_{n,k})-(\gamma_+)_k\left(\frac{1}{\nu(\Omega_1)}\int_{\Omega_1}u_{n,k}\right)\right)\left(u_{n,k}-\frac{1}{\nu(\Omega_1)}\int_{\Omega_1}u_{n,k}d\nu\right) d\nu \\ [14pt]
\displaystyle \hspace{10pt} + \int_{\Omega_1}\left((\gamma_-)_n(u_{n,k})-(\gamma_-)_n\left(\frac{1}{\nu(\Omega_1)}\int_{\Omega_1}u_{n,k}\right)\right)\left(u_{n,k}-\frac{1}{\nu(\Omega_1)}\int_{\Omega_1}u_{n,k}d\nu\right) d\nu \ge 0,\end{array}$$
and for the second
$$\begin{array}{l}
\displaystyle\int_{\Omega_2}\left((\beta_+)_k(u_{n,k})+(\beta_-)_n(u_{n,k})\right)\left(u_{n,k}-\frac{1}{\nu(\Omega_1)}\int_{\Omega_1}u_{n,k}d\nu\right)d\nu \\ [14pt]
\displaystyle =\int_{\Omega_2}\left((\beta_+)_k(u_{n,k})-(\beta_+)_k\left(\frac{1}{\nu(\Omega_2)}\int_{\Omega_2}u_{n,k}\right)\right)\left(u_{n,k}-\frac{1}{\nu(\Omega_2)}\int_{\Omega_2}u_{n,k}d\nu\right) d\nu \\ [14pt]
\displaystyle \hspace{10pt} +\int_{\Omega_2}\left((\beta_-)_n(u_{n,k})-(\beta_-)_n\left(\frac{1}{\nu(\Omega_2)}\int_{\Omega_2}u_{n,k}\right)\right)\left(u_{n,k}-\frac{1}{\nu(\Omega_2)}\int_{\Omega_2}u_{n,k}d\nu\right) d\nu \\ [14pt]
\displaystyle \hspace{10pt} -\int_{\Omega_2}\left((\beta_+)_k(u_{n,k})+(\beta_-)_n(u_{n,k})\right)\left(\frac{1}{\nu(\Omega_1)}\int_{\Omega_1}u_{n,k}d\nu-\frac{1}{\nu(\Omega_2)}\int_{\Omega_2}u_{n,k}d\nu\right) d\nu \\ [14pt]
\displaystyle \ge -\int_{\Omega_2}\left((\beta_+)_k(u_{n,k})+(\beta_-)_n(u_{n,k})\right)\left(\frac{1}{\nu(\Omega_1)}\int_{\Omega_1}u_{n,k}d\nu-\frac{1}{\nu(\Omega_2)}\int_{\Omega_2}u_{n,k}d\nu\right) d\nu.
\end{array}$$
Since $  F_{n,k} (s):=\frac1n |s|^{p-2}s^+-\frac1k  |s|^{p-2}s^-$ is nondecreasing, for the fourth summand on the left hand side of \eqref{first233} we have that
$$\begin{array}{l}
\displaystyle  \int_\Omega \left(\frac{1}{n}|u_{n,k}(x)|^{p-2}u_{n,k}^+(x) -\frac{1}{k}|u_{n,k}(x)|^{p-2}u_{n,k}^-(x)\right)\left(u_{n,k}(x)-\frac{1}{\nu(\Omega_1)}\int_{\Omega_1}u_{n,k}d\nu\right)d\nu(x)\\[14pt]
\displaystyle  =\int_{\Omega_1} \left(F_{n,k} (u_{n,k}(x)) - F_{n,k} \left(\frac{1}{\nu(\Omega_1)}\int_{\Omega_1}u_{n,k}d\nu \right) \right) \left(u_{n,k}(x)-\frac{1}{\nu(\Omega_1)}\int_{\Omega_1}u_{n,k}d\nu\right)d\nu(x)\\[14pt]
\displaystyle  \hspace{10pt} + \int_{\Omega_2} \left(F_{n,k} (u_{n,k}(x)) - F_{n,k} \left(\frac{1}{\nu(\Omega_2)}\int_{\Omega_1}u_{n,k}d\nu \right) \right) \left(u_{n,k}(x)-\frac{1}{\nu(\Omega_2)}\int_{\Omega_2}u_{n,k}d\nu\right)d\nu(x)\\[14pt]
\displaystyle  \hspace{10pt} - \int_{\Omega_2} F_{n,k} (u_{n,k}(x)) \left(\frac{1}{\nu(\Omega_1)}\int_{\Omega_1}u_{n,k}d\nu-\frac{1}{\nu(\Omega_2)}\int_{\Omega_2}u_{n,k}d\nu\right)d\nu(x)\\[14pt]
\displaystyle  \ge-\int_{\Omega_2} F_{n,k} (u_{n,k}(x)) \left(\frac{1}{\nu(\Omega_1)}\int_{\Omega_1}u_{n,k}d\nu-\frac{1}{\nu(\Omega_2)}\int_{\Omega_2}u_{n,k}d\nu\right)d\nu(x).
\end{array}$$
Finally, recalling \eqref{llo2} for the third summand in \eqref{first233}, we get
$$\begin{array}{l}
\displaystyle   \frac{c_p}{2}\int_{\Omega}\int_{\Omega}|u_{n,k}(y)-u_{n,k}(x)|^p dm_x(y)d\nu(x) \\ [14pt]
\displaystyle \le \int_{\Omega}\varphi_{n,k} \left(u_{n,k}-\frac{1}{\nu(\Omega_1)}\int_{\Omega_1}u_{n,k}d\nu\right)d\nu \\ [14pt]
\displaystyle \hspace{10pt} +\int_{\Omega_2}\left((\beta_+)_k(u_{n,k})+(\beta_-)_n(u_{n,k})\right)\left(\frac{1}{\nu(\Omega_1)}\int_{\Omega_1}u_{n,k}d\nu-\frac{1}{\nu(\Omega_2)}\int_{\Omega_2}u_{r,n,k}d\nu\right) d\nu \\ [14pt]
\displaystyle  \hspace{10pt} + \int_{\Omega_2}\left(\frac{1}{n}|u_{n,k}(x)|^{p-2}u_{n,k}^+(x) -\frac{1}{k}|u_{n,k}(x)|^{p-2}u_{n,k}^-(x)\right)\\[12pt]
\displaystyle \hspace{80pt} \times\left(\frac{1}{\nu(\Omega_1)}\int_{\Omega_1}u_{n,k}d\nu-\frac{1}{\nu(\Omega_2)}\int_{\Omega_2}u_{n,k}d\nu\right) d\nu .
\end{array}$$

Now, by H\"older's inequality and the generalised Poincar\'e type inequality with   $l=\nu(\Omega_1)$  (let $\Lambda_1$ denote the constant appearing in the generalised Poincar\'{e} type inequality in Assumption \ref{assumption3}),
$$\begin{array}{rl}
\displaystyle\int_{\Omega}\varphi_{n,k}  \left(u_{n,k}-\frac{1}{\nu(\Omega_1)}\int_{\Omega_1}u_{n,k}d\nu\right)d\nu
&\displaystyle\le\Vert\varphi\Vert_{L^{p'}(\Omega,\nu)}\left\Vert u_{n,k}-\frac{1}{\nu(\Omega_1)}\int_{\Omega_1}u_{n,k}d\nu\right\Vert_{L^p(\Omega,\nu)} \\ [14pt]
 &\displaystyle\le\Lambda_1\Vert\varphi\Vert_{L^{p'}(\Omega,\nu)} \left(\int_{\Omega}\int_{\Omega}|u_{n,k}(y)-u_{n,k}(x)|^p dm_x(y)d\nu(x)\right)^{\frac1p},
\end{array}$$
and, by \eqref{cota1}, \eqref{cota2}   and the generalised Poincar\'e type inequality with $l=\nu(\Omega_1)$ and with $l=\nu(\Omega_2)$   (let $\Lambda_2$ denote the constant appearing in the Poincar\'{e} type inequality for the latter case), we obtain
$$\begin{array}{l}
\displaystyle\int_{\Omega_2}\left(\left((\beta_+)_k(u_{n,k})+(\beta_-)_n(u_{n,k})\right)+\frac{1}{n}|u_{n,k}(x)|^{p-2}u_{n,k}^+(x) -\frac{1}{k}|u_{n,k}(x)|^{p-2}u_{n,k}^-(x)\right) \\ [12pt]
\displaystyle  \hspace{200pt} \times \left(\frac{1}{\nu(\Omega_1)}\int_{\Omega_1}u_{n,k}d\nu-\frac{1}{\nu(\Omega_2)}\int_{\Omega_2}u_{n,k}d\nu\right) d\nu\\ [14pt]
\displaystyle  \le \Vert\varphi\Vert_{L^{1}(\Omega,\nu)} \left|\frac{1}{\nu(\Omega_1)}\int_{\Omega_1}u_{n,k}d\nu-\frac{1}{\nu(\Omega_2)}\int_{\Omega_2}u_{n,k}d\nu\right|
\\ [14pt]
\displaystyle  \le \Vert\varphi\Vert_{L^{1}(\Omega,\nu)}\frac{1}{\nu(\Omega)^{\frac1p} } \left(\left\Vert u_{n,k}-\frac{1}{\nu(\Omega_1)}\int_{\Omega_1}u_{n,k}d\nu\right\Vert_{L^p(\Omega,\nu)}+\left\Vert u_{n,k}-\frac{1}{\nu(\Omega_2)}\int_{\Omega_2}u_{n,k}d\nu\right\Vert_{L^p(\Omega,\nu)}\right)\\ [14pt]
\displaystyle  \le \Vert\varphi\Vert_{L^{1}(\Omega,\nu)} \frac{\Lambda_1+\Lambda_2}{\nu(\Omega)^{\frac1p} }\left(\int_{\Omega}\int_{\Omega}|u_{n,k}(y)-u_{n,k}(x)|^p dm_x(y)d\nu(x)\right)^{\frac1p}.
\end{array}$$
  Therefore, by \eqref{first233} and the subsequent equations,
\begin{equation}\label{Lpvariationbound}
\begin{array}{l}
\displaystyle   \frac{c_p}{2} \left(\int_{\Omega}\int_{\Omega}|u_{n,k}(y)-u_{n,k}(x)|^p dm_x(y)d\nu(x)\right)^{\frac{1}{p'}}
\displaystyle \le \Lambda_1\Vert\varphi\Vert_{L^{p'}(\Omega,\nu)}+ \frac{\Lambda_1+\Lambda_2}{\nu(\Omega)^{\frac1p} }\Vert\varphi\Vert_{L^{1}(\Omega,\nu)}.
\end{array}
\end{equation}

\subsubsection{Existence of solutions of $(GP_\varphi)$}\label{existencesection}

Observe that a solution $(u,v)$ of~$(GP_\varphi)$ satisfies
$$\int_{\Omega_1}vd\nu+\int_{\Omega_2}vd\nu=\int_\Omega\varphi,$$
therefore, since $v\in\gamma(u)$ in $\Omega_1 $
and $v\in\beta(u)$ in $\Omega_2$,   we   need  $\varphi$ to satisfy
$$
 \mathcal{R}_{\gamma,\beta}^-\le\int_\Omega\varphi d\nu\le\mathcal{R}_{\gamma,\beta}^+.
$$
  We  prove the existence of solutions when the inequalities in the previous equation are strict. This suffices for what we need in the next section. Recall that we are working under the Assumptions~\ref{assumption1} to~\ref{assumption3}.

\begin{theorem}\label{existenceeli01}
Given  $\varphi\in L^{p'}(\Omega,\nu)$ such that
$$
 \mathcal{R}_{\gamma,\beta}^-<\int_\Omega\varphi d\nu<\mathcal{R}_{\gamma,\beta}^+,
$$
Problem $(GP_\varphi)$ stated in~\eqref{02091131general}  has a  solution.
\end{theorem}

Observe then that any solution $(u,v)$ of $(GP_\varphi)$ under such assumptions also satisfies
$$
 \mathcal{R}_{\gamma,\beta}^-<\int_\Omega vd\nu<\mathcal{R}_{\gamma,\beta}^+.
$$
 This will be used later on.

We divide the proof into three cases.

\begin{proof}[Proof of Theorem~\ref{existenceeli01} when $\mathcal{R}_{\gamma,\beta}^\pm=\pm\infty$]
Suppose that
$$
\mathcal{R}_{\gamma,\beta}^-= -\infty \ \hbox{ and } \mathcal{R}_{\gamma,\beta}^+= +\infty.
$$
  Let $\varphi \in L^{p'}(\Omega,\nu)$, $\varphi_{n,k}$ defined as in \eqref{varphink} and let $u_{n,k}\in L^\infty(\Omega,\nu)$, $n, k\in\N$, be solutions of the {\it Approximate Problem} \eqref{EE1nk}--\eqref{EE2nk}.

\noindent {\it Step A (Boundedness)}.   Let us first see that $\{\Vert u_{n,k}\Vert_{L^p(\Omega,\nu)}\}_{n,k}$ is bounded.

\noindent {\it   Step  1.}  We start by proving that $\{\Vert u_{n,k}^+\Vert_{L^p(\Omega,\nu)}\}_{n,k}$ is bounded.  We  see this case by case.  Since $\mathcal{R}_{\gamma,\beta}^+=+\infty$, then $\sup \mbox{Ran}(\gamma)=+\infty$ or $\sup \mbox{Ran}(\beta)=+\infty$.

\noindent {\it Case 1.1.} Suppose that $\sup \mbox{Ran}(\gamma)=+\infty$. Then, by \eqref{cota1},
$$\int_{\Omega_1}(\gamma_+)_k(u_{n,k})d\nu \le M:=\int_\Omega \varphi d\nu \quad \hbox{for every } n, k\in\N.$$
Let $z^+_{n,k}:= (\gamma_+)_k(u_{n,k})$ and $\widetilde{\Omega}_{n,k}:=\left\{x\in \Omega_1 : z_{n,k}^+(x)<\frac{2M}{\nu(\Omega_1)}\right\}$. Then
$$0\le \int_{\widetilde{\Omega}_{n,k}}z_{n,k}^+d\nu =\int_{\Omega_1}z_{n,k}^+d\nu-\int_{\Omega_1\setminus \widetilde{\Omega}_{n,k}}z_{n,k}^+d\nu\le M-(\nu(\Omega_1)-\nu(\widetilde{\Omega}_{n,k}))\frac{2M}{\nu(\Omega_1)}=
\nu(\widetilde{\Omega}_{n,k})\frac{2M}{\nu(\Omega_1)}-M,$$
from where
$$\nu(\widetilde{\Omega}_{n,k})\ge \frac{\nu(\Omega_1)}{2}.$$

\noindent {\it Case 1.1.1.} Assume that $\sup D(\gamma)=+\infty$. Let
$r_0\in\R$ be such that $\gamma^0(r_0)>2M/\nu(\Omega_1)$ and let $k_0\in\N$ such that
\begin{equation}\label{r0lemma1segunda}\frac{2M}{\nu(\Omega_1)}<(\gamma_+)_k(r_0)\le
\gamma^0(r_0)\quad \hbox{for $k\ge k_0$}.\end{equation}
Then, since in $\widetilde{\Omega}_{n,k}$  we have that $ (\gamma_+)_k(u_{n,k})=z^+_{n,k}<\frac{2M}{\nu(\Omega_1)}$, from \eqref{r0lemma1segunda} we get that
$$  u_{n,k}^+ \le  r_0\quad\hbox{in $\widetilde{\Omega}_{n,k}$}\quad \hbox{for every $k\ge k_0$ and every $n\in\N$}.$$
Therefore, this bound, the generalised Poincar\'{e} type inequality with $l=\frac{\nu(\Omega_1)}{2}$  and \eqref{Lpvariationbound} yield the boundedness of $\{\Vert u_{n,k}^+\Vert_{L^p(\Omega,\nu)}\}_{n,k}$.

\noindent {\it Case 1.1.2.} Suppose that  $r_\gamma:=\sup D(\gamma)<+\infty$ and let $h>0$. Since $r_\gamma+h\not\in D(\gamma)$,
$$(\gamma_+)_k(r_\gamma+h)\uparrow +\infty \ \hbox{ as $k\to +\infty$}.$$
Take $k_0\in\N$ such that $(\gamma_+)_k(r_\gamma+h)\ge \frac{2M}{\nu(\Omega_1)}$ for every $k\ge k_0$. Then,
$$(\gamma_+)_k(u_{n,k}^+)<\frac{2M}{\nu(\Omega_1)}\le (\gamma_+)_k(r_\gamma+h) \quad\hbox{in $\widetilde{\Omega}_{n,k}$}\quad \hbox{for every $k\ge k_0$ and every $n\in\N$},$$
thus
$$u_{n,k}^+\le r_\gamma+h \quad\hbox{in $\widetilde{\Omega}_{n,k}$}\quad \hbox{for every $k\ge k_0$ and every $n\in\N$}.$$
Therefore, again, this bound together with the generalised Poincar\'{e} type inequality with $l=\frac{\nu(\Omega_1)}{2}$    and \eqref{cota1} yield the boundedness of $\{\Vert u_{n,k}^+\Vert_{L^p(\Omega,\nu)}\}_{n,k}$.

\noindent {\it Case 1.2.} If $ \sup \mbox{Ran}(\beta)=+\infty$ we proceed similarly.

\noindent {\it   Step 2.}   Using that $\mathcal{R}_{\gamma,\beta}^-=-\infty$ we obtain that $\{\Vert u_{n,k}^-\Vert_{L^p(\Omega,\nu)}\}_{n,k}$ is bounded with an analogous argument.

Consequently, we get that $\{\Vert u_{n,k}\Vert_{L^p(\Omega,\nu)}\}_{n,k}$ is  bounded as desired.

\noindent {\it Step B (Taking limits in $n$)}.  The monotonicity properties obtained in Subsection~\ref{secmon}  together with the boundedness of $\{\Vert u_{n,k}\Vert_{L^p(\Omega,\nu)}\}_{n,k}$ allow us to apply the monotone convergence theorem to obtain $u_k\in L^p(\Omega,\nu)$, $k\in \N$, and $u\in L^p(\Omega,\nu)$ such that, taking a subsequence if necessary, $u_{n,k}\stackrel{n}{\rightarrow} u_k$ in $L^p(\Omega,\nu)$ and pointwise $\nu$-a.e. in $\Omega$ for $k\in\N$, and $u_{k}\stackrel{k}{\rightarrow} u$ in $L^p(\Omega,\nu)$ and pointwise $\nu$-a.e. in $\Omega$.

We now want to take limits, in $n$ and then in $k$, in \eqref{EE1nk} and \eqref{EE2nk}. Since $u_{n,k}\stackrel{n}{\rightarrow} u_k$ in $L^p(\Omega,\nu)$ and pointwise $\nu$-a.e. in $\Omega$,
\begin{equation}\label{apconvnk}
\int_{\Omega}\a_p(\cdot,y,u_{n,k}(y)-u_{n,k}(\cdot))dm_{(\cdot)}(y)\stackrel{n}{\longrightarrow}\int_{\Omega}\a_p(\cdot,y,u_{k}(y)-u_{k}(\cdot))dm_{(\cdot)}(y),
\end{equation}
$$\frac{1}{n}|u_{n,k}|^{p-2}u_{n,k}^+\stackrel{n}{\longrightarrow} 0$$
and
$$\frac{1}{k}|u_{n,k}|^{p-2}u_{n,k}^-\stackrel{n}{\longrightarrow} \frac{1}{k}|u_{k}|^{p-2}u_{k}^-$$
in $L^{p'}(\Omega,\nu)$ and,   up to a subsequence, for $\nu$-a.e. $x\in\Omega$.
Indeed, the second and third limits follow because $|u_{n,k}|^{p-2}u_{n,k}^+\stackrel{n}{\rightarrow}|u_{k}|^{p-2}u_{k}^+$ in $L^{p'}(\Omega,\nu)$. Now, since $\{u_{n,k}\}_n$ is nonincreasing in $n$,  $|u_{n,k}|\le \max\{|u_{1,k}|,|u_k|\}$ $\nu$-a.e. in $\Omega$, for every $n$, $k\in\N$, so Lemma \ref{convergenciaap} yields the convergence \eqref{apconvnk} in $L^{p'}(\Omega,\nu)$.

Now, isolating $(\gamma_+)_k(u_{n,k})+(\gamma_-)_n(u_{n,k})$ and $(\beta_+)_k(u_{n,k})+(\beta_-)_n(u_{n,k})$ in equations~\eqref{EE1nk} and~\eqref{EE2nk}, respectively, and taking the positive parts, we get that
  $$ \begin{array}{l}
  \displaystyle(\gamma_+)_k(u_{n,k}(x)) \\ [6pt]
 \displaystyle
=\left(\int_{\Omega}\a_p(x,y,u_{n,k}(y)-u_{n,k}(x))dm_x(y)
+\frac{1}{n}|u_{n,k}(x)|^{p-2}u_{n,k}^+(x) -\frac{1}{k}|u_{n,k}(x)|^{p-2}u_{n,k}^-(x)
+\varphi_{n,k}(x)\right)^+   \end{array}
$$
for $x\in\Omega_1$, and
$$ \begin{array}{l}
  \displaystyle(\beta_+)_k(u_{n,k}(x)) \\ [6pt]
 \displaystyle  = \left(\int_{\Omega}\a_p(x,y,u_{n,k}(y)-u_{n,k}(x))dm_x(y)+
\frac{1}{n}|u_{n,k}(x)|^{p-2}u_{n,k}^+(x)  -\frac{1}{k}|u_{n,k}(x)|^{p-2}u_{n,k}^-(x)+\varphi_{n,k}(x)\right)^+  \end{array}
$$
for $x\in\Omega_2$. Therefore, since the right hand sides of these equations converge in $L^{p'}(\Omega_1,\nu)$ and $L^{p'}(\Omega_2,\nu)$ (and also $\nu$-a.e. in $\Omega_1$ and $\Omega_2$), respectively,  there exist $z^+_k\in L^{p'}(\Omega_1,\nu)$ and $\omega^+_k\in L^{p'}(\Omega_2,\nu)$ such that  $(\gamma_+)_k(u_{n,k})\stackrel{n}{\rightarrow}z^+_k$ in $L^{p'}(\Omega_1,\nu)$ and pointwise $\nu$-a.e. in $\Omega_1$, and $(\beta_+)_k(u_{n,k})\stackrel{n}{\rightarrow}\omega^+_k$ in $L^{p'}(\Omega_2,\nu)$ and pointwise $\nu$-a.e. in $\Omega_2$. Moreover, since  $(\gamma_+)_k$ and $(\beta_+)_k$ are maximal monotone graphs, $z^+_k= (\gamma_+)_k(u_{k})$ $\nu$-a.e. in $\Omega_1$, and $\omega^+_k= (\beta_+)_k(u_{k})$ $\nu$-a.e. in $\Omega_2$.

  Similarly, taking the negative parts, there exist
$$  \lim_{n\to +\infty}(\gamma_-)_n(u_{n,k}(x))=z_k^-(x)\ \hbox{
in $L^{p'}(\Omega_1,\nu)$ and for $\nu$-a.e. $x\in\Omega_1$,}
$$
 and
$$ \lim_{n\to +\infty}(\beta_-)_n(u_{n,k}(x))=\omega_k^-(x)
\ \hbox{
in $L^{p'}(\Omega_2,\nu)$ and for $\nu$-a.e. $x\in\Omega_2$}.
$$
Moreover, by \cite[Lemma G]{BCrS}, $z^-_k\in \gamma_-(u_{k})$ and $\omega^-_k\in \beta_-(u_{k})$.
Therefore, we have obtained that
\begin{equation}\label{EE1onlyn} \begin{array}{l}\displaystyle z_k^+(x)+z^-_k(x)-\int_{\Omega}\a_p(x,y,u_{k}(y)-u_{k}(x))dm_x(y) -\frac{1}{k}|u_{k}(x)|^{p-2}u_{k}^-(x) = \varphi_{k}(x), \end{array}\end{equation}
for $\nu$-a.e. $x\in\Omega_1$, and
\begin{equation}\label{EE2onlyn} \begin{array}{l}\displaystyle \omega_k^+(x)+\omega_k^-(x)-\int_{\Omega}\a_p(x,y,u_{k}(y)-u_{k}(x))dm_x(y) -\frac{1}{k}|u_{k}(x)|^{p-2} u_{k}^-(x) =\varphi_{k}(x)\end{array} \end{equation}
for $\nu$-a.e. $x\in\Omega_2$.

\noindent {\it Step C (Taking limits in $k$)}.   Now again, isolating $z_k^++z^-_k$ and $\omega_k^++\omega_k^-$ in equations \eqref{EE1onlyn} and \eqref{EE2onlyn}, respectively, and taking the positive and negative parts as above, we get that there exist $z^+\in L^{p'}(\Omega_1,\nu)$, $z^-\in L^{p'}(\Omega_1,\nu)$, $\omega^+\in L^{p'}(\Omega_2,\nu)$ and $\omega^-\in L^{p'}(\Omega_2,\nu)$ such that $z^+_k \stackrel{k}{\rightarrow}z^+$ and $z^-_k \stackrel{k}{\rightarrow}z^-$ in $L^{p'}(\Omega_1,\nu)$ and pointwise $\nu$-a.e. in $\Omega_1$, and $\omega^+_k \stackrel{k}{\rightarrow}\omega^+$ and $\omega^-_k \stackrel{k}{\rightarrow}\omega^-$ in $L^{p'}(\Omega_2,\nu)$ and pointwise $\nu$-a.e. in $\Omega_2$. In addition,  by  the maximal monotonicity of $\gamma_-$ and $\beta_-$, $z^-\in \gamma_-(u)$ and $\omega^-\in \beta_-(u)$ $\nu$-a.e. in $\Omega_1$ and $\Omega_2$, respectively. Moreover,  by \cite[Lemma G]{BCrS}, $z^+\in \gamma_+(u)$ and $\omega^+\in \beta_+(u)$ $\nu$-a.e. in $\Omega_1$ and $\Omega_2$, respectively.

Consequently,
$$ z(x)-\int_{\Omega}\a_p(x,y,u(y)-u(x))dm_x(y)=\varphi(x)
\ \hbox{
for $\nu$-a.e. $x\in\Omega_1$,}$$ and
$$ \omega(x)-\int_{\Omega}\a_p(x,y,u(y)-u(x))dm_x(y)=\varphi(x)
\ \hbox{
for $\nu$-a.e. $x\in\Omega_2$,}$$ where $z=z^++z^-  \in \gamma(u)$ $\nu$-a.e. in $\Omega_1$ and $\omega=\omega^++\omega^-\in\beta(u)$ $\nu$-a.e. in $\Omega_2$. The proof of existence in this case is done. \qed
\end{proof}

\begin{proof}[Proof of Theorem~\ref{existenceeli01} when $\mathcal{R}_{\gamma,\beta}^\pm$ are finite]
Suppose that
$$
-\infty<\mathcal{R}_{\gamma,\beta}^-<\mathcal{R}_{\gamma,\beta}^+< +\infty .
$$
Let $\varphi\in L^{p'}(\Omega,\nu)$, and assume that it satisfies $$\mathcal{R}_{\gamma,\beta}^-<\int_{\Omega}\varphi d\nu < \mathcal{R}_{\gamma,\beta}^+.$$
 Then, for $\varphi_{n,k}$ defined as in \eqref{varphink}, there exist $M_1, M_2\in\R$ and $n_0, k_0\in\N$ such that
\begin{equation}\label{varphinkM1M2}
\mathcal{R}_{\gamma,\beta}^-<M_2<\int_{\Omega}\varphi_{n,k}d\nu < M_1<\mathcal{R}_{\gamma,\beta}^+
\end{equation}
for every $n\ge n_0$ and $k\ge k_0$.  For $n, k\in\N$ let $u_{n,k}\in L^\infty(\Omega,\nu)$ be the solution of the {\it Approximate Problem} \eqref{EE1nk}--\eqref{EE2nk},  and let
\begin{equation}\label{M3}
M_3:=\sup_{n,k\in\N} \left\Vert u_{n,k} - \frac{1}{\nu(\Omega_1)}\int_{\Omega_1}u_{n,k}d\nu \right\Vert_{L^p(\Omega,\nu)}<+\infty.
\end{equation}
Observe that $M_3$ is finite by the generalised Poincar\'e type inequality together with \eqref{Lpvariationbound}. Let $k_1\in\N$ such that $k_1\ge k_0$ and $M_1+\frac{1}{k}M_3\nu(\Omega)^{\frac{1}{p(p-1)}}<\mathcal{R}_{\gamma,\beta}^+$ for every $k\ge k_1$.

\noindent   {\it Step D (Boundedness of $\{\Vert u_{n,k}\Vert_{L^p(\Omega,\nu)}\}_n$  and passing to the limit in $n$)} Let us see that, for each $k\in \mathbb{N}$, $\{\Vert u_{n,k}\Vert_{L^p(\Omega,\nu)}\}_n$ is bounded.  Fix $k\ge k_1$ and suppose that $\{\Vert u_{n,k}\Vert_{L^p(\Omega,\nu)}\}_{n}$ is not bounded. Then, by \eqref{M3}, since $u_{n,k}$ is nondecreasing in $n$,
$$\frac{1}{\nu(\Omega_1)}\int_{\Omega_1}u_{n,k}d\nu \stackrel{n\to +\infty}{\longrightarrow} +\infty .$$
Thus, using again that $u_{n,k}$ is nondecreasing in $n$, there exists $n_1\ge n_0$ such that
$$\begin{array}{rl}
\displaystyle u_{n,k}^-&\displaystyle\le \left(u_{n,k}-\frac{1}{\nu(\Omega_1)}\int_{\Omega_1}u_{n,k}d\nu\right)^-+\left( \frac{1}{\nu(\Omega_1)}\int_{\Omega_1}u_{n,k}d\nu \right)^-\\ [14pt]
&\displaystyle=\left(u_{n,k}-\frac{1}{\nu(\Omega_1)}\int_{\Omega_1}u_{n,k}d\nu\right)^-
\end{array}$$
for every $n\ge n_1$, and thus
$$
\Vert u_{n,k}^-\Vert_{L^p(\Omega,\nu)}\le M_3 \ \hbox{ for every $n\ge n_1$}.
$$
Consequently, $\Vert u_{n,k}^-\Vert_{L^{p-1}(\Omega,\nu)}\le M_3\nu(\Omega)^{\frac{1}{p(p-1)}}$ for $n\ge n_1$. Then, with this bound and \eqref{varphinkM1M2} at hand, integrating \eqref{EE1nk} and \eqref{EE2nk} with respect to $\nu$ over $\Omega_1$ and $\Omega_2$, respectively, adding both equations and neglecting some nonnegative terms we get
$$
\begin{array}{l}
\displaystyle \int_{\Omega_1}\underbrace{(\gamma_+)_k(u_{n,k}(x))+(\gamma_-)_n(u_{n,k}(x))}_{z_{n,k}(x)}d\nu(x)+\int_{\Omega_2}\underbrace{(\beta_+)_k(u_{n,k}(x))+(\beta_-)_n(u_{n,k}(x))}_{\omega_{n,k}(x)}d\nu(x)  \\ [14pt]
\displaystyle \ \le \underbrace{M_1 +\frac{1}{k}M_3\nu(\Omega)^{\frac{1}{p(p-1)}}}_{M_4}< \mathcal{R}_{\gamma,\beta}^+ .
\end{array}
$$
Therefore, for each $n\in\N$, either
\begin{equation}\label{option1rep}
\int_{\Omega_1}z_{n,k}d\nu<\nu(\Omega_1)\sup \mbox{Ran}(\gamma)-\frac{\delta}{2}
\end{equation}
or
\begin{equation}\label{option2rep}
\int_{\Omega_2}\omega_{n,k} d\nu<\nu(\Omega_2)\sup \mbox{Ran}(\beta)-\frac{\delta}{2},
\end{equation}
where $\delta:=\mathcal{R}_{\gamma,\beta}^+- M_4>0$.

For $n\in\N$ such that \eqref{option1rep} holds let $K_{n,k}:=\left\{x\in\Omega_1 \, : \, z_{n,k}(x)<\sup \mbox{Ran}(\gamma)-\frac{\delta}{4\nu(\Omega_1)}\right\}$. Then
$$\int_{K_{n,k}}z_{n,k}d\nu=\int_{\Omega_1}z_{n,k}d\nu-\int_{\Omega_1\setminus K_{n,k}}z_{n,k}d\nu<-\frac{\delta}{4}+\nu(K_{n,k})\left(\sup \mbox{Ran}(\gamma)-\frac{\delta}{4\nu(\Omega_1)}\right),$$
and
$$\int_{K_{n,k}}z_{n,k}d\nu\ge \nu(K_{n,k})\inf \mbox{Ran}(\gamma).$$
Therefore,
$$\nu(K_{n,k})\left(\sup \mbox{Ran}(\gamma)-\inf \mbox{Ran}(\gamma)-\frac{\delta}{4\nu(\Omega_1)}\right)\ge \frac{\delta}{4},$$
thus $\nu(K_{n,k})>0$, $\displaystyle \sup \mbox{Ran}(\gamma)-\inf \mbox{Ran}(\gamma)-\frac{\delta}{4\nu(\Omega_1)}>0$ and
$$\nu(K_{n,k})\ge \frac{\delta/4}{\sup \mbox{Ran}(\gamma)-\inf \mbox{Ran}(\gamma)-\frac{\delta}{4\nu(\Omega_1)}}.$$
 Note that, if $\sup \mbox{Ran}(\gamma)-\frac{\delta}{4\nu(\Omega_1)}\le 0$ then $z_{n,k}\le 0$ in $K_{n,k}$, thus $u_{n,k}^+=0$ in $K_{n,k}$ and, consequently, $\Vert u_{n,k}^+\Vert_{L^p(K_{n,k},\nu)}=0$. Therefore, by the generalised Poincar\'{e} type inequality and \eqref{Lpvariationbound} we get that $\{\Vert u_{n,k}\Vert_{L^p(\Omega,\nu)}\}_{n}$ is bounded, which is a contradiction. We may therefore suppose that $\sup \mbox{Ran}(\gamma) -\frac{\delta}{4\nu(\Omega_1)} > 0$. Then, for $k_2\ge k_1$ large enough so that $\sup \mbox{Ran}((\gamma_+)_{k})>\sup \mbox{Ran}(\gamma)-\frac{\delta}{4\nu(\Omega_1)}$ for $k\ge k_2$,
$$\Vert u_{n,k}^+\Vert_{L^p(K_{n,k},\nu)}\le \nu(K_{n,k})^{\frac1p}(\gamma_+)_{k}^{-1}\left(\sup \mbox{Ran}(\gamma)-\frac{\delta}{4\nu(\Omega_1)} \right)$$
and by the generalised Poincar\'{e} type inequality and \eqref{Lpvariationbound} we get that $\{\Vert u_{n,k}\Vert_{L^p(\Omega,\nu)}\}_{n}$ is bounded, which is a contradiction.
Similarly for $n\in\N$ such that \eqref{option2rep} holds.

We have obtained that $\{\Vert u_{n,k}\Vert_{L^p(\Omega,\nu)}\}_{n}$ is bounded for each $k\in\N$. Therefore, since $\{u_{n,k}\}_n$ is nondecreasing in $n$, we may apply the monotone convergence theorem to obtain $u_k\in L^p(\Omega,\nu)$, $k\in \N$, such that $u_{n,k}\stackrel{n}{\rightarrow} u_k$ in $L^p(\Omega,\nu)$ and pointwise $\nu$-a.e. in $\Omega$ for $k\in\N$.
  Proceeding now  like in {\it Step B} of the previous proof we get:
  $z_k^+\in L^{p'}(\Omega_1,\nu)$ and $\omega_k^+\in L^{p'}(\Omega_2,\nu)$  such that $z^+_k\in \gamma_+(u_{k})$ and $\omega^+_k\in \beta_+(u_{k})$ $\nu$-a.e. in $\Omega_1$ and $\Omega_2$, respectively; and  $z_k^-\in L^{p'}(\Omega_1,\nu)$ and $\omega_k^-\in L^{p'}(\Omega_2,\nu)$ with $z^-_k\in \gamma_-(u_{k})$ and $\omega^-_k\in \beta_-(u_{k})$, $\nu$-a.e. $\Omega_1$ and $\Omega_2$, respectively, and such that
\begin{equation}\label{EE1onlyk} \begin{array}{l}\displaystyle z_k^+(x)+z^-_k(x)-\int_{\Omega}\a_p(x,y,u_{k}(y)-u_{k}(x))dm_x(y) -\frac{1}{k}|u_{k}(x)|^{p-2}u_{k}^-(x) = \varphi_{k}(x), \end{array}\end{equation}
for $\nu$-a.e. every $x\in\Omega_1$, and
\begin{equation}\label{EE2onlyk} \begin{array}{l}\displaystyle \omega_k^+(x)+\omega_k^-(x)-\int_{\Omega}\a_p(x,y,u_{k}(y)-u_{k}(x))dm_x(y) -\frac{1}{k}|u_{k}(x)|^{p-2}u_{k}^-(x) =\varphi_{k}(x)\end{array} \end{equation}
for $\nu$-a.e. every $x\in\Omega_2$.

\noindent   {\it Step E} (Boundedness of $\{\Vert u_{k}\Vert_{L^p(\Omega,\nu)}\}_k$ and passing to the limit in $k$) We now see that $\{\Vert u_{k}\Vert_{L^p(\Omega,\nu)}\}_{k}$ is bounded.  Since $u_k^+\le u_1^+$, it is enough to see that $\{\Vert u_{k}^-\Vert_{L^p(\Omega,\nu)}\}_{k}$ is bounded.

Now, \eqref{EE1onlyk} and \eqref{EE2onlyk} yield
$$
\begin{array}{l}
\displaystyle \int_{\Omega_1}\underbrace{z_k^+(x)+z^-_k(x)}_{z_{k}(x)}d\nu(x)+\int_{\Omega_2}\underbrace{\omega_k^+(x)+\omega_k^-(x)}_{\omega_{k}(x)}d\nu(x)  \ge M_2 > \mathcal{R}_{\gamma,\beta}^- .
\end{array}
$$
Therefore, for each $k\in\N$, either
\begin{equation}\label{option11}
\int_{\Omega_1}z_{k}d\nu>\nu(\Omega_1)\inf \mbox{Ran}(\gamma)+\frac{\delta'}{2}
\end{equation}
or
\begin{equation}\label{option22}
\int_{\Omega_2}\omega_{k} d\nu>\nu(\Omega_2)\inf \mbox{Ran}(\beta)+\frac{\delta'}{2},
\end{equation}
where $\delta':=M_2-\mathcal{R}_{\gamma,\beta}^->0$.

For $k\in\N$ such that \eqref{option11} holds let $K_{k}:=\{x\in\Omega_1 \, : \, z_{k}(x)>\inf \mbox{Ran}(\gamma) +\frac{\delta'}{4\nu(\Omega_1)}\}$. Then
 $$\begin{array}{rl}
\displaystyle\int_{K_{k}}z_{k}d\nu&\displaystyle=\int_{\Omega_1}z_{k}d\nu-\int_{\Omega_1\setminus K_{k}}z_{k}d\nu\\[14pt]
&\displaystyle > \left(\nu(\Omega_1)\inf \mbox{Ran}(\gamma) +\frac{\delta'}{2}\right)-(\nu(\Omega_1)-\nu( K_k))\left(\inf \mbox{Ran}(\gamma) +\frac{\delta'}{4\nu(\Omega_1)}\right)\\[14pt]
&\displaystyle=\frac{\delta'}{4}+\nu(K_{k})\left(\inf \mbox{Ran}(\gamma) +\frac{\delta'}{4\nu(\Omega_1)}\right),
\end{array}$$
and
$$\int_{K_{k}}z_{k}d\nu\le \nu(K_{k})\sup \mbox{Ran}(\gamma).$$
Therefore,
$$\nu(K_{k})\left(\sup \mbox{Ran}(\gamma) -\inf \mbox{Ran}(\gamma) -\frac{\delta'}{4\nu(\Omega_1)}\right)\ge \frac{\delta'}{4},$$
thus $\nu(K_{k})>0$, $\displaystyle \sup \mbox{Ran}(\gamma) -\inf \mbox{Ran}(\gamma) -\frac{\delta'}{4\nu(\Omega_1)}>0$ and
$$\nu(K_{k})\ge \frac{\delta'/4}{\sup \mbox{Ran}(\gamma) -\inf \mbox{Ran}(\gamma) -\frac{\delta'}{4\nu(\Omega_1)}}.$$
   Now, if $\inf \mbox{Ran}(\gamma) +\frac{\delta'}{4\nu(\Omega_1)}\ge 0$ then $z_{k}\ge 0$ in $K_{k}$, thus $u_{n,k}^-=0$ in $K_{k}$ and $\Vert u_{k}^-\Vert_{L^p(K_{k},\nu)}=0$; so by the generalised Poincar\'{e} type inequality and \eqref{Lpvariationbound} we get that $\{\Vert u_{k}\Vert_{L^p(\Omega,\nu)}\}_{n}$ is bounded.  If  $\inf \mbox{Ran}(\gamma)+\frac{\delta'}{4\nu(\Omega_1)} < 0$,
then
$$\Vert u_{k}^-\Vert_{L^p(K_{n,k},\nu)}\le- \nu(K_{k})^{\frac1p}\gamma_-^{-1}\left(\inf \mbox{Ran}(\gamma) +\frac{\delta'}{4\nu(\Omega_1)} \right)$$
  and by the generalised Poincar\'{e}   inequality and \eqref{Lpvariationbound} we get that $\{\Vert u_{k}\Vert_{L^p(\Omega,\nu)}\}_{k}$ is bounded.
Similarly for $k\in\N$ such that \eqref{option22} holds.

 Now, proceeding as in {\it Step C} of the previous proof, we finish this proof. \qed
\end{proof}

  Finally, we give the proof of the remaining case.

\begin{proof}[Proof of Theorem~\ref{existenceeli01} in the mixed case]
Let us see the existence for
\begin{equation}\label{varbetacond3}
-\infty<\mathcal{R}_{\gamma,\beta}^-<\mathcal{R}_{\gamma,\beta}^+= +\infty,
\end{equation}
or
\begin{equation}\label{varbetacond3r}
-\infty=\mathcal{R}_{\gamma,\beta}^-<\mathcal{R}_{\gamma,\beta}^+< +\infty.
\end{equation}
Suppose that~\eqref{varbetacond3} holds and let $\varphi\in L^{p'}(\Omega,\nu)$ satisfying $$\mathcal{R}_{\gamma,\beta}^-<\int_{\Omega}\varphi d\nu .$$
If~\eqref{varbetacond3r} holds and we have $\varphi\in L^{p'}(\Omega,\nu)$ satisfying $\displaystyle \int_{\Omega}\varphi d\nu< \mathcal{R}_{\gamma,\beta}^+$, the argument is analogous.

Let $\varphi_{n,k}$ be defined as in \eqref{varphink} and let $u_{n,k}\in L^\infty(\Omega,\nu)$, $n, k\in\N$, be the solution of the {\it Approximate Problem} \eqref{EE1nk}--\eqref{EE2nk}.
Then, by Lemma \ref{LemaAcotLp} together with \eqref{cota1},  $\{\Vert u_{n,k}^+\Vert_{L^p(\Omega,\nu)}\}_{n,k}$ is bounded. However, for a fixed $k\in\N$, since $u_{n,k}$ is nondecreasing in $n$,  $\{\Vert u_{n,k}^-\Vert_{L^p(\Omega,\nu)}\}_{n}$ is also bounded. Therefore, proceeding as in {\it Step B} of the first case, we obtain $u_k\in L^p(\Omega,\nu)$,   $z_k^+$, $z_k^-\in L^{p'}(\Omega_1,\nu)$ and $\omega_k^+$, $\omega_k^-\in L^{p'}(\Omega_2,\nu)$, $k\in\N$, such that
\begin{equation}\label{EE1onlykvar} \begin{array}{l}\displaystyle z_k^+(x)+z^-_k(x)-\int_{\Omega}\a_p(x,y,u_{k}(y)-u_{k}(x))dm_x(y) -\frac{1}{k}|u_{k}(x)|^{p-2}u_{k}^-(x) = \varphi_{k}(x), \end{array}\end{equation}
for $\nu$-a.e. $x\in\Omega_1$, and
\begin{equation}\label{EE2onlykvar} \begin{array}{l}\displaystyle \omega_k^+(x)+\omega_k^-(x)-\int_{\Omega}\a_p(x,y,u_{k}(y)-u_{k}(x))dm_x(y) -\frac{1}{k}|u_{k}(x)|^{p-2}u_{k}^-(x) =\varphi_{k}(x)\end{array} \end{equation}
for $\nu$-a.e. $x\in\Omega_2$; where, for $k\in\N$,
$$z^+_k= (\gamma_+)_k(u_{k}),\ z^-_k\in \gamma_-(u_{k}) \ \hbox{ $\nu$-a.e. in $\Omega_1$},$$
and
$$\omega^+_k= (\beta_+)_k(u_{k}),\ \omega^-_k\in \beta_-(u_{k}) \ \hbox{ $\nu$-a.e. in $\Omega_2$}.$$

  We  now prove that $\{\Vert u_{k}\Vert_{L^p(\Omega,\nu)}\}_{k}$ is bounded.
Proceeding as in {\it Step E} of the previous proof   and using the same notation, we get that for each $k\in\N$, either
\begin{equation}\label{option11var}
\int_{\Omega_1}z_{k}d\nu>\nu(\Omega_1)\inf \mbox{Ran}(\gamma) +\frac{\delta'}{2}
\end{equation}
or
\begin{equation}\label{option22var}
\int_{\Omega_2}\omega_{k} d\nu>\nu(\Omega_2)\inf \mbox{Ran}(\beta) +\frac{\delta'}{2}.
\end{equation}

\noindent {\it Case 1.}  For $k\in\N$ such that \eqref{option11var} holds, let  $$K_{k}:=\left\{x\in\Omega_1 \, : \, z_{k}(x)>\inf \mbox{Ran}(\gamma) +\frac{\delta'}{4\nu(\Omega_1)}\right\}.$$ Then,
\begin{equation}\label{zklowerbound}
\int_{K_{k}}z_{k}d\nu=\int_{\Omega_1}z_{k}d\nu-\int_{\Omega_1\setminus K_{k}}z_{k}d\nu>\frac{\delta'}{4}+\nu(K_{k})\left(\inf \mbox{Ran}(\gamma) +\frac{\delta'}{4\nu(\Omega_1)}\right).
\end{equation}

Now, by~\eqref{EE1onlykvar},
$$\begin{array}{c}\displaystyle
  \int_{\{x\in\Omega_1 \, : \,z_k(x)>h\}} z_k d\nu\le  \int_{\{x\in\Omega_1 \, : \,z_k(x)>h\}}\frac{|z_k|^{p'}}{h^{p'-1}}d\nu\\ \\
\displaystyle  =  \frac{1}{h^{p'-1}}\int_{\{x\in\Omega_1 \, : \,z_k(x)>h\}}\left|\int_{\Omega}a_p(x,y,u_k(y)-u_k(x)dm_x(y)+\varphi_k(x)\right|^{p'}d\nu(x).
\end{array}
  $$
  Thus, for a constant $D_1$ independent of $k$ and $h$,
  $$\begin{array}{c}\displaystyle
  \int_{\{x\in\Omega_1 \, : \,z_k(x)>h\}} z_k d\nu\le      \frac{D_1}{h^{p'-1}}\left(\int_{\Omega_1}\int_{\Omega}|a_p(x,y,u_k(y)-u_k(x)|^{p'}dm_x(y)d\nu(x)+\int_{\Omega_1}|\varphi_k|^{p'}d\nu\right).
\end{array}
  $$
Hence, by~\eqref{llo1} and~\eqref{Lpvariationbound}, there exist constants $D_2$ and $D_3$, independent of $k$ and $h$, such that
$$\begin{array}{c}\displaystyle
  \int_{\{x\in\Omega_1 \, : \,z_k(x)>h\}} z_k d\nu\le      \frac{D_2}{h^{p'-1}}\left(\int_{\Omega_1}\int_{\Omega}|u_k(y)-u_k(x)|^{p}dm_x(y)d\nu(x)+\int_{\Omega_1}|\varphi_k|^{p'}d\nu+1\right)\le      \frac{D_3}{h^{p'-1}}.
\end{array}
  $$
Consequently, we may find $h>0$ such that
$$\sup_{k\in \N}\int_{\{x\in\Omega_1 \, : \,z_k(x)>h\}} z_k d\nu< \frac{\delta'}{8}. $$
Therefore,
\begin{equation}\label{correccion002}
\int_{K_{k}}z_{k}d\nu=\int_{K_{k}\cap\{z_k>h\}}z_{k}d\nu
+ \int_{K_{k}\cap\{z_k\le h\}}z_{k}d\nu \le \frac{\delta'}{8}+\nu(K_k)h.
\end{equation}
Recalling \eqref{zklowerbound} we get
$$\frac{\delta'}{4}+\nu(K_{k})\left(\inf \mbox{Ran}(\gamma) +\frac{\delta'}{4\nu(\Omega_1)}\right)< \frac{\delta'}{8}+\nu(K_k)h,$$
thus
$$\frac{\delta'}{8}< \nu(K_{k})\left(h-\inf \mbox{Ran}(\gamma) -\frac{\delta'}{4\nu(\Omega_1)}\right). $$
Consequently, $\displaystyle h-\inf \mbox{Ran}(\gamma) -\frac{\delta'}{4\nu(\Omega_1)}>0$ and
$$\nu(K_{k})\ge \frac{\delta'/4}{h-\inf \mbox{Ran}(\gamma) -\frac{\delta'}{4\nu(\Omega_1)}}>0.
$$
From here we conclude as in the  previous proof.

\noindent {\it Case 2.}   For $k\in\N$ such that \eqref{option22var} holds, let $$\tilde K_{k}:=\{x\in\Omega_2 \, : \, w_{k}(x)>\inf \mbox{Ran}(\beta) +\frac{\delta'}{4\nu(\Omega_2)}\}$$
and proceed similarly. \qed
\end{proof}

\begin{remark}\label{BONUS} \item{(i)}   Taking limits in \eqref{Lpvariationbound} we obtain that, if $[u,v]$ is a solution of $(GP_\varphi^{\a_p,\gamma,\beta})$, then
 $$
\begin{array}{l}
\displaystyle \frac{c_p}{2}\left(\int_{\Omega}\int_{\Omega}|u(y)-u(x)|^p dm_x(y)d\nu(x)\right)^{\frac{1}{p'}}
\displaystyle \le \Lambda_1\Vert\varphi\Vert_{L^{p'}(\Omega,\nu)}+ \frac{\Lambda_1+\Lambda_2}{\nu(\Omega)^{\frac1p} }\Vert\varphi\Vert_{L^{1}(\Omega,\nu)},
\end{array}
$$
where $c_p$ is the constant in \eqref{llo2}, and $\Lambda_1$ and $\Lambda_2$ come from the   generalised Poincar\'e type inequality  and depend only on $p$, $\Omega_1$ and $\Omega_2$.

\item{(ii)} Observe that, on account of~\eqref{llo1} and the above estimate, we   have
$$
\begin{array}{l}
\displaystyle \left(\int_{\Omega}\left\vert\int_{\Omega}\a_p(x,y,u(y)-u(x))dm_x(y)\right|^{p'}d\nu(x)\right)^{\frac{1}{p'}}
\displaystyle \le C_p\nu(\Omega)+ \frac{2C_p}{c_p}(2\Lambda_1+\Lambda_2)\Vert\varphi\Vert_{L^{p'}(\Omega,\nu)}.
\end{array}
$$
Therefore,   since $[u,v]$ is a solution of $(GP_\varphi^{\a_p,\gamma,\beta})$,
$$
\begin{array}{l}
\displaystyle \Vert v\Vert_{L^{p'}(\Omega,\nu)} \le C_p\nu(\Omega)+\left(\frac{2C_p}{c_p}(2\Lambda_1+\Lambda_2)+1\right)\Vert\varphi\Vert_{L^{p'}(\Omega,\nu)}.
\end{array}
$$

\item{(iii)}     When $\varphi=0$ in $\Omega_2$,    we can easily get that   $v\ll \varphi$ in $\Omega_1.$
\end{remark}

\subsection{Other boundary conditions}\label{secdipierro001}

We can now ask for existence and uniqueness of solutions of the following problem (which was introduced in Section \ref{sub21})
 \begin{equation}\label{02091132}
\left\{ \begin{array}{ll} \gamma\big(u(x))-\hbox{div}_m\a_p u(x) \ni \varphi(x), &  x\in W, \\ [10pt]
\mathcal{N}^{\a_p}_2  u(x)+\beta\big(u(x)\big)\ni \varphi(x), & x\in\partial_mW, \end{array} \right.
\end{equation}
or, of the more general problem,
$$
 \left\{ \begin{array}{ll}\displaystyle \gamma\big(u(x))-   \int_{W\cup\Omega_2} \a_p(x,y,u(y)-u(x)) dm_x(y) \ni \varphi(x), &  x\in \Omega_1=W, \\ [14pt]
 \displaystyle \mathcal{N}^{\a_p}_2  u(x)+\beta\big(u(x)\big)\ni \varphi(x),  & x\in\Omega_2\subseteq\partial_mW. \end{array} \right.
$$
Recall that $\mathcal{N}^{\a_p}_2$ is defined as follows
$$
\mathcal{N}^{\a_p}_2 u(x):=  -\int_{W} \a_p(x,y,u(y)-u(x)) dm_x(y), \ \ x\in \partial_mW,
$$
which involves integration with respect to $\nu$ only over $W$, or more specifically over $\partial_m(X\setminus W)$.

For Problem~\eqref{02091132} we know that, in general, we do not have an appropriate Poincar\'e type inequality to work with (see Remark~\ref{comrem23}). Therefore, other techniques must be used to obtain the existence of solutions. In the particular case of $\gamma(r)=\beta(r)=r$ this was done in~\cite{MST4} by exploiting further monotonicity techniques.

However, if a generalised Poincar\'{e} type inequality (as defined in Definition~\ref{poincareatappendixdef}) is satisfied on $(A,B)=(\Omega_1, \Omega_2)$, we could solve the above problem by using the same techniques that we have used to solve Problem \eqref{02091131}. Indeed, we can work analogously but with the integration by parts formula given in Remark~\ref{remmonII} below. Note that this kind of Poincar\'e type inequality holds, for example, for finite graphs; even if $\Omega_2=\partial_m W$.

\begin{remark}\label{remmonII}
   Let $\Omega:=\Omega_1\cup \Omega_2$.
The following integration by parts formula holds:  Let $u$ be a measurable function such that
   $$[(x,y)\mapsto \a_p(x,y,u(y)-u(x))]\in L^{q}( \left(\Omega\times\Omega\right)\setminus \left(\Omega_2\times\Omega_2\right),\nu\otimes m_x)$$ and let $w \in L^{q'}(\Omega,\nu)$. Then
$$ \begin{array}{l}
 \displaystyle
-\int_{\Omega_1}\int_{\Omega}  \a_p(x,y,u(y)- u (x))dm_x(y)w(x)d\nu(x)\\ [12pt]
\displaystyle
\hspace{10pt} -\int_{\Omega_2}\int_{\Omega_1}  \a_p(x,y,u(y)- u (x))dm_x(y)w(x)d\nu(x)   \\ [14pt]
= \displaystyle\frac{1}{2} \int_{\left(\Omega\times\Omega\right)\setminus \left(\Omega_2\times\Omega_2\right)} \a_p(x,y,u(y)-u(x)) (w(y) - w(x)) d(\nu\otimes m_x)(x,y) .
\end{array}
$$

  \end{remark}

\begin{remark}
It is possible to consider this type of problems but with the random walk and the nonlocal Leray-Lions operator having a different behaviour on each subset $\Omega_i$, $i=1, 2$. For example, one could consider a problem, posed in $\Omega_1\cup\Omega_2\subset \mathbb{R}^N$, such as the following
$$
 \left\{ \begin{array}{ll}\displaystyle \gamma\big(u(x))- \int_{\Omega_1} \a_p^1(x,y,u(y)-u(x)) J_1(x-y)dy \\[10pt]
 \hspace{80pt} \displaystyle- \int_{\Omega_2} \a_p^3(x,y,u(y)-u(x)) J_3(x-y)dx\ni \varphi(x), &  x\in\Omega_1, \\ [14pt]
\displaystyle \beta\big(u(x)\big)- \int_{\Omega_1} \a_p^3(x,y,u(y)-u(x)) J_3(x-y)dy\\[10pt]
\hspace{80pt}\displaystyle -\int_{\Omega_2} \a_p^2(x,y,u(y)-u(x)) J_2(x-y)dx\ni\varphi(x), & x\in\Omega_2, \end{array} \right.
$$
where $J_i$ are kernels like the one in Example~\ref{ejem02}, and $\a_p^i$ are functions like the one in Subsection~\ref{sub21}, $i=1,2,3$. This could be done by obtaining a Poincar\'{e} type inequality involving $\frac{1}{\alpha_0}J_0$, where $J_0$ is the minimum of the previous three kernels and $\alpha_0=\int_{\mathbb{R}^N}J_0(z)dz$.
This idea has been used in~\cite{capannaetal} to study a homogenization problem.
\end{remark}

\section{Doubly nonlinear diffusion problems}\label{lasec3}

We  study two kinds of nonlocal $p$-Laplacian type diffusions problems. In one of them we cover nonlocal nonlinear diffusion problems with nonlinear dynamical boundary conditions and on the other we tackle nonlinear boundary conditions. We work under the Assumptions~\ref{assumption1} to~\ref{assumption3} used in Subsection~\ref{efzly}.

\subsection{Nonlinear dynamical boundary conditions}\label{subsect2}
Our aim in this subsection  is to study the following diffusion problem
\begin{equation}\label{sabore001bevol}
\left\{ \begin{array}{ll}
\displaystyle v_t(t,x)  -   \int_{\Omega} \a_p(x,y,u(t,y)-u(t,x)) dm_x(y)=f(t,x),    &x\in  \Omega_1,\ 0<t<T,
\\ [14pt]
 \displaystyle v(t,x)\in\gamma\big(u(t,x)\big), &
    x\in  \Omega_1,\ 0<t<T,
\\ [8pt]
\displaystyle w_t(t,x) -   \int_{\Omega} \a_p(x,y,u(t,y)-u(t,x)) dm_x(y)=g(t,x) ,   &x\in\Omega_2, \  0<t<T,
\\ [14pt]
 w(t,x) \in\beta\big(u(t,x)\big),    &x\in \Omega_2, \  0<t<T,
 \\ [10pt]
  v(0,x) = v_0(x),    &x\in \Omega_1,
 \\ [10pt]
 w(0,x) = w_0(x),    &x\in \Omega_2, \end{array} \right.
\end{equation}
of which Problem~\eqref{sabore001bevolparticularintro01} is a particular case and
which covers the case of dynamic evolution on the boundary $\partial_mW$ when $\beta\neq \mathbb{R}\times\{0\}$. This includes, in particular, for $\gamma=\mathbb{R}\times\{0\}$, the problem where the dynamic evolution occurs only on the boundary:
$$
\left\{ \begin{array}{ll}
\displaystyle  - \hbox{div}_m\a_p u(t,x)=f(t,x),    &x\in  W,\ 0<t<T,
\\ [10pt]
\displaystyle w_t(t,x)+\mathcal{N}^{\a_p}_\mathbf{1} u(t,x)=g(t,x) ,   &x\in\partial_m W, \  0<t<T,
\\ [10pt]
w(t,x) \in\beta\big(u(t,x)\big),    &x\in\partial_m W, \  0<t<T,
 \\ [10pt]
 w(0,x) = w_0(x),    &x\in \partial_m W. \end{array} \right.
$$
See~\cite{AIMTifb} for the reference local model.

Note that we may abbreviate Problem~\eqref{sabore001bevol} by using $v$ instead of $(v,w)$ and $f$ instead of $(f,g)$ as
 \begin{equation}\label{sabore001bevolshort}
\left\{ \begin{array}{ll}
\displaystyle v_t(t,x)  -   \int_{\Omega} \a_p(x,y,u(t,y)-u(t,x)) dm_x(y)=f(t,x),    &x\in  \Omega,\ 0<t<T,
\\ [14pt]
\displaystyle v(t,x)\in\gamma\big(u(t,x)\big), &
    x\in  \Omega_1,\ 0<t<T,
\\ [10pt]
v(t,x) \in\beta\big(u(t,x)\big),    &x\in \Omega_2, \  0<t<T,
 \\ [10pt]
 v(0,x) = v_0(x),    &x\in \Omega.\end{array} \right.
\end{equation}

To solve this problem we  use nonlinear semigroup theory. To this end we  introduce a multivalued operator  associated to Problem~\eqref{sabore001bevolshort} that allows us to rewrite it as an abstract Cauchy problem. Observe that this operator is defined on $L^1(\Omega,\nu) \equiv\left(L^1(\Omega_1,\nu)\times L^1(\Omega_2,\nu)\right).$

\begin{definition}  {\rm  We say that   $ (v, \hat v) \in \mathcal{B}^{m,\gamma,\beta}_{\a_p}$ if $v,\hat v \in L^1(\Omega,\nu)$,
  and there exists  $ u\in L^p(\Omega,\nu)$  with
  $$u\in \hbox{Dom}(\gamma)  \hbox{ and } v\in\gamma(u) \quad\hbox{$\nu$-a.e.  in } \Omega_1,$$
  and
  $$u\in \hbox{Dom}(\beta)  \hbox{ and }  v\in\beta(u) \quad\hbox{$\nu$-a.e.  in } \Omega_2,$$
  such that
$$
(x,y)\mapsto a_p(x,y,u(y)-u(x))\in L^{p'}(\Omega\times\Omega,\nu\otimes m_x)
$$
and
$$
 -   \int_{\Omega} \a_p(x,y,u(y)-u(x)) dm_x(y) = \hat v \quad \hbox{in} \ \ \Omega;
$$
that is, $[u,v]$ is a solution of   $(GP_{v+{\hat v}})$ (see~\eqref{02091131general} and Definition~\ref{defsol01}).
}
\end{definition}

On account of the results given in Subsection~\ref{efzly} (Theorems \ref{maxandcont01} and \ref{existenceeli01}) we have the following result. Recall that an operator $A$ in $L^1(\Omega,\nu)$ is {\it $T$-accretive} if
$$\Vert (u - \hat u)^+ \Vert_{L^1(\Omega,\nu)} \leq \Vert (u - \hat u + \lambda (v - \hat v))^+
\Vert_{L^1(\Omega,\nu)} \quad \hbox{for every } (u, v), (\hat u, \hat v) \in A \ \
\hbox{and} \ \lambda > 0.$$
In fact, $A$ is $T$-accretive if, and only if, its resolvents are contractions and order-preserving (see, for example, \cite[Appendix]{ElLibro} for further details).

  \begin{theorem}\label{rangecondition02}    The  operator $\mathcal{B}^{m,\gamma,\beta}_{\a_p}$ is   $T$-accretive  in $L^1(\Omega,\nu)$ and satisfies the range condition
$$\left\{\varphi\in L^{p'}(\Omega,\nu):\mathcal{R}_{\gamma,\beta}^-<\int_\Omega\varphi d\nu<\mathcal{R}_{\gamma,\beta}^+\right\} \subset R(I+ \lambda\mathcal{B}^{m,\gamma,\beta}_{\a_p})\quad\hbox{for every }\lambda>0.
$$
\end{theorem}

With respect to the domain of such operator we can prove the following result.

\begin{theorem}\label{remdom01bevol}  It holds that
 $$\overline{D(\mathcal{B}^{m,\gamma,\beta}_{\a_p})}^{L^{p'}(\Omega,\nu)}
 =\left\{v\in L^{p'}(\Omega,\nu) \, : \, \varGamma^-\le v\le \varGamma^+\ \hbox{$\nu$-\rm a.e. in }\Omega_1, \,
  \mathfrak{B}^-\le v\le \mathfrak{B}^+\ \hbox{$\nu$-\rm a.e. in }\Omega_2
 \right\}.$$
Therefore, we also have that
  $$\overline{D(\mathcal{B}^{m,\gamma,\beta}_{\a_p})}^{L^{1}(\Omega,\nu)}
 =\left\{v\in L^{1}(\Omega,\nu) \, : \, \varGamma^-\le v\le \varGamma^+\ \hbox{$\nu$-\rm a.e. in }\Omega_1, \,
  \mathfrak{B}^-\le v\le \mathfrak{B}^+\ \hbox{$\nu$-\rm a.e. in }\Omega_2
 \right\}.$$
\end{theorem}

\begin{proof}
It is obvious that
 $$\overline{D(\mathcal{B}^{m,\gamma,\beta}_{\a_p})}^{L^{p'}(\Omega,\nu)}
 \subset\left\{v\in L^{p'}(\Omega,\nu) \, : \, \varGamma^-\le v\le \varGamma^+\ \hbox{$\nu$-\rm a.e. in }\Omega_1, \,
  \mathfrak{B}^-\le v\le \mathfrak{B}^+\ \hbox{$\nu$-\rm a.e. in }\Omega_2
 \right\}.$$
  For the other inclusion it is enough to see that
$$\left\{v\in L^{\infty}(\Omega,\nu) \, : \, \varGamma^-\le v\le \varGamma^+\ \hbox{$\nu$-\rm a.e. in }\Omega_1, \,
  \mathfrak{B}^-\le v\le \mathfrak{B}^+\ \hbox{$\nu$-\rm a.e. in }\Omega_2
 \right\}\subset\overline{D(\mathcal{B}^{m,\gamma,\beta}_{\a_p})}^{L^{p'}(\Omega,\nu)}.
$$

Suppose first that $\gamma$ and $\beta$ satisfy
$$\begin{array}{ll}\varGamma^-<0, &\varGamma^+>0,\\[6pt]
\mathfrak{B}^-=0, & \mathfrak{B}^+>0.
\end{array}$$
It is enough to see that for any $v\in L^\infty(\Omega,\nu)$ such that there exist $m_1<0$, $\widetilde{m_i}\in \R$, $\widetilde{M_i}\in\R$, $M_i>0$, $i=1,2$, satisfying
$$\begin{array}{c}\varGamma^-<m_1< \widetilde{m_1}\le v\le \widetilde{M_1}< M_1< \varGamma^+\ \ \hbox{$\nu$-\rm a.e. in }\Omega_1, \\[6pt]
 0<\widetilde{m_2}\le v\le \widetilde{M_2}< M_2<\mathfrak{B}^+\ \ \hbox{$\nu$-\rm a.e. in }\Omega_2,
\end{array}
 $$
 it holds that $v\in \overline{D(\mathcal{B}^{m,\gamma,\beta}_{\a_p})}^{L^{p'}(\Omega,\nu)}$.

   By the results in Subsection \ref{existencesection} we know that, for $n\in\N$, there exists $u_n\in L^p(\Omega,\nu)$ and $v_n\in L^{p'}(\Omega,\nu)$ such that $[u_n,v_n]$ is a solution of $\displaystyle\left(GP_v^{\frac{1}{n}\a_p,\gamma,\beta}\right)$, i.e.,
$v_n\in \gamma(u_n)$ $\nu$-a.e. in $\Omega_1$, $v_n\in \beta(u_n)$ $\nu$-a.e. in $\Omega_2$ and
$$
 v_n(x) -\frac1n\int_{\Omega} \a_p(x,y,u_n(y)-u_n(x)) dm_x(y) = v(x) \quad \hbox{ for $\nu$-a.e. } x\in\Omega.
$$
In other words, $(v_n,n(v-v_n))\in \mathcal{B}^{m,\gamma,\beta}_{\a_p}$ or, equivalently,
$$v_n:=\left(I+\frac{1}{n}\mathcal{B}^{m,\gamma,\beta}_{\a_p}\right)^{-1}(v) \in  D(\mathcal{B}^{m,\gamma,\beta}_{\a_p}).$$
Let us see that  $v_n\stackrel{n}{\longrightarrow} v$ in $L^{p'}(\Omega,\nu)$.

 Let $a_{m_1}\le 0$ and $a_{M_1}\ge 0$ such that
 $$\hbox{$m_1\in\gamma(a_{m_1})\,$ and $\, M_1\in \gamma(a_{M_1})$,}$$
 and let $b_{M_2}\ge 0$ such that
 $$\hbox{$M_2\in \beta(b_{M_2})$.}$$
Set
$$\widehat{v}(x):=\left\{\begin{array}{cc}
             M_1,\, & x\in\Omega_1, \\
             M_2,\, & x\in\Omega_2,
           \end{array}\right.$$
   $$\widehat{u}(x):=\left\{\begin{array}{cc}
             a_{M_1},\, & x\in\Omega_1, \\
             b_{M_2},\, & x\in\Omega_2,
           \end{array}\right. $$
      and
      $$\widehat{\varphi}_n(x):=\left\{\begin{array}{cc}
             \displaystyle  M_1-\frac{1}{n}\int_{\Omega} \a_p(x,y,\widehat{u}(y)-\widehat{u}(x))dm_x(y),\, & x\in\Omega_1, \\ [12pt]
              \displaystyle M_2-\frac{1}{n}\int_{\Omega} \a_p(x,y,\widehat{u}(y)-\widehat{u}(x))dm_x(y),\, & x\in\Omega_2.
           \end{array}\right.$$
Then, $[\widehat{u},\widehat{v}]$ is a solution of $\displaystyle\left(GP_{\widehat{\varphi}_n}^{\frac{1}{n}\a_p,\gamma,\beta}\right)$.

Similarly, for
$$\widetilde{v}(x):=\left\{\begin{array}{cc}
             m_1, \, & x\in\Omega_1, \\
             0, \, & x\in\Omega_2,
           \end{array}\right.$$
              $$\widetilde{u}(x):=\left\{\begin{array}{cc}
             a_{m_1},\, & x\in\Omega_1, \\
             0,\, & x\in\Omega_2,
           \end{array}\right. $$
      and
      $$\widetilde{\varphi}_n(x):=\left\{\begin{array}{cc}
             \displaystyle  m_1-\frac{1}{n}\int_{\Omega_2} \a_p(x,y,-a_{m_1})dm_x(y), & x\in\Omega_1,\, \\ [12pt]
              \displaystyle \frac{1}{n}\int_{\Omega_1} \a_p(x,y,- a_{m_1})dm_x(y), & x\in\Omega_2,
           \end{array}\right.$$
 we have that $[\widetilde{u},\widetilde{v}]$ is a solution of $\displaystyle\left(GP_{\widetilde{\varphi}_n}^{\frac{1}{n}\a_p,\gamma,\beta}\right)$.

  Now, recalling~\eqref{llo1}, we have that there exists $n_0\in\N$ such that
  $$v\le\widetilde{M_1}\1_{\Omega_1} + \widetilde{M_2}\1_{\Omega_2}< \widehat{\varphi}_n \ \ \hbox{$\nu$-\rm a.e. in $\Omega$}$$
and
$$v\ge\widetilde{m_1}\1_{\Omega_1} + \widetilde{m_2}\1_{\Omega_2}> \widetilde{\varphi}_n \ \ \hbox{$\nu$-\rm a.e. in $\Omega$}$$
for $n\ge n_0$. Consequently, by the maximum principle (Theorem \ref{maxandcont01}) we obtain that
$$\widetilde{u}\le u_n\le\widehat{u}, $$
thus
$$\left\{\Vert u_n\Vert_{L^{\infty}(\Omega,\nu)}\right\}_{n} \hbox{ is bounded.}$$
Finally, since
$$
 v_n(x) - v(x) =  \frac1n\int_{\Omega} \a_p(x,y,u_n(y)-u_n(x)) dm_x(y)\quad \hbox{ $\nu$-a.e. in} \  \Omega,
$$
we conclude that, on account of~\eqref{llo1}, $$\hbox{$v_n\stackrel{n}{\longrightarrow} v$ in   $L^{p'}(\Omega,\nu)$.}$$

The other cases follow similarly, we  see two of them. Note that, since $\mathcal{R}_{\gamma,\beta}^-<\mathcal{R}_{\gamma,\beta}^+$, it is not possible to have $\gamma=\mathbb{R}\times\{0\}$ and $\beta=\mathbb{R}\times\{0\}$ simultaneously.   For example, suppose that we have
$$\begin{array}{ll}\varGamma^-=0,\, &\varGamma^+>0,\\[6pt]
\mathfrak{B}^-=0,\, & \mathfrak{B}^+>0.
\end{array}$$
We  use the same notation. Let $v\in L^\infty(\Omega,\nu)$ such that there exist $\widetilde{m_i}\in \R$, $\widetilde{M_i}\in\R$, $M_i>0$, $i=1,2$, satisfying
$$\begin{array}{l}0< \widetilde{m_1}\le v\le \widetilde{M_1}< M_1< \varGamma^+\ \hbox{ $\nu$-\rm a.e. in }\Omega_1, \\[6pt]
 0<\widetilde{m_2}\le v\le \widetilde{M_2}< M_2<\mathfrak{B}^+\ \hbox{ $\nu$-\rm a.e. in }\Omega_2.
\end{array}
 $$
 As before, the results in Subsection \ref{existencesection} ensure that there exist $u_n\in L^p(\Omega,\nu)$ and $v_n\in L^{p'}(\Omega,\nu)$, $n\in\N$, such that $[u_n,v_n]$ is a solution of $\displaystyle\left(GP_v^{\frac{1}{n}\a_p,\gamma,\beta}\right)$.
  Let $\ a_{M_1}\ge 0$ and $b_{M_2}\ge 0$ such that
 $$\hbox{$M_1\in \gamma(a_{M_1})\,$ and $\, M_2\in \beta(b_{M_2})$.}$$
Now again, let
$$\widehat{v}(x):=\left\{\begin{array}{cc}
             M_1,\, & x\in\Omega_1, \\
             M_2,\, & x\in\Omega_2,
           \end{array}\right.$$
   $$\widehat{u}(x):=\left\{\begin{array}{cc}
             a_{M_1},\, & x\in\Omega_1, \\
             b_{M_2},\, & x\in\Omega_2,
           \end{array}\right. $$
      and
      $$\widehat{\varphi}_n(x):=\left\{\begin{array}{cc}
             \displaystyle  M_1-\frac{1}{n}\int_{\Omega} \a_p(x,y,\widehat{u}(y)-\widehat{u}(x))dm_x(y),\, & x\in\Omega_1, \\ [12pt]
              \displaystyle M_2-\frac{1}{n}\int_{\Omega} \a_p(x,y,\widehat{u}(y)-\widehat{u}(x))dm_x(y),\, & x\in\Omega_2.
           \end{array}\right.$$
Then, as before, $[\widehat{u},\widehat{v}]$ is a solution of $\displaystyle\left(GP_{\widehat{\varphi}_n}^{\frac{1}{n}\a_p,\gamma,\beta}\right)$.

Now, taking $\widetilde{v}$, $\widetilde{u}$ and $\widetilde\varphi$ all equal to the null function in $\Omega$ and recalling that $\a_p(x,y,0)=0$ for every $x, y\in X$, we obviously have that $[\widetilde{u},\widetilde{v}]$ is a solution of $\displaystyle\left(GP_{0}^{\frac{1}{n}\a_p,\gamma,\beta}\right)$. Consequently, again by the second part of the maximum principle, we obtain, as desired, that $0\le u_n\le \widehat{v}$ for $n$ large enough.

  Finally, as a further example of a case which does not follow exactly with the same argument, suppose that $\gamma:=\R\times\{0\}$ and, for example,
$$\mathfrak{B}^-=0, \ \mathfrak{B}^+>0.$$
In this case we take $0\not\equiv v\in L^\infty(\Omega,\nu)$ such that $v=0$ in $\Omega_1$ and
$$0\le v <M_2 \ \hbox{ $\nu$-\rm a.e. in $\Omega_2$ \ for some constant $M_2>0$}.$$
As in the previous cases, there exist $u_n\in L^p(\Omega,\nu)$ and $v_n\in L^{p'}(\Omega,\nu)$, $n\in\N$, such that $[u_n,v_n]$ is a solution of $\displaystyle\left(GP_v^{\frac{1}{n}\a_p,\gamma,\beta}\right)$.
Let $b_{M_2}\ge 0 $ such that $M_2\in\beta(b_{M_2})$,
$$\widehat{v}(x):=\left\{\begin{array}{cc}
             0, & x\in\Omega_1, \\
             M_2,\, & x\in\Omega_2,
           \end{array}\right.$$
   $$\widehat{u}(x):=b_{M_2}, \ x\in\Omega, $$
      and
      $$\varphi_n(x):=\left\{\begin{array}{lc}
             \displaystyle  0, & x\in\Omega_1, \\
              \displaystyle M_2,\, & x\in\Omega_2.
           \end{array}\right.$$
Then, $[\widehat{u},\widehat{v}]$ is a solution of $\displaystyle\left(GP_{\varphi_n}^{\frac{1}{n}\a_p,\gamma,\beta}\right)$. Finally, take $\widetilde{v}$ and $\widetilde{u}$ again equal to the null function in $\Omega$ so that $[\widetilde{u},\widetilde{v}]$ is a solution of $\displaystyle\left(GP_{0}^{\frac{1}{n}\a_p,\gamma,\beta}\right)$. Consequently, for $n$ large enough, we get that $0\le u_n\le \widehat{v}$.
\qed
\end{proof}

In the next result we state the existence and uniqueness of solutions of Problem~\eqref{sabore001bevolshort}.

\begin{theorem} \label{nsth01bevol}
Let  $T>0$. For any
$v_0\in L^{1}(\Omega,\nu)$ and $f\in L^1(0,T;L^{1}(\Omega,\nu))$ such that
$$\varGamma^-\le v_0\le \varGamma^+\quad\hbox{$\nu$-\rm a.e. in }\Omega_1,$$
$$\mathfrak{B}^-\le v_0\le \mathfrak{B}^+\quad\hbox{$\nu$-\rm a.e. in }\Omega_2,$$
and
\begin{equation}\label{loquenec001}\mathcal{R}_{\gamma,\beta}^-<\int_{\Omega}v_0d\nu
+\int_0^t\int_\Omega fd\nu ds <\mathcal{R}_{\gamma,\beta}^+ \quad\hbox{for every } 0 \le t\le T,
\end{equation}
  there exists a unique mild-solution $v\in C([0,T];L^1(\Omega,\nu))$  of Problem~\eqref{sabore001bevolshort}.

  Let $v$ and $\widetilde v$ be the mild solutions of Problem~\eqref{sabore001bevolshort} with respective data $v_0,\ \widetilde v_0\in L^{1}(\Omega,\nu)$ and
   $f,\ \widetilde f\in L^1(0,T;L^{1}(\Omega,\nu))$. Then
   $$
   \begin{array}{rl}\displaystyle\int_\Omega \left(v(t,x)-\widetilde v(t,x)\right)^+d\nu(x)&\displaystyle\le
   \int_\Omega \left(v_0(x)-\widetilde v_0(x)\right)^+d\nu(x)\\[12pt]
   &\displaystyle
   \hspace{10pt} +\int_0^t\int_{\Omega}\left(f(s,x)-\widetilde f(s,x)\right)^+d\nu(x)ds \quad\hbox{for every } 0\le   t\le T.
   \end{array}$$

If, in addition to the previous assumptions on the data, we impose that
 \begin{equation}\label{conditiononjstar}
  v_0\in L^{p'}(\Omega,\nu),   \ f\in L^{p'}(0,T;L^{p'}(\Omega,\nu)) \
 \hbox{ and } \ \int_{\Omega_1}j_\gamma^*(v_0)d\nu+\int_{\Omega_2}j_\beta^*(v_0)d\nu<+\infty,
 \end{equation}
     then the mild solution  $v$ belongs to $W^{1,1}(0,T;L^{1}(\Omega,\nu))$ and satisfies
     $$
     \left\{
     \begin{array}{l}\partial_tv(t)+\mathcal{B}^{m,\gamma,\beta}_{\a_p}v(t)\ni f(t)\quad\hbox{for a.e. }t\in(0,T),\\[6pt]
     v(0)=v_0,\end{array}\right.$$
     that is, $v$ is  a strong solution.
\end{theorem}

\begin{proof}
We start by proving the existence of mild solutions. For $n\in\N$, consider the partition
$$t_0^n=0<t_1^n<\cdots <t_{n-1}^n<t_n^n=T$$
where $t_i^n:=iT/n$, $i=1,\ldots,n$. Given $\epsilon>0$,  there exists   $n\in \mathbb{N}$, $f_i^n\in L^{p'}(\Omega,\nu)$, $i=1,\ldots n$, and $v_0^n \in \overline{D(\mathcal{B}^{m,\gamma,\beta}_{\a_p})}^{L^{p'}(\Omega,\nu)}$ (i.e., $v_0^n\in L^{p'}(\Omega,\nu)$ satisfying $\varGamma^-\le v_0^n\le \varGamma^+$ $\nu$-a.e. in $\Omega_1$, and
  $\mathfrak{B}^-\le v_0^n\le \mathfrak{B}^+$ $\nu$-a.e. in $\Omega_2$)  such that   $T/n\le \epsilon$,
\begin{equation}\label{laap01} \sum_{i=1}^n\int_{t_{i-1}^n}^{t_i^n}\Vert f(t)- f_i^n\Vert_{L^{1}(\Omega,\nu)}dt  \le   \epsilon
\end{equation}
and
\begin{equation}\label{laap02}\Vert v_0-v_0^n\Vert_{L^1(\Omega,\nu)} \le  \epsilon.
\end{equation}
Then, setting
$$f_n(t):=f_i^n  \ \hbox{ for $t\in]t_{i-1}^n,t_i^n]$, $i=1,\ldots, n$},
$$ we have that
$$ \int_0^T \Vert f(t)- f_n(t) \Vert_{L^1(\Omega,\nu)}dt \le \epsilon.$$
By the results in Subsection \ref{existencesection} we  see that, for $n$ large enough, we may recursively find a solution $[u_i^n,v_i^n]$ of  $\displaystyle \left(GP^{\frac{T}{n}\a_p,\gamma,\beta}_{\frac{T}{n} f_i^n+v_{i-1}^n}\right)$, $i=1,\ldots,n$, in other words,
$$
v_i^n(x)-\frac{T}{n}\int_\Omega
\a_p(x,y,u_i^n(y)-u_i^n(x))dm_x(y)=\frac{T}{n} f_i^n(x)+v_{i-1}^n(x), \ \, x\in\Omega,
$$
or, equivalently,
\begin{equation}\label{solutiondiscretization2}
\frac{v_i^n(x)-v_{i-1}^n(x)}{T/n}-\int_\Omega
\a_p(x,y,u_i^n(y)-u_i^n(x))dm_x(y)= f_i^n(x), \ \, x\in\Omega, \end{equation}
with $v_i^n(x)\in\gamma(u_i^n(x))$ for $\nu$-a.e. $x\in\Omega_1$ and $v_i^n(x)\in\beta(u_i^n(x))$ for $\nu$-a.e. $x\in\Omega_2$, $i=1,\ldots,n$. That is, we may find the  unique solution $v_i^n$ of the time discretization scheme associated with~\eqref{sabore001bevolshort}:
$$v_i^n+\frac{T}{n} \mathcal{B}_{\a_p}^{m,\gamma,\beta}(v_i^n)\ni \frac{T}{n} f_i^n + v_{i-1}^n \ \hbox{ for $i=1,\ldots,n$}.$$
  However, to apply the results in Subsection \ref{existencesection}, we must ensure that
\begin{equation}\label{finalexistenceconditionphi}
 \mathcal{R}_{\gamma,\beta}^-<\int_\Omega\left(\frac{T}{n} f_i^n + v_{i-1}^n\right) d\nu<\mathcal{R}_{\gamma,\beta}^+
\end{equation}
 holds for each step.
For the first step we need that
  $$\mathcal{R}_{\gamma,\beta}^-<\int_{\Omega}v_0^n d\nu
+\frac{T}{n}\int_\Omega f_1^n d\nu  <\mathcal{R}_{\gamma,\beta}^+$$
holds so that condition~\eqref{finalexistenceconditionphi} is satisfied.  Integrating \eqref{solutiondiscretization2} with respect to $\nu$ over $\Omega$  we get
$$\int_\Omega v_1^n d\nu=\int_{\Omega}v_0^n d\nu
+\frac{T}{n}\int_\Omega f_1^n d\nu$$
thus
$$\frac{T}{n} \int_\Omega f_2^n d\nu +\int_\Omega v_{1}^n d\nu=\frac{T}{n} \sum_{j=1}^2 \int_\Omega f_j^n d\nu +\int_\Omega v_{0}^n d\nu, $$
so that, for the second step, we need
 $$\mathcal{R}_{\gamma,\beta}^-<\frac{T}{n} \sum_{j=1}^2 \int_\Omega f_j^n d\nu +\int_\Omega v_{0}^n d\nu  <\mathcal{R}_{\gamma,\beta}^+.$$
Therefore, we recursively obtain that, for each $n$ and each step $i=1,\ldots, n$, the following must be satisfied:
$$\mathcal{R}_{\gamma,\beta}^-< \frac{T}{n} \sum_{j=1}^i \int_\Omega f_j^n d\nu +\int_\Omega v_{0}^{n} d\nu <\mathcal{R}_{\gamma,\beta}^+. $$
However, taking $n$ large enough,  this holds thanks to~\eqref{loquenec001},~\eqref{laap01} and~\eqref{laap02}.

 Therefore,
$$  v_n (t):=\left\{\begin{array}{ll}
v_0^n, & \hbox{if $t=0$},\\[6pt]
v_i^n, & \hbox{if $t\in ]t_{i-1}^n,t_i^n]$, $i=1,\ldots,n$},
\end{array} \right.$$
    is an   $\epsilon$-approximate solution of Problem \eqref{sabore001bevolshort} as defined in nonlinear semigroup theory.
   Consequently, by nonlinear semigroup theory (see~\cite{Benilantesis}, \cite[Theorem 4.1]{BARBU2}, or~\cite[Theorem A.27]{ElLibro}) and on account of Theorem \ref{rangecondition02} and Theorem \ref{remdom01bevol} we have that Problem \eqref{sabore001bevolshort} has a unique mild solution   $v\in C([0,T];L^1(\Omega,\nu))$ with \begin{equation}\label{vepsilonconvergence}
v_n(t)\stackrel{n}{\longrightarrow} v(t) \ \hbox{ in $L^1(\Omega,\nu)$ uniformly for $t\in[0,T]$}.
\end{equation}
Uniqueness and the maximum principle   for mild solutions is guaranteed by the $T$-accretivity of the operator.

Let us now see that  $v$ is a strong solution of Problem~\eqref{sabore001bevolshort} when \eqref{conditiononjstar} holds.   Note that, since $v_0\in L^{p'}(\Omega,\nu)$, we may take $v_0^n=v_0$ for every $n\in\N$ in the previous computations and $f_i^n\in L^{p'}(\Omega,\nu)$, $i=1,\ldots n$, additionally satisfying
$$\sum_{i=1}^n\int_{t_{i-1}^n}^{t_i^n}\Vert f(t)- f_i^n\Vert^{p'}_{L^{p'}(\Omega,\nu)}dt \le \epsilon.
$$

Let us define
 $$  u_n (t)=
u_i^n  \  \hbox{ for $t\in ]t_{i-1}^n,t_i^n]$, $\, i=1,\ldots,n$}.
$$

Multiplying equation \eqref{solutiondiscretization2} by $u_i^n$ and integrating over $\Omega$ with respect to $\nu$ we obtain
\begin{equation}\label{soltimesuiintegrated}
\begin{array}{c}
\displaystyle\int_\Omega\frac{v_i^n(x)-v_{i-1}^n(x)}{T/n}u_i^n(x)d\nu(x)-\int_\Omega \int_\Omega
\a_p(x,y,u_i^n(y)-u_i^n(x))dm_x(y)u_i^n(x)d\nu(x)\\[14pt]
\displaystyle= \int_\Omega f_i^n(x)u_i^n(x) d\nu(x).
\end{array}
\end{equation}
Now, since $v_i^n(x)\in\gamma(u_i^n(x))$ for $\nu$-a.e. $x\in\Omega_1$ and $v_i^n(x)\in\beta(u_i^n(x))$ for $\nu$-a.e. $x\in\Omega_2$,
$$
\left\{\begin{array}{cc}
  u_i^n(x)\in\gamma^{-1}(v_i^n(x))=\partial j_\gamma^*(v_i^n(x))\, & \hbox{for $\nu$-a.e. $x\in\Omega_1$}, \\ [8pt]
    u_i^n(x)\in\beta^{-1}(v_i^n(x))=\partial j_\beta^*(v_i^n(x))\, & \hbox{for $\nu$-a.e. $x\in\Omega_2$}.
\end{array}\right.
$$
Consequently,
$$
\left\{\begin{array}{cc}
  j_\gamma^*(v_{i-1}^n(x))- j_\gamma^*(v_{i}^n(x))\ge (v_{i-1}^n(x)-v_i^n(x))u_i^n(x)\, & \hbox{for $\nu$-a.e. $x\in\Omega_1$}, \\ [8pt]
    j_\beta^*(v_{i-1}^n(x))- j_\beta^*(v_{i}^n(x))\ge (v_{i-1}^n(x)-v_i^n(x))u_i^n(x)\, & \hbox{for $\nu$-a.e. $x\in\Omega_2$}.
\end{array}\right.
$$
Therefore, from \eqref{soltimesuiintegrated} it follows that
$$
\begin{array}{l}
\displaystyle\frac{1}{T/n}\int_{\Omega_1}(j_\gamma^*(v_{i}^n(x))- j_\gamma^*(v_{i-1}^n(x)))d\nu(x)+
\frac{1}{T/n}\int_{\Omega_2}(j_\beta^*(v_{i}^n(x))- j_\beta^*(v_{i-1}^n(x)))d\nu(x) \\ [12pt]
\displaystyle\hspace{10pt} -\int_\Omega \int_\Omega
\a_p(x,y,u_i^n(y)-u_i^n(x))u_i^n(x)dm_x(y)d\nu(x)\\ [14pt]
\displaystyle \le \int_\Omega f_i^n(x)u_i^n(x) d\nu(x),
\end{array}
$$
$i=1,\ldots,n$. Then, integrating this equation over $]t_{i-1}^n,t_i^n]$ and adding for $1\le i \le n$ we get
$$
\begin{array}{l}
\displaystyle \int_{\Omega_1}(j_\gamma^*(v_{n}^n(x))- j_\gamma^*(v_{0}(x)))d\nu(x)+
 \int_{\Omega_2}(j_\beta^*(v_{n}^n(x))- j_\beta^*(v_{0}(x)))d\nu(x) \\ [12pt]
\displaystyle\hspace{10pt} -\sum_{i=1}^n\int_{t_{i-1}^n}^{t_i^n} \int_\Omega \int_\Omega
\a_p(x,y,u_i^n(y)-u_i^n(x))dm_x(y)u_i^n(x)d\nu(x)dt\\ [14pt]
\displaystyle \le \sum_{i=1}^n\int_{t_{i-1}^n}^{t_i^n}\int_\Omega f_i^n(x)u_i^n(x) d\nu(x)dt,
\end{array}
$$
which, recalling the definitions of $f_n$, $u_n$ and $v_n$, and integrating by parts, can be rewritten as
\begin{equation}\label{jstarinequality}
\begin{array}{l}
\displaystyle \int_{\Omega_1}(j_\gamma^*(v_{n}^n(x))- j_\gamma^*(v_{0}(x)))d\nu(x)+
 \int_{\Omega_2}(j_\beta^*(v_{n}^n(x))- j_\beta^*(v_{0}(x)))d\nu(x) \\ [12pt]
\displaystyle\hspace{10pt} +\frac{1}{2}\int_0^T\int_\Omega \int_\Omega
\a_p(x,y,u_n(t)(y)-u_n(t)(x))(u_n(t)(y)-u_n(t)(x))dm_x(y)d\nu(x)dt\\ [14pt]
\displaystyle \le \int_0^T\int_\Omega f_n(t)(x)u_n(t)(x) d\nu(x)dt.
\end{array}
\end{equation}
This, together with \eqref{llo2} and the fact that $j^*_\gamma$ and $j^*_\beta$ are nonnegative, yields
$$
\begin{array}{l}
\displaystyle \frac{c_p}{2}\int_0^T\int_\Omega \int_\Omega
|u_n(t)(y)-u_n(t)(x)|^p dm_x(y)d\nu(x)dt\\ [14pt]
\displaystyle \le \frac{1}{2}\int_0^T\int_\Omega \int_\Omega
\a_p(x,y,u_n(t)(y)-u_n(t)(x))(u_n(t)(y)-u_n(t)(x))dm_x(y)d\nu(x)dt\\ [14pt]
\displaystyle \le \int_{\Omega_1}( j_\gamma^*(v_{0}(x)))d\nu(x)+
 \int_{\Omega_2}( j_\beta^*(v_{0}(x)))d\nu(x) + \int_0^T\int_\Omega f_n(t)(x)u_n(t)(x) d\nu(x)dt\\[14pt]
 \displaystyle \le \int_{\Omega_1}( j_\gamma^*(v_{0}(x)))d\nu(x)+
 \int_{\Omega_2}( j_\beta^*(v_{0}(x)))d\nu(x) + \int_0^T \Vert f_n(t)\Vert_{L^{p'}(\Omega,\nu)} \Vert u_n(t) \Vert_{L^{p}(\Omega,\nu)} dt.
\end{array}
$$
Therefore, for any $\delta>0$, by \eqref{conditiononjstar} and Young's inequality, there exists $C(\delta)>0$ such that
\begin{equation}\label{bounddelta}
\displaystyle \int_0^T\int_\Omega \int_\Omega
|u_n(t)(y)-u_n(t)(x)|^p dm_x(y)d\nu(x)dt \le C(\delta)+ \delta \int_0^T \Vert u_n(t) \Vert_{L^{p}(\Omega,\nu)}^{p} dt.
\end{equation}

Now, by \eqref{vepsilonconvergence}, if $\mathcal{R}_{\gamma,\beta}^+=+\infty$, there exists $M>0$ and $n_0\in\N$ such that
$$\sup_{t\in[0,T]}\int_\Omega v_{n}^+(t)(x)d\nu(x)<M \quad \hbox{for every } n\ge n_0,$$
and, if $\mathcal{R}_{\gamma,\beta}^+<+\infty$, there exist $M\in\R$, $h>0$ and $n_0\in\N$ such that
$$\sup_{t\in[0,T]}\int_\Omega v_{n}(t)(x)d\nu(x)<M<\mathcal{R}_{\gamma,\beta}^+,$$
and
$$\sup_{t\in[0,T]}  \int_{\{x\in\Omega \, : \, v_{n}(t)(x)<-h\}}|v_{n}(t)(x)|d\nu(x)<\frac{\mathcal{R}_{\gamma,\beta}^+-M}{8} \quad \hbox{for every } n\ge n_0.$$
Consequently, Lemma \ref{LemaAcotLp} and Lemma \ref{LemaAcotLp2} yield
$$\Vert u_{n}^+(t)\Vert_{L^{p}(\Omega,\nu)}\le C_2\left(\left(\int_\Omega \int_\Omega
|u_{n}^+(t)(y)-u_{n}^+(t)(x)|^p dm_x(y)d\nu(x)\right)^\frac1p+1 \right) $$
for some constant $C_2>0$.
Similarly, we may find $C_3>0$ such that
$$\Vert u_{n}^-(t)\Vert_{L^{p}(\Omega,\nu)}\le C_3\left(\left(\int_\Omega \int_\Omega
|u_{n}^-(t)(y)-u_{n}^-(t)(x)|^p dm_x(y)d\nu(x)\right)^\frac1p+1 \right) .$$
Consequently, by \eqref{bounddelta}, choosing $\delta$ small enough, we deduce that $\{u_{n}\}_n$ is bounded in $L^p(0,T;L^p(\Omega,\nu))$. Therefore, there exists a subsequence, which we continue to denote by $\{u_{n}\}_n$, and $u\in L^p(0,T;L^p(\Omega,\nu))$ such that
$$u_{n}\stackrel{n}{\rightharpoonup} u \ \hbox{ weakly in } L^p(0,T;L^p(\Omega,\nu)).$$
 Then, since $\gamma$ and $\beta$ are maximal monotone graphs, we conclude that $v(t)(x)\in\gamma (u(t)(x))$ for $\mathcal{L}^1\otimes\nu$-a.e. $(t,x)\in (0,T)\times\Omega_1$   and $v(t)(x)\in\beta(u(t)(x))$   for $\mathcal{L}^1\otimes\nu$-a.e. $(t,x)\in (0,T)\times\Omega_2$.

Note that, since, by \eqref{bounddelta},
$$
\displaystyle \left\{\int_0^T\int_\Omega \int_\Omega
|u_{n}(t)(y)-u_{n}(t)(x)|^p dm_x(y)d\nu(x)dt\right\}_n \quad\hbox{is bounded},
$$
 then, by \eqref{llo1},  $\{[(t,x,y)\mapsto \a_p(x,y,u_{n}(t)(y)-u_{n}(t)(x))]\}_n$ is bounded in $L^{p'}(0,T; L^{p'}(\Omega\times \Omega,\nu\otimes m_x))$ so we may take a further subsequence, which we still denote in the same way, such that
$$[(t,x,y)\mapsto \a_p(x,y,u_{n}(t)(y)-u_{n}(t)(x))]\stackrel{n}{\rightharpoonup} \Phi, \ \hbox{ weakly in } L^{p'}(0,T; L^{p'}(\Omega\times \Omega,\nu\otimes m_x)).$$
Note that, for any $\xi\in L^p(\Omega,\nu)$, by the integrations by parts formula we know that
$$
  \begin{array}{l}
   \displaystyle -\int_\Omega\int_\Omega \a_p(x,y,u_{n}(t)(y)-u_{n}(t)(x))\xi(x) dm_x(y)d\nu(x)\\ [14pt]
  \displaystyle =\frac12\int_\Omega\int_\Omega \a_p(x,y,u_{n}(t)(y)-u_{n}(t)(x))(\xi(y)-\xi(x))dm_x(y)d\nu(x)
  \end{array}
$$
for $t\in [0,T]$, thus taking limits as $n\to\infty$ we have
\begin{equation}\label{integrationbypartsforphi}
-\int_\Omega\int_\Omega \Phi(t,x,y) \xi(x) dm_x(y)d\nu(x)= \frac12\int_\Omega\int_\Omega \Phi(t,x,y)(\xi(y)-\xi(x))dm_x(y)d\nu(x).
\end{equation}
Now, from \eqref{solutiondiscretization2} we have that
\begin{equation}\label{limittime}
\frac{v_{n}(t)(x)-v_{n}(t-{T/n})(x)}{T/n}-\int_\Omega
\a_p(x,y,u_{n}(t)(y)-u_{n}(t)(x))dm_x(y)= f_{n}(t)(x) \end{equation}
for $t\in[0,T]$ and $x\in\Omega$.
Let   $\Psi\in W^{1,1}_0(0,T;L^p(\Omega,\nu))$, $\hbox{supp}(\Psi)\subset\subset[0,T]$, then
$$\begin{array}{c}
\displaystyle\int_0^T \frac{v_{n}(t)(x)-v_{n}(t-{T/n})(x)}{T/n}\Psi(t)(x)dt\\[14pt]
\displaystyle=-\int_0^{T-T/n} v_{n}(t)(x)\frac{\Psi(t+T/n)(x)-\Psi(t)(x)}{T/n}dt+\int_{T-T/n}^T \frac{v_{n}\Psi(t)(x)}{T/n}dt-\int_0^{T/n} \frac{v_0\Psi(t)(x)}{T/n}dt
\end{array}$$
for $x\in\Omega$. Therefore, multiplying \eqref{limittime} (for the previously chosen subsequence) by $\Psi$, integrating over $(0,T)\times\Omega$ with respect to $\mathcal{L}^1\otimes\nu$ and taking limits  we get
\begin{equation}\label{weakderivative}
\begin{array}{c}
\displaystyle -\int_0^T \int_\Omega  v(t)(x)\frac{d}{dt}\Psi(t)(x)d\nu(x) dt-  \int_0^T\int_\Omega\int_\Omega
\Phi(t,x,y)dm_x(y)\Psi(t)(x)d\nu(x)dt \\ \\
\displaystyle = \int_0^T\int_\Omega f(t)(x)\Psi(t)(x)d\nu(x)dt.
\end{array}
\end{equation}
Therefore, taking $\Psi(t)(x)=\psi(t)\xi(x)$, where $\psi\in C_c^\infty(0,T)$ and $\xi\in L^p(\Omega,\nu)$, we obtain that
$$
\displaystyle \int_0^T  v(t)(x)\psi'(t)  dt=-\int_0^T \int_\Omega
\Phi(t,x,y)\psi(t)dm_x(y) dt
  - \int_0^T f(t)(x)\psi(t) dt, \, \hbox{ for $\nu$-a.e. $x\in\Omega$}.
$$
It follows that
$$
\displaystyle v'(t)(x)  =  \int_\Omega
\Phi(t,x,y)dm_x(y)     +   f(t)(x) \quad\hbox{for a.e. $t\in (0,T)$ and $\nu$-a.e. $x\in\Omega$}.
$$
Therefore, since $v\in C([0,T];L^1(\Omega,\nu))$, $\Phi\in L^{p'}(0,T; L^{p'}(\Omega\times \Omega,\nu\otimes m_x))$ and $f\in L^{p'}(0,T;L^{p'}(\Omega,\nu))$, we have  $v'\in  L^{p'}(0,T;L^{p'}(\Omega,\nu))$ and $v\in W^{1,1}(0,T;L^{1}(\Omega,\nu))$.

Hence, to conclude it remains to prove that
$$ \int_\Omega \Phi(t,x,y)dm_x(y)=\int_\Omega \a_p(x,y,u(t)(y)-u(t)(x))dm_x(y)$$
for $\mathcal{L}^1\otimes\nu$-a.e. $(t,x)\in [0,T]\times\Omega$.
  To this aim we make use of the following claim that will be proved later on:
\begin{equation}\label{finalclaim}
\begin{array}{l}
\displaystyle\limsup_n \int_0^T \int_{\Omega}\int_{\Omega} \a_p(x,y,u_{n}(t)(y)-u_{n}(t)(x))(u_{n}(t)(y)-u_{n}(t)(x)))dm_x(y)d\nu(x)dt\\[14pt]
\displaystyle \le \int_0^T \int_{\Omega}\int_{\Omega} \Phi(t,x,y)(u(t)(y)-u(t)(x)))dm_x(y)d\nu(x)dt.
\end{array}\end{equation}
Now, let $\rho\in L^p(0,T;L^p(\Omega,\nu))$. By \eqref{llo3} we have
$$
\begin{array}{l}
 \displaystyle \int_0^T \int_\Omega\int_\Omega \a_p(x,y,\rho(t)(y)-\rho(t)(x)) \\[12pt]  \hspace{80pt} \displaystyle\times(u_{n}(t)(y)-\rho(t)(y)-(u_{n}(t)(x)-\rho(t)(x)))dm_x(y)d\nu(x)dt\\[14pt]
 \displaystyle \le\int_0^T \int_\Omega\int_\Omega \a_p(x,y,u_{n}(t)(y)-u_{n}(t)(x))\\[12pt]
 \displaystyle \hspace{80pt} \times(u_{n}(t)(y)-\rho(t)(y)-(u_{n}(t)(x)-\rho(t)(x)))dm_x(y)d\nu(x)dt.
  \end{array}
$$
Thus, taking limits as $n\to\infty$ and using \eqref{finalclaim}, we obtain
$$
\begin{array}{l}
 \displaystyle \int_0^T \int_\Omega\int_\Omega \a_p(x,y,\rho(t)(y)-\rho(t)(x)) \\[12pt]  \hspace{80pt} \displaystyle\times(u(t)(y)-\rho(t)(y)-(u(t)(x)-\rho(t)(x)))dm_x(y)d\nu(x)dt\\[14pt]
 \displaystyle \le\int_0^T \int_\Omega\int_\Omega \Phi(t,x,y)(u(t)(y)-\rho(t)(y)-(u(t)(x)-\rho(t)(x)))dm_x(y)d\nu(x)dt,
  \end{array}
$$
which, integrating by parts and recalling \eqref{integrationbypartsforphi}, becomes
$$
\begin{array}{l}
 \displaystyle \int_0^T \int_\Omega\int_\Omega \a_p(x,y,\rho(t)(y)-\rho(t)(x))dm_x(y)(u(t)(x)-\rho(t)(x))d\nu(x)dt\\[14pt]
 \displaystyle \ge\int_0^T \int_\Omega\int_\Omega \Phi(t,x,y)dm_x(y)(u(t)(x)-\rho(t)(x))d\nu(x)dt.
  \end{array}
$$
To conclude, take $\rho=u\pm \lambda\xi$ for $\lambda>0$ and $\xi\in L^p(0,T;L^p(\Omega,\nu))$ to get
$$
\begin{array}{l}
 \displaystyle \int_0^T \int_\Omega\int_\Omega \a_p(x,y,(u\pm \lambda\xi)(t)(y)-(u\pm \lambda\xi)(t)(x))dm_x(y)\xi(t)(x)d\nu(x)dt\\[14pt]
 \displaystyle \ge\int_0^T \int_\Omega\int_\Omega \Phi(t,x,y)dm_x(y)\xi(t)(x)d\nu(x)dt
  \end{array}
$$
which, letting $\lambda\to 0 $ yields
$$
\begin{array}{l}
 \displaystyle \int_0^T \int_\Omega\int_\Omega \a_p(x,y,u(t)(y)-u(t)(x))dm_x(y)\xi(t)(x)d\nu(x)dt\\[14pt]
 \displaystyle =\int_0^T \int_\Omega\int_\Omega \Phi(t,x,y)dm_x(y)\xi(t)(x)d\nu(x)dt
  \end{array}
$$
for any $\xi\in L^p(0,T;L^p(\Omega,\nu))$. Therefore,
$$\int_\Omega \a_p(x,y,u(t)(y)-u(t)(x))dm_x(y)=\int_\Omega \Phi(t,x,y)dm_x(y) $$
for $\mathcal{L}^1\otimes\nu$-a.e. $(t,x)\in [0,T]\times\Omega$.

Let us prove claim \eqref{finalclaim}. By \eqref{jstarinequality} and Fatou's lemma, we have
\begin{equation}\label{Fatouapplication}
\begin{array}{l}
\displaystyle\limsup_n\frac{1}{2}\int_0^T\int_\Omega \int_\Omega
\a_p(x,y,u_{n}(t)(y)-u_{n}(t)(x))(u_{n}(t)(y)-u_{n}(t)(x))dm_x(y)d\nu(x)dt\\[14pt]
\displaystyle \le -\int_{\Omega_1}(j_\gamma^*(v(T)(x))- j_\gamma^*(v(0)(x)))d\nu(x)
 -\int_{\Omega_2}(j_\beta^*(v(T)(x))- j_\beta^*(v(0)(x)))d\nu(x)\\[12pt]
 \ \ \displaystyle +\int_0^T\int_\Omega f(t)(x)u(t)(x) d\nu(x)dt.
\end{array}
\end{equation}
Moreover, by \eqref{weakderivative},
\begin{equation}\label{DefF}
\int_0^T  v(t)(x)\frac{d}{dt}\Psi(t)(x)dt=\int_0^T F(t)(x)\Psi(t)(x)dt, \, \hbox{ for $\nu$-a.e. $x\in\Omega$,}
\end{equation}
where $F$ is given by
\begin{equation}\label{definitionF}
F(t)(x)=-\int_\Omega \Phi(t,x,y)dm_x(y)-f(t)(x), \ x\in\Omega .
\end{equation}
Let $\psi\in   C_c^\infty(0,T)$, $\psi\ge 0$, $\tau>0$ and
$$\eta_\tau (t)(x)=\frac{1}{\tau}\int_t^{t+\tau} u(s)(x) \psi(s)ds, \ \ t\in[0,T], \, x\in\Omega.$$
Then, for $\tau$ small enough,  $\eta_\tau\in W^{1,1}_0(0,T;L^p(\Omega,\nu))$ and we may use it as a test function in \eqref{DefF} to obtain
$$
\begin{array}{rl}
\displaystyle\int_0^T F(t)(x)\eta_\tau(t)(x) dt&\displaystyle=\int_0^T v(t)(x)\frac{d}{dt}\eta_\tau(t)(x)\\[14pt]
&\displaystyle= \int_0^T v(t)(x)\frac{u(t+\tau)(x)\psi(t+\tau)-u(t)(x)\psi(t)}{\tau} dt \\[14pt]
&\displaystyle=\int_0^T \frac{v(t-\tau)(x)-v(t)(x)}{\tau} u(t)(x)\psi(t) dt.
\end{array}$$
Now,
$$\gamma^{-1}(r)= \partial j_{\gamma^{-1}}(r)=\partial\left(\int_0^r (\gamma^{-1})^0(s) ds\right),$$
thus, for $v,\hat v \in \gamma(u)$,
$$(\hat v -v)u\le\int_{v}^{\hat v}(\gamma^{-1})^0(s)ds.$$
A similar fact holds for $\beta$.
Then, for $\tau>0$ fixed, since $v(t)(x)\in\gamma (u(t)(x))$ for $\mathcal{L}^1\otimes\nu$-a.e. $(t,x)\in (0,T)\times\Omega_1$   and $v(t)(x)\in\beta(u(t)(x))$   for $\mathcal{L}^1\otimes\nu$-a.e. $(t,x)\in (0,T)\times\Omega_2$,
$$
\begin{array}{l}
\displaystyle\int_0^T \int_\Omega F(t)(x)\eta_\tau(t)(x)d\nu(x) dt \\[14pt]
\displaystyle\le \frac{1}{\tau}\int_0^T\int_{\Omega_1} \int_{v(t)(x)}^{v(t-\tau)(x)}(\gamma^{-1})^0(s)ds d\nu(x)\psi(t)dt+\frac{1}{\tau}\int_0^T\int_{\Omega_2} \int_{v(t)(x)}^{v(t-\tau)(x)}(\beta^{-1})^0(s)ds d\nu(x)\psi(t)dt\\[14pt]
\displaystyle =\int_0^T\int_{\Omega_1}\int_0^{v(t)(x)}(\gamma^{-1})^0(s)dsd\nu(x) \frac{\psi(t+\tau)-\psi(t)}{\tau}dt\\[12pt]
\displaystyle \hspace{10pt} +\int_0^T\int_{\Omega_2}\int_0^{v(t)(x)}(\beta^{-1})^0(s)dsd\nu(x) \frac{\psi(t+\tau)-\psi(t)}{\tau}dt.
\end{array}$$
Letting $\tau\to 0^+$ in the above expression, by the Dominated Convergence Theorem,
$$
\begin{array}{l}
\displaystyle\int_0^T \int_\Omega F(t)(x)u(t)(x)\psi(t)d\nu(x) dt \\[14pt]
\displaystyle\le\int_0^T\int_{\Omega_1}\int_0^{v(t)(x)}(\gamma^{-1})^0(s)ds d\nu(x)\psi'(t) dt+\int_0^T\int_{\Omega_2}\int_0^{v(t)(x)}(\beta^{-1})^0(s)ds d\nu(x)\psi'(t) dt\\[14pt]
\displaystyle=\int_0^T\int_{\Omega_1} j_{\gamma^{-1}}(v(t)(x))d\nu(x) \psi'(t) dt+\int_0^T\int_{\Omega_2} j_{\beta^{-1}}(v(t)(x))d\nu(x) \psi'(t) dt\\[14pt]
\displaystyle = \int_0^T \int_{\Omega_1}j_{\gamma}^*(v(t)(x))d\nu(x) \psi'(t) dt+\int_0^T \int_{\Omega_2}j_{\beta}^*(v(t)(x))d\nu(x) \psi'(t) dt.
\end{array}$$
Taking
$$\tilde{\eta}_\tau (t)(x)=\frac{1}{\tau}\int_t^{t+\tau} u(s-\tau)(x) \psi(s)ds$$
yields the opposite inequality so that, in fact,
$$
\begin{array}{l}
\displaystyle \int_0^T\int_\Omega F(t)(x)u(t)(x)d\nu(x)\psi(t) dt \displaystyle \\[14pt]
\displaystyle= \int_0^T \int_{\Omega_1}j_{\gamma}^*(v(t)(x))d\nu(x) \psi'(t) dt+\int_0^T \int_{\Omega_2}j_{\beta}^*(v(t)(x))d\nu(x) \psi'(t) dt.
\end{array}
$$
Then,
\begin{equation}\label{W11jgamma}
  -\frac{d}{dt}\left(\int_{\Omega_1}j_{\gamma}^*(v(t)(x))d\nu(x)+\int_{\Omega_2}j_{\beta}^*(v(t)(x))d\nu(x)\right)=\int_\Omega F(t)(x)u(t)(x)d\nu(x)
\end{equation}
in $\mathcal{D}'(]0,T[)$, thus, in particular,
$$\int_{\Omega_1}j_{\gamma}^*(v(t)(x))d\nu(x)+\int_{\Omega_2}j_{\beta}^*(v(t)(x))d\nu(x)\in W^{1,1}(0,T). $$
Therefore, integrating from $0$ to $T$ in \eqref{W11jgamma} and recalling \eqref{definitionF} we get
$$\begin{array}{l}
\displaystyle \int_0^T\int_\Omega \int_\Omega
\Phi(t,x,y)u(t)(x)dm_x(y)d\nu(x)dt\\[14pt]
\displaystyle = -\int_{\Omega_1}(j_\gamma^*(v(T)(x))- j_\gamma^*(v(0)(x)))d\nu(x)
 -\int_{\Omega_2}(j_\beta^*(v(T)(x))- j_\beta^*(v(0)(x)))d\nu(x)\\[12pt]
 \ \ \displaystyle +\int_0^T\int_\Omega f(t)(x)u(t)(x) d\nu(x)dt
\end{array}$$
which, together with \eqref{Fatouapplication}, yields the claim \eqref{finalclaim}. \qed
\end{proof}

 Observe that we have imposed the compatibility condition~\eqref{loquenec001} because, for a strong solution,
$$\int_\Omega v_0d\nu+ \int_0^t\int_\Omega fd\nu ds=\int_\Omega v(t)d\nu, \hbox{ for } t\in[0,T].$$

\begin{example}\label{exDtoN} Let $W\subset X$ be a measurable set such that $W_m$ is $m$-connected. Given $f\in L^1(\partial_mW,\nu)$, we say that a function $u\in L^1(W_m,\nu)$ is an $\a_p$-lifting of $f$ to $W_m=W\cup\partial_mW$ if
 $$
\left\{ \begin{array}{ll}  -\hbox{div}_m\a_p u(x) = 0, &  x\in W, \\ [8pt]
u(x)=f(x), & x\in\partial_m W. \end{array} \right.
$$
We define the Dirichlet-to-Neumann operator $\mathfrak{D}_{\a_p}\subset L^1(\partial_mW,\nu)\times L^1(\partial_mW ,\nu)$ as follows:
$(f,\psi)\in \mathfrak{D}_{\a_p}$ if
$$\mathcal{N}^{\a_p}_1  u(x)=\psi(x),\quad x\in \partial_mW,$$
where $u$ is an $\a_p$-lifting of $f$ to $W_m$.

Then, rewriting the operator $\mathfrak{D}_{\a_p}$ as $\mathcal{B}^{m,\gamma,\beta}_{\a_p}$ for $\gamma(r)=0$ and $\beta(r)=r$, $r\in\R$, ($\Omega_1=W$ and $\Omega_2=\partial_mW$), by the results in this subsection we have that $\mathfrak{D}_{\a_p}$ is $T$-accretive in $L^1(\partial_mW,\nu)$ (it is easy to see that, in fact, in this situation, it is completely accretive), it satisfies the range condition
$$L^{p'}(\partial_mW,\nu)\subset R(I+\mathfrak{D}_{\a_p}),$$ and it has dense domain.
 The non-homogeneous Cauchy evolution problem for this nonlocal Dirichlet-to-Neumann operator is a particular case of Problem~\eqref{sabore001bevolshort}:
$$
\left\{ \begin{array}{ll}
\displaystyle    -   div_m\a_p(u)(x)=0,    &x\in  W,\ 0<t<T,
\\ \\
\displaystyle u_t(t,x) + \mathcal{N}^{\a_p}_1  u(t,x)=g(t,x) ,   &x\in\partial_mW, \  0<t<T,
 \\ \\w(0,x) = w_0(x),    &x\in \partial_mW. \end{array} \right.
$$
See, for example,  \cite{ammar1}, \cite{ammar2}, \cite{Chilletal},  \cite{Dhauer}, \cite{Sauter} and the references therein, for  other  evolution problems with  $p$-Dirichlet-to-Neumann operators, see~\cite{julioetal} for the problem with convolution kernels.
\end{example}

\subsection{Nonlinear boundary conditions}

In this subsection our aim is to study the following diffusion problem
$$
  \left(DP_{f,v_0}^{ \a_p,\gamma,\beta}\right) \quad\left\{ \begin{array}{ll} \displaystyle v_t(t,x) -   \int_{\Omega} \a_p(x,y,u(t,y)-u(t,x)) dm_x(y)=f(t,x),    &x\in  \Omega_1,\ 0<t<T,
\\ \\ \displaystyle v(t,x)\in\gamma\big(u(t,x)\big), &
    x\in  \Omega_1,\ 0<t<T,
\\ \\ \displaystyle    \int_{\Omega} \a_p(x,y,u(t,y)-u(t,x)) dm_x(y) \in\beta\big(u(t,x)\big),    &x\in \Omega_2, \  0<t<T, \\ \\ v(0,x) = v_0(x),    &x\in \Omega_1, \end{array} \right.
$$
that in particular covers Problem~\eqref{sabore001particularintro01}. See~\cite{BCrS} for the reference local model.

We  assume that $$\varGamma^-<\varGamma^+$$ since, otherwise, we do not have an evolution problem. Hence, $\mathcal{R}_{\gamma,\beta}^-<\mathcal{R}_{\gamma,\beta}^+.$ Moreover, we  also assume that
$$\mathfrak{B}^-<\mathfrak{B}^+,$$
since the case $ \mathfrak{B}^-=\mathfrak{B}^+$ ($\beta=\R\times\{0\}$) is treated with more generality in Subsection~\ref{subsect2}.

 We will again make use of nonlinear semigroup theory. To this end we  introduce the corresponding operator  associated to $\left(DP_{f,v_0}^{ \a_p,\gamma,\beta}\right)$, which is now  defined in $L^1(\Omega_1,\nu)$.

\begin{definition}   {\rm  We say that   $(v,\hat v) \in B^{m,\gamma,\beta}_{\a_p}$ if $ v,\hat v \in L^1(\Omega_1,\nu)$
  and there exist  $ u\in L^p(\Omega,\nu)$  and $w\in L^1(\Omega_2,\nu)$  with
  $$u\in \hbox{Dom}(\gamma)  \hbox{ and } v\in\gamma(u) \ \hbox{ $\nu$-a.e. in }\Omega_1,$$
and  $$u\in \hbox{Dom}(\beta)  \hbox{ and } w\in\beta(u) \ \hbox{ $\nu$-a.e. in }\Omega_2,$$
  such that
$$
(x,y)\mapsto a_p(x,y,u(y)-u(x))\in L^{p'}(\Omega\times\Omega,\nu\otimes m_x)
$$
and
$$
\left\{ \begin{array}{ll} \displaystyle -   \int_{\Omega} \a_p(x,y,u(y)-u(x)) dm_x(y) = \hat v \quad &\hbox{in} \ \ \Omega_1, \\ [12pt]
\displaystyle
 w-   \int_{\Omega} \a_p(x,y,u(y)-u(x)) dm_x(y)= 0 \quad &\hbox{in} \ \ \Omega_2; \end{array} \right.
$$
that is, $[u,(v,w)]$ is a solution of $(GP_{(v+\hat v,\mathbf{0})})$, where $\mathbf{0}$ is the null function in $\Omega_2$ (see~\eqref{02091131general} and Definition~\ref{defsol01}).
}
\end{definition}

Set
$$\begin{array}{c}
\mathcal{R}_{\gamma,\lambda\beta}^-:=\nu(\Omega_1)\varGamma^- + \lambda\nu(\Omega_2) \mathfrak{B}^-,
\\[6pt]
 \mathcal{R}_{\gamma,\lambda\beta}^+:=\nu(\Omega_1)\varGamma^+ + \lambda\nu(\Omega_2)\mathfrak{B}^+.
 \end{array}
$$
On account of the results given in Subsection~\ref{efzly} (Theorems \ref{maxandcont01} and \ref{existenceeli01})   we have:

  \begin{theorem}\label{t32}    The  operator $B^{m,\gamma,\beta}_{\a_p}$ is   $T$-accretive in $L^1(\Omega,\nu)$ and satisfies the range condition
$$\left\{\varphi\in L^{p'}(\Omega_1,\nu):\mathcal{R}_{\gamma,\lambda\beta}^-<\int_{\Omega_1}\varphi d\nu<\mathcal{R}_{\gamma,\lambda\beta}^+\right\}\subset R(I+ \lambda B^{m,\gamma,\beta}_{\a_p})\quad\hbox{for every } \lambda>0.
$$
\end{theorem}

\begin{remark}\label{macre01}
Observe that, if $\mathcal{R}_{\gamma,\beta}^-=-\infty$ and $\mathcal{R}_{\gamma,\beta}^+=+\infty$, then the closure of $B^{m,\gamma,\beta}_{\a_p}$ is $m$-$T$-accretive in $L^1(\Omega_1,\nu)$.
\end{remark}

With respect to the domain of this operator we prove the following result.

\begin{theorem}\label{remdom01}
  $$\overline{D(B^{m,\gamma,\beta}_{\a_p})}^{L^{p'}(\Omega_1,\nu)}=\big\{v\in L^{p'}(\Omega_1,\nu)\, :\, \varGamma^-\le v\le \varGamma^+ \big\}.$$
Therefore, we also have
  $$\overline{D(B^{m,\gamma,\beta}_{\a_p})}^{L^{1}(\Omega_1,\nu)}=\big\{v\in L^{1}(\Omega_1,\nu)\, :\, \varGamma^-\le v\le \varGamma^+  \big\}.$$
\end{theorem}

\begin{proof}   It is obvious that
 $$\overline{D(B^{m,\gamma,\beta}_{\a_p})}^{L^{p'}(\Omega_1,\nu)}
 \subset\left\{v\in L^{p'}(\Omega_1,\nu)\,:\, \varGamma^-\le v\le \varGamma^+\ \hbox{$\nu$-\rm a.e. in }\Omega_1
 \right\}.$$
  For the other inclusion it is enough to see that
$$\left\{v\in L^{\infty}(\Omega_1,\nu)\,:\, \varGamma^-\le v\le \varGamma^+\ \hbox{$\nu$-\rm a.e. in }\Omega_1
 \right\}\subset\overline{D(B^{m,\gamma,\beta}_{\a_p})}^{L^{p'}(\Omega_1,\nu)}.
$$
We work on a case-by-case basis.

\noindent (A) Suppose that $\varGamma^-<0<\varGamma^+ $. It is enough to see that for any $v\in L^\infty(\Omega_1,\nu)$ such that there exist $m\in\R$, $\widetilde{m}<0$, $\widetilde{M}>0$, $M\in \R$ satisfying
$$\varGamma^-<m<\widetilde{m}< v<\widetilde{M}<M<\varGamma^+ \ \hbox{ $\nu$-\rm a.e. in } \Omega_1$$
it holds that $v\in \overline{D(B^{m,\gamma,\beta}_{\a_p})}^{L^{p'}(\Omega_1,\nu)}$.

   By the results in Subsection \ref{existencesection} we know that, for $n\in\N$, there exist $u_n\in L^p(\Omega,\nu)$, $v_n\in L^{p'}(\Omega_1,\nu)$ and $w_n\in L^{p'}(\Omega_2,\nu)$, such that $[u_n,(v_n,\frac{1}{n} w_n)]$ is a solution of  $\left(GP_{(v,\mathbf{0})}^{\frac{1}{n}\a_p,\gamma,\beta}\right)$, i.e.,
$v_n\in \gamma(u_n)$ $\nu$-a.e. in $\Omega_1$, $w_n\in \beta(u_n)$ $\nu$-a.e. in $\Omega_2$ and
$$
\left\{\begin{array}{ll}
\displaystyle v_n(x) -\frac1n\int_{\Omega} \a_p(x,y,u_n(y)-u_n(x)) dm_x(y) = v(x), &  \hbox{for } x\in\Omega_1,\\ [14pt]
\displaystyle  w_n(x) - \int_{\Omega} \a_p(x,y,u_n(y)-u_n(x)) dm_x(y) = 0, &  \hbox{for }  x\in\Omega_2.
\end{array}\right.
$$
In other words, $(v_n,n(v-v_n))\in B^{m,\gamma,\beta}_{\a_p}$ or, equivalently,
$$v_n:=\left(I+\frac{1}{n}B^{m,\gamma,\beta}_{\a_p}\right)^{-1}(v) \in  D(B^{m,\gamma,\beta}_{\a_p}).$$
Let us see that  $v_n\stackrel{n}{\longrightarrow} v$ in $L^{p'}(\Omega_1,\nu)$.

\noindent (A1) Suppose first that $\hbox{sup} D(\beta)=+\infty$. Take $a_M>0$ such that $M\in \gamma(a_M)$ and let $N\in\beta(a_M)$.   Let
$$\widehat{v}(x):=\left\{\begin{array}{cc}
             M, & x\in\Omega_1, \\
             N, & x\in\Omega_2,
           \end{array}\right.$$
   $$\widehat{u}(x):=a_M , \ x\in\Omega,$$
      and
      $$\varphi(x):=\left\{\begin{array}{cc}
             \displaystyle  M, & x\in\Omega_1, \\
              \displaystyle 0, & x\in\Omega_2.
           \end{array}\right.$$
Then, $[\widehat{u}, \widehat{v}]$ is a supersolution of $\left(GP^{\frac{1}{n}\a_p,\gamma,\beta}_{\varphi}\right)$ and $(v,\mathbf{0})\le \varphi$. Thus, by the maximum principle  (Theorem~\ref{maxandcont01}),
$$u_n\le \widehat{u}=a_M \, \hbox{ $\nu$-\rm a.e. in }\Omega\quad \hbox{for every } n\in\N.$$

\noindent (A2) Suppose now that $\hbox{sup}D(\beta)=r_\beta<+\infty$.   Again, by the results in Subsection \ref{existencesection} we know that, for $n\in\N$, there exist $\widetilde{u}_n\in L^p(\Omega,\nu)$, $\widetilde{v}_n\in L^{p'}(\Omega_1,\nu)$ and $\widetilde{w}_n\in L^{p'}(\Omega_2,\nu)$, such that $[\widetilde{u}_n,(\widetilde{v}_n,\frac{1}{n}\widetilde{w}_n)]$ is a solution of  $\left(GP_{(M,\mathbf{0})}^{\frac{1}{n}\a_p,\gamma,\beta}\right)$. Therefore, by the maximum principle (Theorem \ref{maxandcont01}),
$$v_n\le \widetilde{v}_n\quad\hbox{$\nu$-\rm a.e. in }\Omega_1.$$
Now, since $\widetilde{v}_n\ll M$ in $\Omega_1$  (recall Remark \ref{BONUS}{\it (iii)}), we have that $\widetilde{v}_n\le M$ and, consequently, also $v_n\le M$. Hence, since $M\le \widetilde{M}<\varGamma^+ $,
$$u_n\le \inf\big(\gamma^{-1}(\widetilde{M})\big) \, \hbox{ $\nu$-\rm a.e. in }\Omega_1,$$
but we also have
$$u_n\le r_\beta \, \hbox{ $\nu$-\rm a.e. in } \Omega_2\quad \hbox{for every } n\in\N.$$

\noindent (B) For $\varGamma^-<0=\varGamma^+$: let $\varGamma^-<m<\widetilde{m}< 0$,    and $v\in L^\infty(\Omega_1,\nu)$ be such that
$$\widetilde{m}\le v< 0.$$
As in the previous case, by the results in Subsection \ref{existencesection}, we know that, for $n\in\N$, there exist $u_n\in L^p(\Omega,\nu)$, $v_n\in L^{p'}(\Omega_1,\nu)$ and $w_n\in L^{p'}(\Omega_2,\nu)$, such that $[u_n,(v_n,\frac{1}{n} w_n)]$ is a solution of  $\left(GP_{(v,\mathbf{0})}^{\frac{1}{n}\a_p,\gamma,\beta}\right)$.
Then, since for the null function $\mathbf{0}$ in $\Omega$,  $[\mathbf{0},\mathbf{0}]$ is a solution of $\left(GP_\mathbf{0}^{\frac{1}{n}\a_p,\gamma,\beta}\right)$ and $v<0$, the maximum principle yields
$$u_n\le 0 \ \hbox{ $\nu$-\rm a.e. in }\Omega\quad \hbox{for every } n\in\N.$$

Therefore, in all the cases, $\{u_n\}_n$ is $L^\infty(\Omega,\nu)$-bounded from above. With a similar reasoning we obtain that, in any of these cases, $\{u_n\}_n$ is also $L^\infty(\Omega,\nu)$-bounded from below. Then, since
$$v_n(x)-v(x)=\frac1n\int_{\Omega} \a_p(x,y,u_n(y)-u_n(x)) dm_x(y)  \ \ \hbox{ in} \ \Omega_1,$$
we obtain that
$$v_n\stackrel{n}{\longrightarrow} v \ \hbox{ in  } L^{p'}(\Omega_1,\nu)$$
as desired. \qed
\end{proof}

The following theorem gives the existence and uniqueness of solutions of Problem $\left(DP_{f,v_0}^{ \a_p,\gamma,\beta}\right)$. Recall that $\varGamma^-<\varGamma^+$ and
$\mathfrak{B}^-<\mathfrak{B}^+$.

\begin{theorem}\label{elsegundo01}  Let  $T>0$.
Let
$v_0\in L^{1}(\Omega_1,\nu)$ and $f\in L^1(0,T;L^{1}(\Omega_1,\nu))$. Assume
 $$\varGamma^-\le v_0\le \varGamma^+ \hbox{ $\nu$-\rm a.e. in }\Omega_1,$$
 and
$$\mathcal{R}_{\gamma,\beta}^+=+\infty \  \hbox{ or  } \
\int_{\Omega_1}f(x,t)d\nu(x)\le \nu(\Omega_2)\mathfrak{B}^+\quad \hbox{for every } 0<t<T,
$$
and
$$\mathcal{R}_{\gamma,\beta}^-=-\infty \ \hbox{ or  } \
\displaystyle \int_{\Omega_1}f(x,t)d\nu(x)\ge \nu(\Omega_2)\mathfrak{B}^-\quad \hbox{for every } 0<t<T.
$$
\\
 Then,  there exists a unique mild-solution $v\in C([0,T];L^1(\Omega_1,\nu))$   of $\displaystyle\left(DP_{f,v_0}^{ \a_p,\gamma,\beta}\right)$.

 Let $v$ and $\widetilde v$ be the mild solutions of the problem with respective data $v_0,\ \widetilde v_0\in L^{1}(\Omega_1,\nu)$ and
   $f,\ \widetilde f\in L^1(0,T;L^{1}(\Omega_1,\nu))$. Then
   $$
   \begin{array}{rl}\displaystyle\int_{\Omega_1} \left(v(t,x)-\widetilde v(t,x)\right)^+d\nu(x)&\displaystyle\le
   \int_{\Omega_1} \left(v_0(x)-\widetilde v_0(x)\right)^+d\nu(x)\\[12pt]
   & \displaystyle \hspace{10pt} +\int_0^t\int_{\Omega_1}\left(f(s,x)-\widetilde f(s,x)\right)^+d\nu(x)ds \quad\hbox{for every } 0\le   t\le T.
   \end{array}$$

Under the additional assumptions
 \begin{equation}\label{newconditionjstar}
 \begin{array}{c}
 \displaystyle v_0\in L^{p'}(\Omega_1,\nu)  \hbox{ and } f\in L^{p'}(0,T;L^{p'}(\Omega_1,\nu)) \hbox{ with } \\[10pt]
\displaystyle\int_{\Omega_1}j_\gamma^*(v_0)d\nu<+\infty \hbox{ and }
\\[12pt]
\displaystyle
 \int_{\Omega_1}v_0^+d\nu+\int_0^T\int_{\Omega_1}f(s)^+d\nu dt< \nu(\Omega_1)\Gamma^+,
\\[12pt]
\displaystyle
 \int_{\Omega_1}v_0^-d\nu+\int_0^T\int_{\Omega_1}f(s)^-d\nu dt<-\nu(\Omega_1)\Gamma^-,
\end{array}
\end{equation}
    the mild solution $v$ belongs to $W^{1,1}(0,T;L^{1}(\Omega_1,\nu))$ and satisfies the equation
     $$
     \left\{\begin{array}{l}\partial_tv(t)+B^{m,\gamma,\beta}_{\a_p}v(t)\ni f(t)\quad\hbox{for a.e. }t\in(0,T),\\[6pt]
     v(0)=v_0,\end{array}\right.$$
     that is, $v$ is  a strong solution.
\end{theorem}

The proof of this result differs, strongly at some points, from the proof of Theorem \ref{nsth01bevol}.

\begin{proof}
We start by proving the existence of mild solutions.
For $n\in\N$, consider the partition $$t_0^n=0<t_1^n<\cdots <t_{n-1}^n<t_n^n=T$$ where $t_i^n:=iT/n$, $i=1,\ldots,n$. Given $\epsilon>0$,  since $\mathfrak{B}^-<\mathfrak{B}^+$, there exist    $n\in \mathbb{N}$, $v_0^n \in \overline{D(B^{m,\gamma,\beta}_{\a_p})}^{L^{p'}(\Omega_1,\nu)}$ (i.e., $v_0^n\in L^{p'}(\Omega_1,\nu)$ satisfying $\varGamma^-\le v_0^n\le \varGamma^+$) and $f_i^n \in L^{p'}(\Omega_1,\nu)$, $i=1,\ldots n$, such that $T/n\le \epsilon$,
$$\Vert v_0-v_0^n\Vert_{L^1(\Omega_1,\nu)} \le\epsilon,$$
  \begin{equation}\label{aproxfconfin}
  \sum_{i=1}^n\int_{t_{i-1}^n}^{t_i^n}\Vert f(t)- f_i^n\Vert_{L^1(\Omega_1,\nu)}dt \le \epsilon
  \end{equation}
  and
 $$\nu(\Omega_2)\mathfrak{B}^-<\int_{\Omega_1}f_i^nd\nu <\nu(\Omega_2)\mathfrak{B}^+.$$
Then, setting
$$f_n(t):=f_i^n, \ \hbox{ for $t\in]t_{i-1}^n,t_i^n]$, $i=1,\ldots, n$},$$
we have that
$$\int_0^T \Vert f(t)- f_n(t) \Vert_{L^{1}(\Omega_1,\nu)}dt \le \epsilon.$$

Using the results in Subsection \ref{existencesection},
 we   see that,  for $n$ large enough, we may recursively find a solution  $[u_i^n,(v_i^n,\frac{T}{n} w_i^n)]$ of $\displaystyle \left(GP^{\frac{T}{n}\a_p,\gamma,\frac{T}{n}\beta}_{\left(\frac{T}{n} f_i^n+v_{i-1}^n,\mathbf{0}\right)}\right)$, $i=1,\ldots,n$, so that
\begin{equation}\label{newsolutiondiscretization}
\left\{\begin{array}{ll}
\displaystyle v_i^n(x)-\frac{T}{n}\int_\Omega
\a_p(x,y,u_i^n(y)-u_i^n(x))dm_x(y)=\frac{T}{n} f_i^n(x)+v_{i-1}^n(x), & \displaystyle x\in\Omega_1\\ [12pt]
\displaystyle   w_i^n(x)- \int_\Omega
\a_p(x,y,u_i^n(y)-u_i^n(x))dm_x(y)=0, & \displaystyle x\in\Omega_2,
\end{array}\right.
\end{equation}
or, equivalently,
\begin{equation}\label{newsolutiondiscretization2}
\left\{\begin{array}{ll}
\displaystyle \frac{v_i^n(x)-v_{i-1}^n(x)}{T/n}-\int_\Omega
\a_p(x,y,u_i^n(y)-u_i^n(x))dm_x(y)= f_i^n(x), & \displaystyle x\in\Omega_1\\ [12pt]
\displaystyle   w_i^n(x)- \int_\Omega
\a_p(x,y,u_i^n(y)-u_i^n(x))dm_x(y)=0, & \displaystyle x\in\Omega_2,
\end{array}\right.
\end{equation}
with $v_i^n(x)\in\gamma(u_i^n(x))$ for $\nu$-a.e. $x\in\Omega_1$ and $w_i^n(x)\in\beta(u_i^n(x))$ for $\nu$-a.e. $x\in\Omega_2$, $i=1,\ldots,n$.   That is, we may find the  unique solution $v_i^n$ of the time discretization scheme associated with $\left(DP_{f,v_0}^{ \a_p,\gamma,\beta}\right)$.

To apply these results we must ensure that
$$
 \mathcal{R}_{\gamma,\frac{T}{n}\beta}^-<\int_{\Omega_1} \left( \frac{T}{n} f_i^n + v_{i-1}^n\right) d\nu<\mathcal{R}_{\gamma,\frac{T}{n}\beta}^+
$$
 holds for each step, but this holds true thanks to the choice of the $f_i^n$, $i=1,\ldots,n$.

Therefore,
$$ v_n (t):=\left\{\begin{array}{ll}
v_0^n, & \hbox{if $t=0$},\\[6pt]
v_i^n, & \hbox{if $t\in ]t_{i-1}^n,t_i^n]$, $i=1,\ldots,n$},
\end{array} \right.$$
is an $\epsilon$-approximate solution of Problem $\left(DP_{f,v_0}^{ \a_p,\gamma,\beta}\right)$.
Consequently, by nonlinear semigroup theory ((see~\cite{Benilantesis}, \cite[Theorem 4.1]{BARBU2}, or~\cite[Theorem A.27]{ElLibro})) and on account of Theorem~\ref{t32} and Theorem~\ref{remdom01} we have that $\left(DP_{f,v_0}^{ \a_p,\gamma,\beta}\right)$ has a unique mild solution $v\in C([0,T];L^1(\Omega_1,\nu))$ with \begin{equation}\label{newvepsilonconvergence}
v_n(t)\stackrel{n}{\longrightarrow} v(t) \ \hbox{ in $L^1(\Omega_1,\nu)$ uniformly for $t\in[0,T]$}.
\end{equation}
Uniqueness and the maximum principle   for mild solutions is guaranteed by the $T$-accretivity of the operator.

We now prove, step by step, that these mild solutions are strong solutions of Problem $\left(DP_{f,v_0}^{ \a_p,\gamma,\beta}\right)$ under the set of assumptions given in~\eqref{newconditionjstar}.

Let us define
 $$  u_n (t):=
u_i^n  \quad  \hbox{for $t\in ]t_{i-1}^n,t_i^n]$, $i=1,\ldots,n$},
$$
and
 $$  w_n (t):=
w_i^n  \quad  \hbox{for $t\in ]t_{i-1}^n,t_i^n]$, $i=1,\ldots,n$}.
$$

  {\it Step 1.} Suppose first that $\mathcal{R}_{\gamma,\beta}^-=-\infty$ and $\mathcal{R}_{\gamma,\beta}^+=+\infty$.

 In  the construction of the mild solution, we now take $v_0^n=v_0$ (since $v_0\in L^{p'}(\Omega_1,\nu)$) and the functions $f_i^n \in L^{p'}(\Omega_1,\nu)$, $i=1,\ldots n$, additionally satisfying
$$\sum_{i=1}^n\int_{t_{i-1}^n}^{t_i^n}\Vert f(t)- f_i^n\Vert^{p'}_{L^{p'}(\Omega_1,\nu)}dt\le \epsilon $$
and
$$ \nu(\Omega_2)\mathfrak{B}^-<\int_{\Omega_1}f_i^nd\nu <\nu(\Omega_2)\mathfrak{B}^+.$$

 Multiplying both equations in \eqref{newsolutiondiscretization2} by $u_i^n$, integrating with respect to $\nu$ over $\Omega_1$ and $\Omega_2$, respectively, and adding them, we obtain
$$
\begin{array}{l}
\displaystyle\int_{\Omega_1}\frac{v_i^n(x)-v_{i-1}^n(x)}{T/n}u_i^n(x)d\nu(x)+\int_{\Omega_2}w_i^n(x) u_i^n(x)d\nu(x)
\\[12pt]
\displaystyle \hspace{10pt} -\int_\Omega \int_\Omega
\a_p(x,y,u_i^n(y)-u_i^n(x))u_i^n(x)dm_x(y)d\nu(x)\\[14pt]
\displaystyle= \int_{\Omega_1} f_i^n(x)u_i^n(x) d\nu(x).
\end{array}
$$
Then, since $w_i^n(x)\in\beta(u_i^n(x))$ for $\nu$-a.e. $x\in\Omega_2$ the second term on the left hand side is nonnegative and integrating by parts the third term we get
\begin{equation}\label{newsoltimesuiintegrated}
\begin{array}{l}
\displaystyle\int_{\Omega_1}\frac{v_i^n(x)-v_{i-1}^n(x)}{T/n}u_i^n(x)d\nu(x)+\frac12 \int_\Omega \int_\Omega
\a_p(x,y,u_i^n(y)-u_i^n(x))(u_i^n(y)-u_i^n(x))dm_x(y)d\nu(x)
\\[14pt]
\displaystyle \le \int_{\Omega_1} f_i^n(x)u_i^n(x) d\nu(x).
\end{array}
\end{equation}
Now, since $v_i^n(x)\in\gamma(u_i^n(x))$ for $\nu$-a.e. $x\in\Omega_1$,
$$
  u_i^n(x)\in\gamma^{-1}(v_i^n(x))=\partial j_\gamma^*(v_i^n(x)) \ \hbox{ for $\nu$-a.e. $x\in\Omega_1$.}
$$
Consequently,
$$
  j_\gamma^*(v_{i-1}^n(x))- j_\gamma^*(v_{i}^n(x))\ge (v_{i-1}^n(x)-v_i^n(x))u_i^n(x) \ \hbox{ for $\nu$-a.e. $x\in\Omega_1$.}
$$
Therefore, from \eqref{newsoltimesuiintegrated} it follows that
$$
\begin{array}{l}
\displaystyle\frac{n}{T}\int_{\Omega_1}(j_\gamma^*(v_{i}^n(x))- j_\gamma^*(v_{i-1}^n(x)))d\nu(x)+\frac12 \int_\Omega \int_\Omega
\a_p(x,y,u_i^n(y)-u_i^n(x))(u_i^n(y)-u_i^n(x))dm_x(y)d\nu(x)\\[14pt]
\displaystyle\le \int_{\Omega_1} f_i^n(x)u_i^n(x) d\nu(x),
\end{array}
$$
$i=1,\ldots,n$. Then, integrating this equation over $]t_{i-1},t_i]$ and adding for $1\le i \le n$ we get
$$
\begin{array}{l}
\displaystyle \int_{\Omega_1}(j_\gamma^*(v_{n}^n(x))- j_\gamma^*(v_{0}(x)))d\nu(x)\\[12pt]
\displaystyle \hspace{10pt} +\frac12 \sum_{i=1}^n\int_{t_{i-1}}^{t_i}\int_\Omega \int_\Omega
\a_p(x,y,u_i^n(y)-u_i^n(x))(u_i^n(y)-u_i^n(x))dm_x(y)d\nu(x)dt\\[14pt]
\displaystyle\le \sum_{i=1}^n\int_{t_{i-1}}^{t_i}\int_{\Omega_1} f_i^n(x)u_i^n(x) d\nu(x)dt,
\end{array}
$$
which, recalling the definitions of $f_n$, $u_n$, $v_n$ and $w_n$, can be rewritten as
\begin{equation}\label{newjstarinequality}
\begin{array}{l}
\displaystyle \int_{\Omega_1}(j_\gamma^*(v_{n}^n(x))- j_\gamma^*(v_{0}(x)))d\nu(x)\\[12pt]
\displaystyle \hspace{10pt} +\frac12 \int_0^T\int_\Omega \int_\Omega
\a_p(x,y,u_n(t)(y)-u_n(t)(x))(u_n(t)(y)-u_n(t)(x))dm_x(y)d\nu(x)dt\\[14pt]
\displaystyle \le \int_0^T\int_{\Omega_1} f_n(t)(x)u_n(t)(x) d\nu(x)dt.
\end{array}
\end{equation}
This, together with \eqref{llo2} and the fact that $j^*_\gamma$ is nonnegative, yields
$$
\begin{array}{l}
\displaystyle \frac{c_p}{2}\int_0^T\int_\Omega \int_\Omega
|u_n(t)(y)-u_n(t)(x)|^p dm_x(y)d\nu(x)dt\\ [14pt]
\displaystyle \le \frac{1}{2}\int_0^T\int_\Omega \int_\Omega
\a_p(x,y,u_n(t)(y)-u_n(t)(x))(u_n(t)(y)-u_n(t)(x))dm_x(y)d\nu(x)dt\\ [14pt]
\displaystyle \le \int_{\Omega_1} j_\gamma^*(v_{0}(x))d\nu(x)+
  \int_0^T\int_{\Omega_1} f_n(t)(x)u_n(t)(x) d\nu(x)dt\\[14pt]
 \displaystyle \le \int_{\Omega_1} j_\gamma^*(v_{0}(x))d\nu(x)+ \int_0^T \Vert f_n(t)\Vert_{L^{p'}(\Omega_1,\nu)} \Vert u_n(t) \Vert_{L^{p}(\Omega_1,\nu)} dt.
\end{array}
$$
Therefore, for any $\delta>0$, by \eqref{newconditionjstar} and Young's inequality, there exists $C(\delta)>0$ such that, in particular,
\begin{equation}\label{newbounddelta}
\displaystyle \int_0^T\int_{\Omega}\int_{\Omega}
|u_n(t)(y)-u_n(t)(x)|^p dm_x(y)d\nu(x)dt \le C(\delta)+ \delta \int_0^T \Vert u_n(t) \Vert_{L^{p}(\Omega_1,\nu)}^{p} dt.
\end{equation}

Observe also that, for any $n\in\N$ and $i\in\{1,\ldots,n\}$, and for $t\in ]t_{i-1}^n,t_i^n]$,
\begin{equation}\label{dsv001}\int_{\Omega_1}v_n^+(t) d\nu+\int_0^{t_i^n}\int_{\Omega_2}w_n^+(s) d\nu ds\le\int_{\Omega_1}v_{0}^+ d\nu+ \int_0^{t_i^n}\int_{\Omega_1}f_n^+(s) d\nu ds.
\end{equation}
 Indeed, multiplying the first equation in \eqref{newsolutiondiscretization} by $\frac{1}{r}T_r^+(u_i^n)$ and integrating with respect to $\nu$ over $\Omega_1$, then multiplying the second by  $\frac{T}{n}\frac{1}{r}T_r^+(u_i^n)$ and integrating with respect to $\nu$ over $\Omega_2$, adding both equations, neglecting the nonnegative term involving $\a_p$ (recall Remark \ref{remmon}) and letting $r\downarrow 0$, we get that
$$\int_{\Omega_1}(v_i^n)^+ d\nu+\frac{T}{n}\int_{\Omega_2}(w_i^n)^+ d\nu\le \int_{\Omega_1}(v_{i-1}^n)^+ d\nu+ \frac{T}{n}\int_{\Omega_1}(f_i^n)^+ d\nu, $$
i.e.,
$$\int_{\Omega_1}(v_i^n)^+ d\nu\le \int_{\Omega_1}(v_{i-1}^n)^+ d\nu+ \frac{T}{n}\int_{\Omega_1}(f_i^n)^+ d\nu -\frac{T}{n}\int_{\Omega_2}(w_i^n)^+ d\nu.$$
Therefore,
$$\int_{\Omega_1}(v_i^n)^+ d\nu\le \int_{\Omega_1}(v_{0}^n)^+ d\nu+ \sum_{j=1}^i\frac{T}{n}\int_{\Omega_1}(f_j^n)^+ d\nu -\sum_{j=1}^i\frac{T}{n}\int_{\Omega_2}(w_j^n)^+ d\nu$$
which is equivalent to \eqref{dsv001}.

   Now, by \eqref{newvepsilonconvergence}, if $\varGamma^+=+\infty$, there exists $M>0$ such that
$$\sup_{t\in[0,T]}\int_{\Omega_1} v_n^+(t)(x)d\nu(x)<M \ \hbox{ for every } n\in\N.$$
Consequently, Lemma \ref{LemaAcotLp} applied for $A=\Omega_1$, $B=\emptyset$ and $\alpha=\gamma$, yields
$$\Vert u_n^+(t)\Vert_{L^{p}(\Omega_1,\nu)}\le   K_2\left(\left(\int_{\Omega_1} \int_{\Omega_1}
|u_n^+(t)(y)-u_n^+(t)(x)|^p dm_x(y)d\nu(x)\right)^\frac1p+1 \right)  $$
for every $n\in\N$, every $0\le t \le T$ and some constant $ K_2>0$.

   Suppose now that $\varGamma^+<+\infty$.
 Then, by \eqref{dsv001} we have that, for any $n\in\N$ and $i\in\{1,\ldots,n\}$, and for $t\in ]t_{i-1}^n,t_i^n]$ if $i\ge 2$, or $t\in [t_{0}^n,t_1^n]$ if $i=1$,
   $$\int_{\Omega_1}v_n^+(t) d\nu\le\int_{\Omega_1}v_{0}^+ d\nu+ \int_0^{t_i^n}\int_{\Omega_1}f_n^+(s) d\nu ds$$
thus, by the assumptions in \eqref{newconditionjstar} and by \eqref{aproxfconfin}, there exists $M\in\R$ such that
$$ \sup_{t\in[0,T]}\int_{\Omega_1} v_n(t)d\nu \le M<\nu(\Omega_1)\Gamma^+$$
for $n$ sufficiently large and, by \eqref{newvepsilonconvergence}, such that
$$\sup_{t\in[0,T]}\int_{\{x\in \Omega_1 \, :\, v_n(t)<-h\}}|v_n(t)|d\nu<\frac{\nu(\Omega_1)\Gamma^+-M}{8}$$
for $n$ sufficiently large. Therefore, we may apply Lemma~\ref{LemaAcotLp2} for $A=\Omega_1$, $B=\emptyset$ and $\alpha=\gamma$ to conclude that there exists a constant $K_2'>0$ such that
$$\Vert u_n^+(t)\Vert_{L^{p}(\Omega_1,\nu)}\le  K_2'\left(\left(\int_{\Omega_1} \int_{\Omega_1}
|u_n^+(t)(y)-u_n^+(t)(x)|^p dm_x(y)d\nu(x)\right)^\frac1p+1 \right)\quad \hbox{for every } 0\le t \le T,  $$
for $n$ sufficiently large.

    Similarly,   we may find $ K_3>0$ such that
$$\Vert u_n^-(t)\Vert_{L^{p}(\Omega_1,\nu)}\le  K_3\left(\left(\int_{\Omega_1} \int_{\Omega_1}
|u_n^-(t)(y)-u_n^-(t)(x)|^p dm_x(y)d\nu(x)\right)^\frac1p+1 \right)\quad \hbox{for every } 0\le t \le T, $$
for $n$ sufficiently large.

Consequently, by the generalised Poincar\'e type inequality together with~\eqref{newbounddelta} for $\delta$ small enough, we get
$$\int_0^T \left\Vert u_n(t) \right\Vert_{L^p(\Omega,\nu)}^{p}dt\le  K_4 \quad \hbox{for every } n\in\N,$$
for some constant $K_4>0$, that is, $\{u_n\}_n$ is bounded in $L^p(0,T;L^p(\Omega,\nu))$. Therefore, there exists a subsequence, which we continue to denote by $\{u_{n}\}_{n}$, and $u\in L^p(0,T;L^p(\Omega,\nu))$ such that
$$u_{n}\stackrel{n}{\rightharpoonup} u \ \hbox{ weakly in } L^p(0,T;L^p(\Omega,\nu)).$$

Note that, since
$
\displaystyle \Big\{\int_0^T\int_\Omega \int_\Omega
|u_n(t)(y)-u_n(t)(x)|^p dm_x(y)d\nu(x)dt\Big\}_n
$
is bounded, then, by \eqref{llo1}, we have that $\{[(t,x,y)\mapsto \a_p(x,y,u_n(t)(y)-u_n(t)(x))]\}_n$ is bounded in $L^{p'}(0,T; L^{p'}(\Omega\times \Omega,\nu\otimes m_x))$ so we may take a further subsequence, which we continue to denote in the same way, such that
$$[(t,x,y)\mapsto \a_p(x,y,u_{n}(t)(y)-u_{n}(t)(x))]\stackrel{n}{\rightharpoonup} \Phi, \ \hbox{ weakly in } L^{p'}(0,T; L^{p'}(\Omega\times \Omega,\nu\otimes m_x)).$$

Now, let   $\Psi\in W^{1,1}_0(0,T;L^p(\Omega,\nu))$, $\hbox{supp}(\Psi)\subset\subset[0,T]$, then
$$\begin{array}{c}
\displaystyle\int_0^T \frac{v_{n}(t)(x)-v_{n}(t-T/n)(x)}{T/n}\Psi(t)(x)dt\\[14pt]
\displaystyle=-\int_0^{T-T/n} v_{n}(t)(x)\frac{\Psi(t+T/n)(x)-\Psi(t)(x)}{T/n}dt+\int_{T-T/n}^T \frac{v_{n}\Psi(t)(x)}{T/n}dt-\int_0^{T/n} \frac{z_0\Psi(t)(x)}{T/n}
\end{array}$$
for $x\in\Omega_1$. Therefore, multiplying both equations in \eqref{newsolutiondiscretization2} by $\Psi$, integrating the first one over $\Omega_1$ and the second one over $\Omega_2$ with respect to $\nu$, adding them, and taking limits as $n\to +\infty$ we get that
$$
\begin{array}{l}
\displaystyle -\int_0^T \int_{\Omega_1}  v(t)(x)\frac{d}{dt}\Psi(t)(x)d\nu(x) dt+\int_0^T \int_{\Omega_2}w(t)(x)\Psi(t)(x)d\nu(x)dt \\ [12pt]
\hspace{10pt}\displaystyle-  \int_0^T\int_\Omega\int_\Omega
\Phi(t,x,y)dm_x(y)\Psi(t)(x)d\nu(x)dt \\ [14pt]
\displaystyle = \int_0^T\int_{\Omega_1} f(t)(x)\Psi(t)(x)d\nu(x)dt.
\end{array}
$$
Therefore, taking $\Psi(t)(x)=\psi(t)\xi(x)$, where $\psi\in   C_c^\infty(0,T)$ and $\xi\in L^p(\Omega,\nu)$, we obtain that
$$
\displaystyle \int_0^T  v(t)(x)\psi'(t)  dt=-\int_0^T \int_\Omega
\Phi(t,x,y)\psi(t)dm_x(y) dt
  - \int_0^T f(t)(x)\psi(t) dt
$$
for $\nu$-a.e. $x\in\Omega_1$.

It follows that
$$
\displaystyle v'(t)(x)  =  \int_\Omega
\Phi(t,x,y)dm_x(y)     +   f(t)(x) \quad\hbox{for a.e. $t\in (0,T)$ and $\nu$-a.e. $x\in\Omega_1$}.
$$
Therefore, since $v\in C([0,T];L^1(\Omega_1,\nu))$, $\Phi\in L^{p'}(0,T; L^{p'}(\Omega\times \Omega,\nu\otimes m_x))$ and $f\in L^{p'}(0,T;L^{p'}(\Omega_1,\nu))$, we get that $v'\in  L^{p'}(0,T;L^{p'}(\Omega_1,\nu))$ and $v\in W^{1,1}(0,T;L^{1}(\Omega_1,\nu))$.

Then, by Remark~\ref{macre01}, we conclude that the mild solution $v$ is, in fact, a strong solution (see~\cite{Benilantesis} or~\cite[Corollary A.34]{ElLibro}). Hence,
\begin{equation}\label{newweakderivative}
\displaystyle v'(t)(x)  -  \int_\Omega
\a_p(x,y,u(t)(y)-u(t)(x))dm_x(y)     =   f(t)(x) \quad\hbox{for a.e. $t\in (0,T)$ and $\nu$-a.e. $x\in\Omega_1$}.
\end{equation}

Let us see, for further use,  that $\displaystyle\int_{\Omega_1}j_{\gamma}^*(v(t))d\nu\in W^{1,1}(0,T)$.  By \eqref{newjstarinequality} and Fatou's lemma, we have
$$
\begin{array}{l}
\displaystyle\limsup_n\frac{1}{2}\int_0^T\int_\Omega \int_\Omega
\a_p(x,y,u_n(t)(y)-u_n(t)(x))(u_n(t)(y)-u_n(t)(x))dm_x(y)d\nu(x)dt\\[14pt]
\displaystyle \le -\int_{\Omega_1}(j_\gamma^*(v(T)(x))- j_\gamma^*(v(0)(x)))d\nu(x) +\int_0^T\int_{\Omega_1}f(t)(x)u(t)(x) d\nu(x)dt.
\end{array}
$$
Moreover, by \eqref{newweakderivative},
\begin{equation}\label{newDefF}
\int_0^T  v(t)(x)\frac{d}{dt}\Psi(t)(x)dt=\int_0^T F(t)(x)\Psi(t)(x)dt,
\end{equation}
where $F$ is given by
$$
F(t)(x)=- \int_\Omega\a_p(x,y,u(t)(y)-u(t)(x))dm_x(y)-f(t)(x), \ \ \ x\in\Omega_1 .
$$
Let $\psi\in   C_c^\infty(0,T)$, $\psi\ge 0$, $\tau>0$ and
$$\eta_\tau (t)(x):=\frac{1}{\tau}\int_t^{t+\tau} u(s)(x) \psi(s)ds, \ \ \ t\in[0,T], \, x\in\Omega_1.$$
Then, for $\tau$ small enough,  $\eta_\tau\in W^{1,1}_0(0,T;L^p(\Omega_1,\nu))$ and we may use it as a test function in \eqref{newDefF} to obtain
$$
\begin{array}{rl}
\displaystyle\int_0^T \int_{\Omega_1} F(t)(x)\eta_\tau(t)(x)d\nu(x) dt&\displaystyle=\int_0^T \int_{\Omega_1} v(t)(x)\frac{d}{dt}\eta_\tau(t)(x)d\nu(x)dt\\[14pt]
&\displaystyle= \int_0^T \int_{\Omega_1}v(t)(x)\frac{u(t+\tau)(x)\psi(t+\tau)-u(t)(x)\psi(t)}{\tau}d\nu(x) dt \\[14pt]
&\displaystyle=\int_0^T \int_{\Omega_1}\frac{v(t-\tau)(x)-v(t)(x)}{\tau} u(t)(x)\psi(t) d\nu(x)dt.
\end{array}$$
Now, since
$$\gamma^{-1}(r)= \partial j_{\gamma^{-1}}(r)=\partial\left(\int_0^r (\gamma^{-1})^0(s) ds\right),$$
$$
\begin{array}{rl}
\displaystyle\int_0^T \int_{\Omega_1}F(t)(x)\eta_\tau(t)(x) d\nu(x)dt &\displaystyle\le \frac{1}{\tau}\int_0^T\int_{\Omega_1}\int_{v(t)(x)}^{v(t-\tau)(x)}(\gamma^{-1})^0(s)ds \psi(t)d\nu(x)dt\\[14pt]
&\displaystyle =\int_0^T\int_{\Omega_1}\int_0^{v(t)(x)}(\gamma^{-1})^0(s)ds \frac{\psi(t+\tau)-\psi(t)}{\tau}d\nu(x)dt,
\end{array}$$
which, letting $\tau\to 0^+$ yields
$$
\begin{array}{rl}
\displaystyle\int_0^T \int_{\Omega_1} F(t)u(t)(x)\psi(t)d\nu(x) dt &\displaystyle\le\int_0^T\int_{\Omega_1}\int_0^{v(t)(x)}(\gamma^{-1})^0(s)ds \Psi'(t) d\nu(x) dt\\[14pt]
&\displaystyle=\int_0^T\int_{\Omega_1} j_{\gamma^{-1}}(v(t)(x)) \psi'(t) d\nu(x) dt\\[14pt]
&\displaystyle = \int_0^T \int_{\Omega_1} j_{\gamma}^*(v(t)(x)) \psi'(t) d\nu(x) dt.
\end{array}$$
Taking
$$\widetilde{\eta}_\tau (t)(x)=\frac{1}{\tau}\int_t^{t+\tau} u(s-\tau) \Psi(s)ds, \ \ t\in[0,T], \, x\in\Omega_1,$$
yields the opposite inequalities so that, in fact,
$$
\int_0^T \int_{\Omega_1}F(t)(x)u(t)(x)d\nu(x)\psi(t) dt \displaystyle = \int_0^T \int_{\Omega_1} j_{\gamma}^*(v(t)(x))d\nu(x)\psi'(t) dt,
$$
i.e.,
$$
  -\frac{d}{dt}\int_{\Omega_1} j_{\gamma}^*(v(t)(x))d\nu(x) =\int_{\Omega_1}F(t)(x)u(t)(x)d\nu(x) \ \hbox{ in $\mathcal{D}'(]0,T[)$},
$$
thus, in particular,
\begin{equation}\label{jstarinsobolev}
\int_{\Omega_1}j_{\gamma}^*(v)d\nu\in W^{1,1}(0,T).
\end{equation}
 \noindent {\it Step 2.}
 Suppose now that, either $\mathcal{R}_{\gamma,\beta}^-=-\infty$ and  $\mathcal{R}_{\gamma,\beta}^+<+\infty$, or $\mathcal{R}_{\gamma,\beta}^->-\infty$ and  $\mathcal{R}_{\gamma,\beta}^+=+\infty$. Recall that we are assuming the hypotheses in \eqref{newconditionjstar}   and that $v_0^n=v_0$ for every $n\in\N$.  Suppose first that $\mathcal{R}_{\gamma,\beta}^-=-\infty$ and $\mathcal{R}_{\gamma,\beta}^+<+\infty$. Then, for $k\in \mathbb{N}$, let $\beta^k:\R\rightarrow \R$ be the following maximal monotone graph
$$\beta^k(r):=\left\{\begin{array}{l}
\beta(r)\quad \hbox{if }r<k,\\[6pt]
[\beta^0(k),\mathfrak{B}^+]\quad\hbox{if }r=k,\\[6pt]
\mathfrak{B}^++r-k\quad\hbox{if }r>k.
\end{array}\right.
$$
We have that $\beta^k\to\beta$ in the sense of maximal monotone graphs. Indeed, given $\lambda>0$ and $s\in \R$ there exists $r\in \R$ such that $s\in r+\lambda\beta(r)$ thus, for $k>r$,  $s\in r+\lambda\beta(r)= r+\lambda\beta^k(r)$, i.e., $r=(I+\lambda\beta)^{-1}(s)=(I+\lambda\beta^k)^{-1}(s)$.

  By Step 1 we know that, since $\mathcal{R}_{\gamma,\beta^k}^-=-\infty$ and $\mathcal{R}_{\gamma,\beta^k}^+=+\infty$, there exists a strong solution $v_k\in W^{1,1}(0,T;L^{1}(\Omega_1,\nu))$  of Problem   $\left(DP_{f-\frac{1}{k},v_0}^{ \a_p,\gamma,\beta^k}\right)$, therefore, there exist  $u_k\in L^p(0,T;L^p(\Omega,\nu))$ and $w_k\in L^{p'}(0,T;L^{p'}(\Omega_2,\nu))$  such that
\begin{equation}\label{solprobaprox}
\left\{ \begin{array}{ll} \displaystyle(v_k)_t(t)(x) -   \int_{\Omega} \a_p(x,y,u_k(t)(y)-u_k(t)(x)) dm_x(y)=f(t)(x)-\frac1k,    &x\in  \Omega_1,\ 0<t<T,
\\ [14pt]
\displaystyle w_k(t)(x)-   \int_{\Omega} \a_p(x,y,u_k(t)(y)-u_k(t)(x)) dm_x(y) =0,    &x\in \Omega_2, \  0<t<T,\end{array} \right.
\end{equation}
with $v_k\in\gamma(u_k)$ $\nu$-a.e. in $\Omega_1$ and $w_k\in\beta^k(u_k)$ $\nu$-a.e. in $\Omega_2$. Let us see that
\begin{equation}\label{arras01}u_k\le u_{k+1}, \ \hbox{ $\nu$-a.e. in $\Omega$}, \, k\in\N,
\end{equation}
and
\begin{equation}\label{arras02}v_k\le v_{k+1}, \ \hbox{ $\nu$-a.e. in $\Omega_1$}, \, k\in\N.
\end{equation}
Going back to the construction of the mild solution, in this case of   $\left(DP_{f-\frac{1}{k},v_0}^{ \a_p,\gamma,\beta^k}\right)$, for each step $n\in\N$ and for each $i\in\{1,\ldots, n\}$, we have that there exists $u_{k,i}^n\in L^p(\Omega,\nu)$, $v_{k,i}^n\in L^{p'}(\Omega_1,\nu)$ and $w_{k,i}^n\in L^{p'}(\Omega_2,\nu)$ such that
$$
\left\{\begin{array}{l}
\displaystyle v_{k,i}^n(x)-\frac{T}{n}\int_\Omega
\a_p(x,y,u_{k,i}^n(y)-u_{k,i}^n(x))dm_x(y)=\frac{T}{n} \left(f_i^n(x)-\frac{1}{k}\right)+v_{k,i-1}^n(x), \ x\in\Omega_1\\ [14pt]
\displaystyle   w_{k,i}^n(x)- \int_\Omega
\a_p(x,y,u_{k,i}^n(y)-u_{k,i}^n(x))dm_x(y)=0, \ x\in\Omega_2,
\end{array}\right.
$$
with $v_{k,i}^n\in\gamma(u_{k,i}^n)$ $\nu$-a.e. in $\Omega_1$ and $w_{k,i}^n\in\beta^k(u_{k,i}^n)$ $\nu$-a.e. in $\Omega_2$.
Let
$$z_{k,i}^n:=\left\{\begin{array}{l}
w_{k+1,i}^n\quad\hbox{if }u_{k+1,i}^n< k,\\[6pt]
\mathfrak{B}^+\quad\hbox{if }u_{k+1,i}^n=k,\\[6pt]
\beta^k(u_{k+1,i}^n)\quad\hbox{if } u_{k+1,i}^n> k,
\end{array}
\right.
$$
for $n\in\N$ and $i\in\{1,\ldots,n\}$ (observe that $\beta^k(r)$ is single-valued for $r>k$ and coincides with $\beta^{k+1}(r)=\beta(r)$ for $r<k$). It is clear that $z_{k,i}^n\in \beta^k(u_{k+1,i}^n)$ and, since $\beta^k\ge \beta^{k+1}$, $z_{k,i}^n\ge w_{k+1,i}^n$. Then,
$$
\left\{\begin{array}{l}
\displaystyle v_{k+1,1}^n(x)-\frac{T}{n}\int_\Omega
\a_p(x,y,u_{k+1,1}^n(y)-u_{k+1,1}^n(x))dm_x(y)=\frac{T}{n} \left(f_1^n(x)-\frac{1}{k+1}\right)+v_{0}(x)\\ [14pt]
\displaystyle > \frac{T}{n} \left(f_1^n(x)-\frac{1}{k}\right)+v_{0}(x)=v_{k,1}^n(x)-\frac{T}{n}\int_\Omega
\a_p(x,y,u_{k,1}^n(y)-u_{k,1}^n(x))dm_x(y)
, \ \ x\in\Omega_1\\
[14pt]
\displaystyle z_{k,1}^n(x)- \int_\Omega
\a_p(x,y,u_{k+1,1}^n(y)-u_{k+1,i}^n(x))dm_x(y) \\ [14pt]
\displaystyle \ge  w_{k+1,1}^n(x)- \int_\Omega
\a_p(x,y,u_{k+1,1}^n(y)-u_{k+1,1}^n(x))dm_x(y)
\\ [14pt]
\displaystyle =0=  w_{k,1}^n(x)- \int_\Omega
\a_p(x,y,u_{k,1}^n(y)-u_{k,1}^n(x))dm_x(y)  , \ \ x\in\Omega_2,
\end{array}\right.
$$
for $n\in\N$. Hence, by the maximum principle (Theorem \ref{maxandcont01}),
$$v_{k,1}^n\le v_{k+1,1}^n \ \ \hbox{
and } \ \  u_{k,1}^n\le u_{k+1,1}^n \ \ \nu\hbox{-a.e.}$$
Proceeding in the same way we get that
$$v_{k,i}^n\le v_{k+1,i}^n\ \ \hbox{
and } \ \  u_{k,i}^n\le u_{k+1,i}^n \ \ \nu\hbox{-a.e.}$$
for each $n\in\N$ and $i\in\{1,\ldots, n\}$. From here we get~\eqref{arras01} and~\eqref{arras02}.

  Since $\gamma^{-1}(r)= \partial j_{\gamma}^*(r)$ and $u_k(t)\in \gamma^{-1}(v_k(t))$ $\nu$-a.e. in $\Omega_1$, we have
$$\int_{\Omega_1}(v_k(t-\tau)(x)-v_k(t)(x))u_k(t)(x)d\nu(x)\le\int_{\Omega_1}j_\gamma^*(v_k(t-\tau)(x))-j_\gamma^*(v_k(t)(x))d\nu(x).$$
Integrating this equation over $[0,T]$, dividing by $\tau$, letting $\tau\to 0^+$ and recalling that, by \eqref{jstarinsobolev}, $\displaystyle\int_{\Omega_1}j^*(v_k)d\nu\in W^{1,1}(0,T)$, we get
$$\begin{array}{rl}
\displaystyle-\int_0^T\int_{\Omega_1}(v_k)_t(t)(x)u_k(t)(x)d\nu(x)dt& \displaystyle \le\int_{\Omega_1}j^*(v(0)(x))-j^*(v_k(T)(x))d\nu(x)\\ [14pt]
&\displaystyle \le \int_{\Omega_1}j^*(v(0)(x))d\nu(x).
\end{array}$$
Therefore, multiplying \eqref{solprobaprox} by $u_k$ and integrating with respect to $\nu$ we get
$$\begin{array}{l}
\displaystyle\frac{1}{2}\int_0^T\int_\Omega\int_{\Omega} \a_p(x,y,u_k(t,y)-u_k(t)(x))(u_k(t)(y)-u_k(t)(x)) dm_x(y)d\nu(x)dt\\ [14pt]
\displaystyle\le \int_0^T\int_{\Omega_1}\left(f(t)(x)-\frac{1}{k}\right)u_k(t)(x)d\nu(x)dt+\int_{\Omega_1}j^*(v(0)(x))d\nu(x).
\end{array}$$
 Now, working as in the previous step, since $\Gamma^{+}<\infty$, we get that $ \displaystyle\left\{\Vert u_k \Vert_{L^{p}(0,T;L^p(\Omega,\nu))}^{p}\right\}_k$ is bounded. Then, by the monotone convergence theorem we get that there exists $u\in L^p(0,T;L^p(\Omega,\nu))$ such that $u_k\stackrel{k}{\longrightarrow}u$ in $L^p(0,T;L^p(\Omega,\nu))$. From this we get, by \cite[Lemma G]{BCrS}, that $v(t)(x)\in \gamma(u(t)(x))$ for a.e. $t\in[0,T]$ and $\nu$-a.e. $x\in\Omega_1$.

Therefore, from \eqref{solprobaprox}   and Lemma \ref{convergenciaap} (note that, by the monotonicity of $\{u_k\}$,  $|u_k|\le \max\{|u_1|,|u|\}\in L^p(\Omega,\nu)$), we get that $(v_k)_t$ converges strongly in $L^{p'}(0,T;L^{p'}(\Omega_1,\nu))$ and $w_k$ converges strongly in   $L^{p'}(0,T;$ $L^{p'}(\Omega_2,\nu))$. In particular, $v\in W^{1,1}(0,T;L^1(\Omega_1,\nu))$, $w(t)(x)\in\beta(u(t)(x))$ for a.e. $t\in[0,T]$ and $\nu$-a.e. $x\in\Omega_2$, and
$$
\left\{ \begin{array}{ll} \displaystyle v_t(t)(x) -   \int_{\Omega} \a_p(x,y,u(t)(y)-u(t)(x)) dm_x(y)=f(t)(x),    &x\in  \Omega_1,\ 0<t<T,
\\ [14pt]
\displaystyle w(t)(x)-   \int_{\Omega} \a_p(x,y,u(t)(y)-u(t)(x)) dm_x(y) =0,    &x\in \Omega_2, \  0<t<T.\end{array} \right.
$$

The case $\mathcal{R}_{\gamma,\beta}^->-\infty$ and $\mathcal{R}_{\gamma,\beta}^+=+\infty$ follows similarly by taking
$$\widetilde \beta^k:=\left\{\begin{array}{l}
\mathfrak{B}^-+r+k\quad\hbox{if }r<-k,\\[6pt]
[\mathfrak{B}^-,\beta^0(-k)]\quad\hbox{if }r=-k,\\[6pt]
\beta(r)\quad \hbox{if }r>-k.
\end{array}\right.
$$
instead of $\beta^k$, $k\in\N$.

 \noindent{\it Step 3.}
Finally, assume that both $\mathcal{R}_{\gamma,\beta}^-$   and $\mathcal{R}_{\gamma,\beta}^+$ are finite. We define, for $k\in \mathbb{N}$,
$$\widetilde \beta^k:=\left\{\begin{array}{l}
\mathfrak{B}^-+r+k\quad\hbox{if }r<-k,\\[6pt]
[\mathfrak{B}^-,\beta^0(-k)]\quad\hbox{if }r=-k,\\[6pt]
\beta(r)\quad \hbox{if }r>-k.
\end{array}\right.
$$
By the previous step we have that, for $k$ large enough such that $f+\frac{1}{k}$ satisfies
$$\int_{\Omega_1}v_0^+d\nu+\int_0^T\int_{\Omega_1}\left(f(s)^++\frac{1}{k}\right)d\nu ds< \nu(\Omega_1)\Gamma^+,$$     there exists a strong solution $v_k\in W^{1,1}(0,T;L^{1}(\Omega_1,\nu))$
of Problem   $\left(DP_{f+\frac{1}{k},v_0}^{ \a_p,\gamma,\widetilde\beta^k}\right)$,
i.e., there exist  $u_k\in L^p(0,T;L^p(\Omega,\nu))$ and $w_k\in L^{p'}(0,T;L^{p'}(\Omega_2,\nu))$  such that
$$
\left\{ \begin{array}{ll} \displaystyle(v_k)_t(t)(x) -   \int_{\Omega} \a_p(x,y,u_k(t)(y)-u_k(t)(x)) dm_x(y)=f(t)(x)+\frac1k,    &x\in  \Omega_1,\ 0<t<T,
\\ [14pt]
\displaystyle w_k(t)(x)-   \int_{\Omega} \a_p(x,y,u_k(t)(y)-u_k(t)(x)) dm_x(y) =0,    &x\in \Omega_2, \  0<t<T,\end{array} \right.
$$
with $v_k\in\gamma(u_k)$ $\nu$-a.e. in $\Omega_1$ and $w_k\in\tilde\beta^k(u_k)$ $\nu$-a.e. in $\Omega_2$.

 Going back to the construction of the mild solution, in this case of  $\left(DP_{f+\frac{1}{k},v_0}^{ \a_p,\gamma,\widetilde\beta^k}\right)$, for each step $n\in\N$ and for each $i\in\{1,\ldots, n\}$, we have that there exists $u_{k,i}^n\in L^p(\Omega,\nu)$, $v_{k,i}^n\in L^{p'}(\Omega_1,\nu)$ and $w_{k,i}^n\in L^{p'}(\Omega_2,\nu)$ such that
$$
\left\{\begin{array}{ll}
\displaystyle v_{k,i}^n(x)-\frac{T}{n}\int_\Omega
\a_p(x,y,u_{k,i}^n(y)-u_{k,i}^n(x))dm_x(y)=\frac{T}{n} \left(f_i^n(x)+\frac{1}{k}\right)+v_{k,i-1}^n(x), &x\in\Omega_1\\ [14pt]
\displaystyle   w_{k,i}^n(x)- \int_\Omega
\a_p(x,y,u_{k,i}^n(y)-u_{k,i}^n(x))dm_x(y)=0, &x\in\Omega_2,
\end{array}\right.
$$
where $v_{k,i}^n\in\gamma(u_{k,i}^n)$ $\nu$-a.e. in $\Omega_1$ and $w_{k,i}^n\in\widetilde\beta^k(u_{k,i}^n)$ $\nu$-a.e. in $\Omega_2$.
Let
$$z_{k,i}^n:=\left\{\begin{array}{l}
w_{k+1,i}^n\quad\hbox{if }u_{k+1,i}^n>- k,\\[6pt]
\mathfrak{B}^-\quad\hbox{if }u_{k+1,i}^n=-k,\\[6pt]
\widetilde\beta^k(u_{k+1,i}^n)\quad\hbox{if } u_{k+1,i}^n<- k,
\end{array}
\right.
$$
for $n\in\N$ and $i\in\{1,\ldots,n\}$ (observe that $\widetilde\beta^k(r)$ is single-valued for $r<-k$ and coincides with $\widetilde \beta^{k+1}(r)=\beta(r)$ for $r>-k$). It is clear that $z_{k,i}^n\in \widetilde\beta^k(u_{k+1,i}^n)$ and, since $\widetilde\beta^k\le \widetilde\beta^{k+1}$, we have that $z_{k,i}^n\le w_{k+1,i}^n$, $i\in\{1,\ldots,n\}$. Then,
$$
\left\{\begin{array}{l}
\displaystyle v_{k+1,1}^n(x)-\frac{T}{n}\int_\Omega
\a_p(x,y,u_{k+1,1}^n(y)-u_{k+1,1}^n(x))dm_x(y)=\frac{T}{n} \left(f_1^n(x)+\frac{1}{k+1}\right)+v_{0}^n(x)\\ [14pt]
\displaystyle < \frac{T}{n} \left(f_1^n(x)+\frac{1}{k}\right)+v_{0}^n(x)=v_{k,1}^n(x)-\frac{T}{n}\int_\Omega
\a_p(x,y,u_{k,1}^n(y)-u_{k,1}^n(x))dm_x(y)
, \ \ x\in\Omega_1\\ [14pt]
\displaystyle z_{k,1}^n(x)- \int_\Omega
\a_p(x,y,u_{k+1,1}^n(y)-u_{k+1,i}^n(x))dm_x(y) \\ [14pt]
\displaystyle \le  w_{k+1,1}^n(x)- \int_\Omega
\a_p(x,y,u_{k+1,1}^n(y)-u_{k+1,1}^n(x))dm_x(y)
\\ [14pt]
\displaystyle =0=  w_{k,1}^n(x)- \int_\Omega
\a_p(x,y,u_{k,1}^n(y)-u_{k,1}^n(x))dm_x(y)  , \ \ x\in\Omega_2,
\end{array}\right.
$$
for $n\in\N$. Hence, by the maximum principle (Theorem \ref{maxandcont01}),
$$v_{k,1}^n\ge v_{k+1,1}^n\ \ \hbox{ and }\ \ u_{k,1}^n\ge u_{k+1,1}^n \ \ \nu\hbox{-a.e.}.$$
Proceeding in the same way we get that, for $n\in\N$ and $i\in\{1,\ldots,n\}$,
$$v_{k,i}^n\ge v_{k+1,i}^n\ \ \hbox{ and }\ \  u_{k,i}^n\ge u_{k+1,i}^n \ \ \nu\hbox{-a.e.} .$$
Therefore,
$$
u_k\ge u_{k+1}, \ \hbox{ $\nu$-a.e. in $\Omega$}, \, k\in\N,
$$
and
$$
v_k\ge v_{k+1}, \ \hbox{ $\nu$-a.e. in $\Omega_1$}, \, k\in\N.
$$

  We can now conclude, as in the previous step, that
$$\int_0^T\Vert u_k^-(t)\Vert_{L^{p}(\Omega_1,\nu)}dt\le  K_5\left(\int_0^T\left(\int_{\Omega_1} \int_{\Omega_1}
|u_k^-(t)(y)-u_k^-(t)(x)|^p dm_x(y)d\nu(x)\right)^\frac1p dt+1 \right) $$
for some constant $  K_5>0$.
  Moreover, by the monotonicity of $\{u_k\}$, we get that $\displaystyle \Big\{\int_0^T\Vert u_k^+(t)\Vert_{L^{p}(\Omega_1,\nu)}dt\Big\}_k$ is bounded.  From this point we can finish the proof as in the previous step. \qed
\end{proof}

\appendix
\renewcommand{\theequation}{A.\arabic{equation}}

  \section{Poincar\'e type inequalities}\label{secineqpoin}

In order to prove the results on existence of solutions of our problems, we have  assumed that appropriate Poincar\'e type inequalities hold. In \cite[Corollary 31]{O}, it is proved that a Poincar\'e type inequality holds on metric random walk spaces (with an invariant measure) with positive coarse Ricci curvature. Under some conditions relating the random walk and the invariant measure some Poincar\'{e} type inequalities are given in \cite[Theorem 4.5]{MST2} (see also~\cite{ElLibro} and~\cite{MST4}). Here we generalise some of these results.

  \begin{definition}\label{poincareatappendixdef}   Let $[X,d,m]$ be a metric random walk space with reversible measure $\nu$ and let $A, B\subset X$ be disjoint measurable sets such that $\nu(A)>0$. Let $Q:=((A\cup B)\times (A\cup B))\setminus (B\times B)$. We say that $[X,d,m]$ satisfies a {\it generalised $(q,p)$-Poincar\'{e} type inequality ($p, q\in[1,+ \infty[$)} on $(A,B)$ (with respect to $\nu$), if, given $0<l\le \nu(A\cup B)$,   there exists a constant $\Lambda>0$ such that,  for any $u \in L^q(A\cup B,\nu)$ and any measurable set $Z\subset A\cup B$ with $\nu(Z)\ge l$,
$$  \left\Vert  u \right\Vert_{L^q(A\cup B,\nu)}  \leq \Lambda\left(\left(\int_{Q} |u(y)-u(x)|^p dm_x(y) d\nu(x) \right)^{\frac1p}+\left| \int_Z u\,d\nu\right|\right).
$$
\end{definition}
In Subsection~\ref{efzly} (Assumption \ref{assumption3}) we have used that the metric random walk space satisfies a generalised $(p,p)$-Poincar\'{e} type inequality on $(\Omega_1\cup\Omega_2,\emptyset)$. This assumption holds true in many important examples, as the next results show.

\begin{lemma}\label{lemmaPoincare}
Let $[X,d,m]$ be a metric random walk space with reversible measure $\nu$ with respect to $m$. Let $A, B\subset X$ be disjoint measurable sets such that $B\subset \partial_m A$, $\nu(A)>0$ and  $A$ is $m$-connected. Suppose that $\nu(A\cup B)<+\infty$ and that
$$
  \nu\left(\left\{ x\in A\cup B \, : \, (m_x\res A)\perp (\nu\res A)\right\}\right)=0.
$$
Let    $q\ge  1$.  Let
$\{u_n\}_n\subset L^q(A\cup B,\nu)$ be a bounded sequence in $L^1(A\cup B,\nu)$     satisfying
\begin{equation}\label{003}
\lim_n \int_{Q}|u_n(y)-u_n(x)|^q dm_x(y)d\nu(x)= 0
\end{equation}
where, as before, $Q=((A\cup B)\times(A\cup B))\setminus(B\times B)$.
Then, there exists $\lambda\in \mathbb{R}$ such that
$$
 u_n(x) \to \lambda\quad\hbox{for } \nu\hbox{-a.e. } x\in A\cup B,
$$
$$
\Vert u_n-\lambda\Vert_{L^q (A,m_x)}\to 0\quad\hbox{for } \nu\hbox{-a.e. } x\in A\cup B,
$$
and
$$
\Vert u_n-\lambda\Vert_{L^q (A\cup B,m_x)}\to 0\quad\hbox{for } \nu\hbox{-a.e. } x\in A.
$$

\end{lemma}

\begin{proof}
  If $B=\emptyset$ (or $\nu$-null) one can skip some steps in the proof.
Let $$F_n(x,y)=|u_n(y)-u_n(x)|, \quad (x,y)\in Q,$$
$$f_n(x)=\int_{A} |u_n(y)-u_n(x)|^q\, dm_x(y), \quad x\in A\cup B,$$
and
$$g_n(x)=\int_{A\cup B} |u_n(y)-u_n(x)|^q\, dm_x(y), \quad x\in A .$$
  Let
$$\mathcal{N}_\perp:=\left\{ x\in A\cup B \, : \, (m_x\res A)\perp (\nu\res A)\right\}.$$
From
\eqref{003}, it follows that
$$f_n\to 0\quad\hbox{in } L^1(A\cup B,\nu)$$
and
$$g_n\to 0\quad\hbox{in } L^1(A,\nu).$$ Passing to a subsequence if necessary, we can assume that
\begin{equation}\label{005}
f_n(x)\to 0\quad \hbox{for every } x\in (A\cup B)\setminus N_f,\quad \hbox{where } N_f\subset A\cup B \hbox{ is } \nu\hbox{-null}
\end{equation}
and
\begin{equation}\label{005g}
g_n(x)\to 0\quad\hbox{for every } x\in A\setminus N_g,\quad \hbox{where } N_g\subset A \hbox{ is } \nu\hbox{-null}.
\end{equation}
On the other hand, by \eqref{003}, we also have that $$F_n\to 0\quad\hbox{in }
L^q(Q, \nu\otimes m_x).$$ Therefore, we can suppose that, up to a subsequence,
\begin{equation}\label{006}
F_n(x,y)\to 0\quad\hbox{for every } (x,y)\in Q\setminus
C,\quad \hbox{where } C\subset Q\hbox{ is } \nu\otimes m_x\hbox{-null}.
\end{equation}
Let $N_1\subset A$ be a $\nu$-null set satisfying that,
$$  \hbox{ for all $x\in
A\setminus N_1$, the section $C_{x}  := \{ y \in A\cup B  :  (x,y) \in C \}$ of $C$ is
$m_x$-null,}$$
and
$N_2\subset A\cup B$ be a $\nu$-null set satisfying that,
$$  \hbox{for all $x\in
(A\cup B)\setminus N_2$, the section $C_{x}'  := \{ y \in A  :  (x,y) \in C \}$ of $C$ is
$m_x$-null.}
$$

Now, since $A$ is $m$-connected and $B\subset\partial_m A$,
$$D:=\{ x\in A\cup B : m_x(A)=0\}$$
is $\nu$-null. Indeed, by the definition of $D$,  $L_m(A\cap D, A)=0$ thus, in particular, $L_m(A\cap D,A\setminus D)=0$ which, since $A$ is $m$-connected, implies that $\nu(A\cap D)=0$ or $\nu(A\cap D)=\nu(A)$. However, if $\nu(A\cap D)=\nu(A)$ then, for any $E$, $F\subset A$, we have $L_m(E,F)\le L_m(D\cap A,A)=0$ which is a contradiction, thus $\nu(D\cap A)=0$. Now, since $B\subset \partial_m A$, $m_x(A)>0$ for every $x\in B$, thus $\nu(B\cap D)=0$.

 Set   $N:=\mathcal{N}_\perp\cup N_f\cup N_g\cup N_1\cup N_2\cup D$  (note that $\nu(N)=0$). Fix  $x_0\in A\setminus N$. Up to a subsequence,  $u_n(x_0)\to\lambda$ for some $\lambda\in[-\infty,+\infty]$, but then, by \eqref{006}, we also have that $u_n(y)\to\lambda$ for every $y\in (A\cup B)\setminus C_{x_0}$. However,     since $x_0\not\in \mathcal{N}_\perp$ and $m_{x_0}(C_{x_0})=0$, we must have that $\nu(A\setminus C_{x_0})>0$; thus, if
 $$S:=\{x\in A\cup B: u_n(x)\to\lambda\}$$
  then $\nu(S\cap A)\ge\nu(A\setminus C_{x_0})>0$. Note that, if $x\in (A\cap S)\setminus N$ then, by \eqref{006} again, $(A\cup B)\setminus C_x\subset S$ thus $m_x((A\cup B)\setminus S)\le m_x(C_x)=0$; therefore,
 $$
 L_m(A\cap S, (A\cup B)\setminus S)=0.
 $$
 In particular, $L_m(A\cap S, A\setminus S)=0$, but, since $A$ is $m$-connected and $\nu(A\cap S)>0$, we must have $\nu(A\setminus S)=0$, i.e. $\nu(A)=\nu(A\cap S)$.

 Finally, suppose that $\nu(B\setminus S)>0$. Let $x\in B\setminus (S\cup N)$. By \eqref{006},  $A\setminus C_x'\subset A\setminus S$,   i.e., $S\cap A\subset C_x'$, thus $m_x(S\cap A)=0$. Therefore, since $x\not\in \mathcal{N}_\perp$, we must have $\nu(A\setminus S)>0$ which is a contradiction with what we have already obtained.
  Consequently, we have obtained that $u_n$ converges $\nu$-a.e. in $A\cup B$ to $\lambda$:
$$
u_n(x)\to\lambda\quad \hbox{ for every } x\in S,\ \nu((A\cup B)\setminus S)=0.$$

Since $\{\Vert u_n\Vert_{L^1 (A\cup B,\nu)}\}_n$ is bounded, by Fatou's Lemma, we must have that $\lambda \in \R$.
On the other hand, by \eqref{005},
$$F_n(x,\cdot)\to 0\quad\hbox{in } L^{q}(A,m_x)\ ,$$
for every $x\in \Omega\setminus N_f$. In other words, $\Vert u_n(\cdot)-u_n(x)\Vert_{L^q (A, m_x)}\to 0$, thus
$$
\Vert u_n-\lambda\Vert_{L^q (A,m_x)}\to 0\quad\hbox{for $\nu$-a.e. }x\in A\cup B.
$$
Similarly, by \eqref{005g},
$$
\Vert u_n-\lambda\Vert_{L^q (A\cup B,m_x)}\to 0\quad\hbox{for $\nu$-a.e. }x\in A.
$$ \qed
\end{proof}

\begin{theorem}\label{lab197}
  Let $p\ge 1$.    Let $[X,d,m]$ be a metric random walk space with reversible measure $\nu$. Let $A,B\subset X$ be disjoint measurable sets such that $B\subset \partial_m A$, $\nu(A)>0$ and  $A$ is $m$-connected. Suppose that $\nu(A\cup B)<+\infty$ and that
$$    \nu\left(\left\{ x\in A\cup B \, : \, (m_x\res A)\perp (\nu\res A)\right\}\right)=0.
$$
  Assume further that, given a $\nu$-null set $N\subset A$, there exist $x_1,x_2,\ldots, x_L\in A\setminus N$ and a constant $C>0$ such that $\nu\res (A\cup B) \le C(m_{x_1}+\cdots+m_{x_L})\res (A \cup B)$. Then, $[X,d,m]$ satisfies a generalised $(p,p)$-Poincar\'e type inequality on $(A,B)$.
\end{theorem}

\begin{proof}
Let $p\ge 1$ and $0<h\le \nu(A\cup B)$. We want to prove that there exists a constant  $\Lambda>0$  such that
$$  \left\Vert  u \right\Vert_{L^p(A\cup B,\nu)}  \leq \Lambda\left(\left(\int_{Q} |u(y)-u(x)|^p dm_x(y) d\nu(x) \right)^{\frac1p}+\left| \int_{Z} u\,d\nu\right|\right)$$
for every $u \in L^p(A\cup B,\nu)$ and every measurable set $Z\subset A\cup B$ with $\nu(Z)\ge l$. Suppose that this inequality is not satisfied for any $\Lambda$. Then, there exists a sequence $\{u_n\}_{n\in\N}\subset   L^p(A\cup B,\nu)$, with $\Vert u_n\Vert_{L^p (A\cup B,\nu)}=1$, and a sequence of measurable sets $Z_n\subset A\cup B$ with $\nu(Z_n)\ge l$, $n\in\N$,   satisfying
$$
\lim_n \int_{Q}|u_n(y)-u_n(x)|^p dm_x(y)d\nu(x)= 0
$$
and
$$
\lim_n\int_{Z_n} u_n\, d\nu= 0.
$$
Therefore, by Lemma~\ref{lemmaPoincare}, there exist $\lambda\in \mathbb{R}$ and a $\nu$-null set $N\subset A$ such that
$$
\Vert u_n-\lambda\Vert_{L^p (A\cup B,m_x)}\stackrel{n}{\longrightarrow} 0\quad\hbox{for every } x\in A\setminus N.
$$

 Now, by hypothesis, there exist $x_1,x_2,\ldots, x_L\in A\setminus N$ and $C>0$ such that $\nu\res (A\cup B) \le C(m_{x_1}+\cdots+m_{x_L})\res (A \cup B)$. Therefore,
$$
\Vert u_n-\lambda\Vert^p_{L^p (A\cup B,\nu)}\le C\sum_{i=1}^L \Vert u_n-\lambda\Vert^p_{L^p (  A\cup B,m_{x_i})}\stackrel{n}{\longrightarrow} 0.
$$
Moreover, since $\{\1_{Z_n}\}_n$ is bounded in $L^{p'}(A\cup B, \nu)$, there exists $\phi\in L^{p'}(A\cup B,\nu)$ such that, up to a subsequence, $\1_{Z_n}\rightharpoonup \phi$ weakly in $L^{p'}(A\cup B,\nu)$ (weakly-$\ast$ in $L^\infty(A\cup B,\nu)$ in the case $p=1$). In addition, $\phi\ge 0$ $\nu$-a.e. in $A\cup B$ and
$$0<l\le\lim_{n\to+\infty}\nu(Z_n)=\lim_{n\to +\infty}\int_{A\cup B} \1_{Z_n}d\nu=\int_{A\cup B}\phi d\nu .$$
Then, since $u_n\stackrel{n}{\longrightarrow}\lambda$ in $L^p(A\cup B,\nu)$ and $\1_{Z_n}\stackrel{n}{\rightharpoonup} \phi$ weakly in $L^{p'}(A\cup B,\nu)$ (weakly-$\ast$ in $L^\infty(A\cup B,\nu)$ in the case $p=1$),
 $$0=\lim_{n\to+\infty}\int_{Z_n}u_n=\lim_{n\to+\infty}\int_{A\cup B}\1_{Z_n}u_n=\lambda\int_{A\cup B}\phi d\nu,$$
thus $\lambda=0$. This is a contradiction with  $||u_n||_{L^p (A\cup B,\nu)}=1$, $n\in\N$, since $u_n\stackrel{n}{\longrightarrow} \lambda$ in $L^p(A\cup B,\nu)$, so the theorem is proved. \qed
\end{proof}

\begin{remark}
If  $\Omega:=\Omega_1\cup\Omega_2$ is $m$-connected we can apply the theorem with $A:=\Omega$ and $B=\emptyset$ to obtain the generalised Poincar\'e type inequality used in Subsection~\ref{efzly} (Assumption \ref{assumption3}).

We can take $A=X$, $B=\emptyset$ and $Z=X$ in the theorem to obtain \cite[Theorem 4.5]{MST2}.
\end{remark}

\begin{remark}\label{comrem23}
The assumption that, given a $\nu$-null set $N\subset A$, there exist $x_1,x_2,\ldots, x_L\in A\setminus N$ and $C>0$ such that $\nu\res (A\cup B)\le C(m_{x_1}+\cdots+m_{x_L})\res (A\cup B)$ is not as strong as it seems.   Indeed, this is trivially satisfied by connected locally finite weighted discrete graphs and is also satisfied by $[\R^N,d,m^J]$ (recall Examples \ref{ejem01} and \ref{ejem02}) if, for a domain $A\subset \R^N$, we take $B\subset \partial_{m^J} A$ such that $dist(B,\R^N\setminus A_{m^J})>0$. Moreover, in the following example we see that if we remove this hypothesis then the statement is not true in general.

Consider the metric random walk space $[\R,d,m^J]$ where $d$ is the Euclidean distance and $J:=\frac12\1_{[-1,1]}$ (recall Example \ref{ejem02}). Let $A:=[-1,1]$ and $B:=\partial_{m^J}A =[-2,2]\setminus A$. Then, if $N=\{-1,1\}$ we may not find points satisfying the aforementioned assumption. In fact, the  statement of the theorem does not hold for any $p>1$ as can be seen by taking $u_n:=\frac{1}{2}n^{\frac{1}{p}}\left(\1_{[-2,-2+\frac1n]}-\1_{[2-\frac1n,2]}\right)$ and $Z:=A\cup B$. Indeed, first note that $\Vert u_n\Vert_{ L^p([-2,2],\nu)}=1$ and $\int_{[-2,2]}u_nd\nu=0$ for every $n\in\N$. Now, $\hbox{supp}(m^J_x)=[x-1,x+1]$ for $x\in [-1,1]$ and, therefore,
$$\begin{array}{rl}
\displaystyle\int_{[-2,2]} |u_n(y)-u_n(x)|^p dm^J_x(y)&\displaystyle =\int_{[-2,-2+\frac1n]\cap [x-1,x+1]}n \, dm^J_x(y)+\int_{[2-\frac1n,2]\cap [x-1,x+1]}n \, dm^J_x(y)\\ [14pt]
& \displaystyle=2n\1_{[1-\frac1n,1]}(x)\int_{[2-\frac1n,x+1]}dm^J_x(y)\\ [14pt]
&\displaystyle=2n\left(x-1+\frac1n\right)\1_{[1-\frac1n,1]}(x)
\end{array}$$
for $x\in [-1,1]$. Consequently,
$$\begin{array}{rl}
\displaystyle\int_{[-1,1]}\int_{[-2,2]}|u_n(y)-u_n(x)|^p dm_x^J(y)d\mathcal{L}^1(x)&\displaystyle= 2n\int_{[1-\frac1n,1]}\left(x-1+\frac1n\right) d\mathcal{L}^1(x) \\ [14pt]
&\displaystyle=2n\left(\frac{1}{2}-\frac{(1-\frac1n)^2}{2}-\frac1n+\frac{1}{n^2}\right)=\frac{1}{n}. \end{array}$$
Finally, by the reversibility of $\mathcal{L}^1$ with respect to $m^J$,
$$\int_{[-2,2]}\int_{[-1,1]}|u_n(y)-u_n(x)|^p dm_x^J(y)d\mathcal{L}^1(x)=\frac{1}{n},$$
thus
$$\int_{\left([-2,2]\times[-2,2]\right)\setminus \left(\left([-2,-1]\cup[1,2]\right)\times\left([-2,-1]\cup[1,2]\right)\right)}|u_n(y)-u_n(x)|^p dm_x^J(y)d\mathcal{L}^1(x)\le \frac{2}{n}\stackrel{n}{\longrightarrow}0.$$

However, in this example, as we mentioned before, we can take $B\subset\partial_m A$ such that $\mbox{dist}(B,\R\setminus [-2,2])>0$ to avoid this problem and to ensure that the hypotheses of the theorem are satisfied so that $(A,B)$ satisfies a generalised $(p,p)$-Poincar\'e type inequality.
\end{remark}

In the following example, the metric random walk space $[X,d,m]$ that is defined, together with  the invariant measure $\nu$, satisfies that $m_x\perp \nu$ for every $x\in X$,  and a Poincar\'e type inequality does not hold.

\begin{example}
Let $p>1$. Let $S^1=\{e^{2\pi i \alpha} \, : \, \alpha\in [0,1)\}$ and let $T_{\theta}:S^1\longrightarrow S^1$ denote the irrational rotation map $T_{\theta}(x)=xe^{2\pi i \theta}$ where $\theta$ is an irrational number. On $S^1$ consider the Borel $\sigma$-algebra $\mathcal{B}$ and the $1$-dimensional Hausdorff measure $\nu:=\mathcal{H}_1\res{S^1}$. It is well known that $T_{\theta}$ is a uniquely ergodic measure-preserving transformation on $(S^1,\mathcal{B},\nu)$.

Now, denote $X:=S^1$ and let $m_x:=\frac12 \delta_{T_{-\theta} (x)}+\frac12 \delta_{T_{\theta}(x)}$, $x\in X$. Then $\nu$ is reversible with respect to the metric random walk space $[X,d,m]$, where $d$ is the metric given by the arclength. Indeed, let $f\in L^1(X\times X, \nu\otimes\nu)$, then
$$\begin{array}{rl}
\displaystyle\int_{S^1}\int_{S^1} f(x,y)dm_x(y)d\nu(x)&\displaystyle=\frac12 \int_{S^1}f(x,T_{-\theta}(x))d\nu(x)+\frac12 \int_{S^1}f(x,T_{\theta}(x))d\nu(x)\\ [14pt]
&\displaystyle=\frac12 \int_{S^1}f(T_{\theta}(x),x)d\nu(x)+\frac12 \int_{S^1}f(T_{-\theta}(x),x)d\nu(x)\\ [14pt]
&\displaystyle=\int_{S^1}\int_{S^1} f(y,x)dm_x(y)d\nu(x).
\end{array}$$
Let us see that this space is $m$-connected. First note that, for $x\in X$,
$$m_x^{\ast 2}:=\frac{1}{2}\delta_x + \frac14 \delta_{T^2_{-\theta} (x)}+\frac14 \delta_{T^2_{\theta}(x)}\ge \frac14 \delta_{T^2_{\theta}(x)}$$
and, by induction, it is easy to see that
$$m_x^{\ast n}\ge \frac{1}{2^n}\delta_{T^n_{\theta}(x)}\ .$$
Here, $m_x^{\ast n}$, $n\in\N$, is defined inductively as follows (see \cite{MST0})
$$
\displaystyle
dm_x^{*n}(y):= \int_{z \in X}  dm_z(y)dm_x^{*(n-1)}(z).
$$
Now, let $A\subset X$ such that $\nu(A)>0$. By the pointwise ergodic theorem,
$$\lim_{n\to+\infty}\frac1n \sum_{k=0}^{n-1}\1_A\left(T_\theta^k(x)\right)=\frac{\nu(A)}{\nu(X)}>0$$
for $\nu$-a.e. $x\in X$. Consequently, for $\nu$-a.e. $x\in X$, there exists $n\in\N$ such that $$\1_A\left(T_\theta^n(x)\right)=\delta_{T_\theta^n(x)}(A)>0$$
thus $\nu\left(\left\{x\in X\, : \, m_x^{\ast n}(A)=0 \ \hbox{ for every } n\in \N\right\}\right)=0$. Then, according to \cite[Definition 2.8]{MST0} (see also \cite[Proposition 2.11]{MST0}),  $[X,d,m]$ with the invariant measure $\nu$ for $m$ is $m$-connected.

Let us see that $[X,d,m,\nu]$ does not satisfy a $(p,p)$-Poincar\'e type inequality. For $n\in\N$ let
$$I_k^n:=\left\{ e^{2\pi i \alpha} \, : \, k\theta-\delta(n) < \alpha < k\theta+\delta(n)\right\}, \ \hbox{ $-1\le k \le 2n$,}$$
where $\delta(n)>0$ is chosen so that
$$I_{k_1}^n\cap I_{k_2}^n=\emptyset \ \hbox{ for every $-1\le k_1, k_2 \le 2n$, $k_1\neq k_2$}$$
(note that $e^{2\pi i (k_1\theta-\delta(n))}\neq e^{2\pi i (k_2\theta-\delta(n))}$ for every $k_1\neq k_2$ since $T_\theta$ is ergodic).
Consider the following sequence of functions:
$$u_n:=\sum_{k=0}^{n-1}\1_{I_k^n}-\sum_{k=n}^{2n-1}\1_{I_k^n}, \quad n\in\N.$$
Then,
$$\int_X u_n d\nu =0 \quad \hbox{ for every $n\in \N$,} $$
and
$$\int_X |u_n|^p d\nu =4n\delta(n) \quad \hbox{ for every $n\in \N$.} $$
Fix $n\in\N$, let us see what happens with
$$\int_X\int_X |u_n(y)-u_n(x)|^p dm_x(y)d\nu(x).$$
If $1\le k\le n-2$ or $n+1\le k\le 2n-2$ and $x\in I_k^n$ then
$$\int_X |u_n(y)-u_n(x)|^p dm_x(y)=\frac12|u_n(T_{-\theta}(x))-u_n(x)|^p +\frac12|u_n(T_\theta(x))-u_n(x)|^p =0$$
since $T_{-\theta}(x)\in I_{k-1}^n$ and $T_{\theta}(x)\in I_{k+1}^n$.
Now, if $x\in I_0^n$ then $T_{-\theta}(x)\in I_{-1}^n$ thus
$$\frac12|u_n(T_{-\theta}(x))-u_n(x)|^p +\frac12|u_n(T_\theta(x))-u_n(x)|^p =\frac12 |-1|^p=\frac12$$
and the same holds if $x\in I_{2n-1}^n$ (then $T_{\theta}(x)\in I_{2n}^n$). For $x\in I_{n-1}$ we have $T_{\theta}(x)\in I_{n}^n$ thus
$$\frac12|u_n(T_{-\theta}(x))-u_n(x)|^p +\frac12|u_n(T_\theta(x))-u_n(x)|^p =\frac12 |-1-1|^p=2^{p-1}$$
and the same result is obtained for $x\in I_{n+1}^n$. Similarly, if $x\in I_{-1}^n$ or $x\in I_{2n}^n$,
$$\frac12|u_n(T_{-\theta}(x))-u_n(x)|^p +\frac12|u_n(T_\theta(x))-u_n(x)|^p =\frac12.$$
Finally, if $x\notin \cup_{k=-1}^{2n}I_k^n$ then $T_{-\theta}(x), T_{\theta}(x)\not\in\cup_{k=0}^{2n-1}I_k^n$ thus
$$\frac12|u_n(T_{-\theta}(x))-u_n(x)|^p +\frac12|u_n(T_\theta(x))-u_n(x)|^p =0.$$
Consequently,
$$\int_X\int_X |u_n(y)-u_n(x)|^p dm_x(y)d\nu(x)=\frac12(4\cdot2\delta(n))+2^{p-1}(2\cdot2\delta(n))=(4+2^{p+1})\delta(n).$$
Therefore, there is no $\Lambda>0$ such that
$$\left\Vert u_n-\frac{1}{2\pi}\int_X u_nd\nu\right\Vert_{L^p(X,\nu)} \le \Lambda\left(\int_X\int_X |u_n(y)-u_n(x)|^p dm_x(y)d\nu(x)\right)^{\frac1p}\quad \hbox{for every } n \in\N,$$
since this would imply
$$4n\delta(n)\le \Lambda(4+2^{p+1})\delta(n) \Longrightarrow n\le \Lambda+2^{p-1}\quad \hbox{for every } n\in\N .$$
\end{example}

  The proofs of the following lemmas are similar to the proof of~\cite[Lemma 4.2]{AIMTifb}.

\begin{lemma}\label{LemaAcotLp}
  Let $p\ge 1$.  Let $[X,d,m]$ be a metric random walk space with reversible measure $\nu$ with respect to $m$. Let $A, B\subset X$ be disjoint measurable sets and assume that $A\cup B$ is non-$\nu$-null and $m$-connected. Suppose that $[X,d,m]$ satisfies a generalised $(p,p)$-Poincar\'e type inequality on $(A\cup B,\emptyset)$. Let $\alpha$ and $\tau$ be maximal monotone graphs in $\R^2$ such that $0\in \alpha(0)$ and $0\in\tau(0)$. Let   $\{u_n\}_{n\in\N}\subset L^p(A\cup B,\nu)$, $\{z_n\}_{n\in\N}\subset L^1(A,\nu)$ and $\{\omega_n\}_{n\in\N}\subset L^1(B,\nu)$ be such that, for every $n\in\N$, $z_n\in \alpha(u_n)$ $\nu$-a.e. in $A$ and $\omega_n\in \tau(u_n)$ $\nu$-a.e. in $B$.
\item[(i)] Suppose that $\mathcal{R}_{\alpha,\tau}^+=+\infty$ and that there exists $M>0$ such that
$$\int_{A}z_n^+d\nu+\int_{B}\omega_n^+d\nu<M   \quad \hbox{for every } n\in\N.$$
Then, there exists a constant $K=K(A,B,M,\alpha, \tau)$     such that
$$  \left\Vert  u^+_n \right\Vert_{L^p(A\cup B,\nu)}  \leq K\left(\left( \int_{(A\cup B)\times(A\cup B)} |u^+_n(y)-u^+_n(x)|^p dm_x(y) d\nu(x) \right)^{\frac1p}+1\right)  \quad \hbox{for every } n\in\N .$$

\item[(ii)] Suppose that $\mathcal{R}_{\gamma,\beta}^-=-\infty$ and that there exists $M>0$ such that
$$\int_{A}z_n^-d\nu+\int_{B}\omega_n^-d\nu<M   \quad \hbox{for every } n\in\N.$$
Then, there exists a constant $\widetilde{K}=\widetilde{K}(A,B,M, \alpha,\tau)$, such that
$$  \left\Vert  u^-_n \right\Vert_{L^p(A\cup B,\nu)}  \leq \widetilde{K}\left(\left(  \int_{(A\cup B)\times(A\cup B)} |u^-_n(y)-u^-_n(x)|^p dm_x(y) d\nu(x) \right)^{\frac1p}+1\right)  \quad \hbox{for every } n\in\N.$$

\end{lemma}

\begin{lemma}\label{LemaAcotLp2}
  Let $p\ge1$.  Let $[X,d,m]$ be a metric random walk space with reversible measure $\nu$ with respect to $m$. Let $A,B\subset X$ be disjoint measurable sets and assume that $A\cup B$ is non-$\nu$-null  and $m$-connected.  Suppose that $[X,d,m]$ satisfies a generalised $(p,p)$-Poincar\'e type inequality on $(A\cup B,\emptyset)$. Let $\alpha$ and $\tau$ be maximal monotone graphs in $\R^2$ such that $0\in \alpha(0)$ and $0\in\tau(0)$. Let $\{u_n\}_{n\in\N}\subset L^p(A\cup B,\nu)$, $\{z_n\}_{n\in\N}\subset L^1(A,\nu)$ and $\{\omega_n\}_{n\in\N}\subset L^1(B,\nu)$ such that, for every $n\in\N$, $z_n\in\alpha(u_n)$ $\nu$-a.e. in $A$ and $\omega_n\in\tau(u_n)$ $\nu$-a.e. in $B$.
\item[(i)] Suppose that $\mathcal{R}_{\alpha,\tau}^+<+\infty$ and that there exists $M\in\R$ and $h>0$ such that
$$\int_{A}z_nd\nu+\int_{B}\omega_nd\nu<M<\mathcal{R}_{\alpha,\tau}^+ \quad \hbox{for every } n\in\N,$$
and
$$\max\left\{\int_{\{x\in A \, :\, z_n<-h\}}|z_n|d\nu,\int_{x\in B \, :\, \omega_n(x)<-h\}}|\omega_n|d\nu\right\}<\frac{\mathcal{R}_{\alpha,\tau}^+-M}{8} \quad \hbox{for every } n\in\N .$$
Then, there exists a constant $K=K(A, B,M,h,\alpha, \tau)$   such that
$$  \left\Vert  u^+_n \right\Vert_{L^p(A\cup B,\nu)}  \leq K\left(\left(\int_{(A\cup B)\times(A\cup B)} |u^+_n(y)-u^+_n(x)|^p dm_x(y) d\nu(x) \right)^{\frac1p}+1\right) \quad\hbox{for every } n\in\N$$

\item[(ii)] Suppose that $\mathcal{R}_{\alpha,\tau}^->-\infty$ and that there exists $M\in\R$ and $h>0$ such that
$$\int_{A}z_nd\nu+\int_{B}\omega_nd\nu>M>\mathcal{R}_{\alpha,\tau}^- \quad \hbox{for every } n\in\N,$$
and
$$\max\left\{\int_{\{x\in A \, :\, z_n>h\}} z_n d\nu,\int_{x\in B \, :\, \omega_n(x)>h\}}\omega_n d\nu\right\}<\frac{M-\mathcal{R}_{\alpha,\tau}^-}{8} \quad \hbox{for every } n\in\N .$$
Then, there exists a constant $\widetilde{K}=\widetilde{K}(A,B,M,h, \alpha, \tau)$  such that
$$  \left\Vert  u^-_n \right\Vert_{L^p(A\cup B,\nu)}  \leq \widetilde{K}\left(\left(\int_{(A\cup B)\times(A\cup B)} |u^-_n(y)-u^-_n(x)|^p dm_x(y) d\nu(x) \right)^{\frac1p}+1\right) \quad \hbox{for every } n\in\N.$$
\end{lemma}

\bigskip

\noindent {\bf Acknowledgment.}
The authors appreciate the suggestions and comments made by the anonymous referee, which have allowed for a better presentation of this work.
The authors are grateful to J. M. Maz\'{o}n for stimulating discussions on this paper. The authors have been partially supported  by the Spanish MICIU and FEDER, project PGC2018-094775-B-100, and by the ``Conselleria de Innovaci\'{o}n, Universidades, Ciencia y Sociedad Digital'', project AICO/2021/223. The first author was also supported by the Spanish MICIU under grant BES-2016-079019 (which is supported by the European FSE); by the Spanish Ministerio de Universidades and NextGenerationUE, programme ``recualificaci\'on del sistema universitario espa\~{n}ol'' (Margarita Salas) under grant UP2021-044; and the Conselleria d'Innovaci\'o, Universitats, Ci\`encia i Societat Digital, programme ``Subvenciones para la contrataci\'{o}n de personal investigador en fase postdoctoral'' (APOSTD 2022), under grant CIAPOS/2021/28.

\end{document}